\documentclass{amsart}
\usepackage{amsmath,amscd,amssymb,amsthm,amsbsy,amscd,stmaryrd}
\usepackage[dvips]{graphicx,color}
\usepackage{epsfig}
\usepackage{mathrsfs}
\usepackage{mathpazo}
\input{xy}

\usepackage{pstricks}
\usepackage{color}
\usepackage{graphicx}
\usepackage[all]{xy}
\usepackage {hyperref}
\usepackage{framed, color}
\definecolor{shadecolor}{rgb}{1,0.8,0.3}

\newtheorem{theorem}{Theorem}[section]
\newtheorem{lemma}[theorem]{Lemma}
\newtheorem{proposition}[theorem]{Proposition}
\newtheorem{corollary}[theorem]{Corollary}
\newtheorem{conjecture}[theorem]{Conjecture}
\newtheorem{question}[theorem]{Question}

\newtheorem{definition}[theorem]{Definition}
\newtheorem{example}[theorem]{Example}

\newtheorem{remark}[theorem]{Remark}

\numberwithin{equation}{section}

\begin{document}
\title{Local curvature estimates for the Ricci-harmonic flow}
\author{Yi Li}
\address{Faculty of Science, Technology and Communication (FSTC), Mathematic Research Unit\\
Campus Belval, Universite du Luxembourg\\
Maison du Nombre, 6, avenue de la Fonte\\
L-4364, Esch-sur-Alzette, Grand-Duchy of Luxembourg}
\email{yilicms@gmail.com}



\maketitle

\begin{abstract} In this paper we give an explicit bound of $\Delta_{g(t)}u(t)$ and the local curvature estimates for the Ricci-harmonic flow
\begin{equation*}
\partial_{t}g(t)=-2\!\ {\rm Ric}_{g(t)}+4\nabla_{g(t)}u(t)
\otimes \nabla_{g(t)}u(t), \ \ \ \partial_{t}u(t)
=\Delta_{g(t)}u(t)
\end{equation*}
under the condition that the Ricci curvature is bounded along the flow. In the second part these local curvature estimates are extended to a class of generalized Ricci flow, introduced by the author \cite{LY1}, whose stable points give Ricci-flat metrics on a
complete manifold, and which is very close to the $(K, N)$-super
Ricci flow recently defined by Xiangdong Li and Songzi Li \cite{LL2014}. Next we propose a conjecture for Einstein's scalar field equations motivated by a result in the first part and the bounded $L^{2}$-curvature conjecture recently solved by Klainerman, Rodnianski and Szeftel \cite{KRS2015}. In the last two parts of this paper, we discuss two notions of ``Riemann curvature tensor'' in the sense of Wylie-Yeroshkin \cite{KW2017, KWY2017, Wylie2015, WY2016}, respectively, and Li \cite{LY3}, whose ``Ricci curvature'' both give the standard Bakey-\'Emery Ricci curvature \cite{BE1985}, and the forward and backward uniqueness for the Ricci-harmonic flow.
\end{abstract}

\renewcommand{\labelenumi}{Case \theenumi.} %
\newtheoremstyle{mystyle}{3pt}{3pt}{\itshape}{}{\bfseries}{}{5mm}{} %
\theoremstyle{mystyle} \newtheorem{Thm}{Theorem} \theoremstyle{mystyle} %
\newtheorem{lem}{Lemma} \newtheoremstyle{citing}{3pt}{3pt}{\itshape}{}{%
\bfseries}{}{5mm}{\thmnote{#3}} \theoremstyle{citing} %
\newtheorem*{citedthm}{}

\tableofcontents

\section{Introduction}\label{section1}

In this paper we continue the study (see \cite{LY1, LY2}) of Ricci-harmonic flow
\begin{equation}
\partial_{t}g(t)=-2\!\ {\rm Ric}_{g(t)}+4\nabla_{g(t)}
u(t)\otimes\nabla_{g(t)}u(t), \ \ \
\partial_{t}u(t)=\Delta_{g(t)}u(t)\label{1.1}
\end{equation}
on a complete $n$-dimensional manifold $M$, where $g(t)$ is a family of smooth
Riemannian metrics on $M$ and $u(t)$ is a family of smooth functions on $M$. The short time existence established by
List \cite{List2005, List2008} and later extended by M\"uller \cite{Muller2009, Muller2012}, says that, given an initial data $(g_{0}, u_{0})$ with $g_{0}$
and $u_{0}$ being, respectively, a smooth Riemannian metric and a smooth function on $M$, the system {\color{blue}{(\ref{1.1})}} exists over a maximal time interval $[0,T_{\max})$, where $T_{\max}$ is a finite number or infinity.
\\

In the following we consider the Ricci-harmonic flow on $M\times[0,T]$ or on $M\times[0,T)$, depending on different situations, with $T\in(0,T_{\max})$. We brief our main results below.
\\

{\bf Convention:} From now on we always omit the time variable $t$ and
write $\Box, \Delta, \nabla$, $u, {\rm Rm}, {\rm Ric}, R, dV_{t}, |\cdot|$ for $\Box_{g(t)}:=
\partial_{t}-\Delta_{g(t)},
\Delta_{g(t)}, \nabla_{g(t)}, u(t), {\rm Rm}_{g(t)}$, ${\rm Ric}_{g(t)}$,
$R_{g(t)}$, $dV_{g(t)}$, $|\cdot|_{g(t)}$ in the concrete computations, respectively. We write $\mathcal{P}\lesssim
\mathcal{Q}$ or $\mathcal{Q}\gtrsim\mathcal{P}$ for two quantities $\mathcal{P}$ and $\mathcal{Q}$, if $\mathcal{P}\leq
C\!\ \mathcal{Q}$ for some
uniform constant\footnote{Given a flow on $M\times[0, T]$ or $M\times[0,T)$, a uniform
constant $C$ in this paper means a positive constant depending only on the initial
data of the flow, $M$, and $T$. Of course, when $T$ varies, $C$ varies too.} $C$ depending only on $g_{0}, u_{0}$, and the dimension $n$. The uniform constants $C$ may vary from
lines to lines.
\\

{\bf (O) Some known results.} It is known that the Ricci-harmonic flow shares the many properties with the Ricci flow. Here we list some results
both for the Ricci-harmonic flow and the Ricci flow, see Table 1. Besides these results, there are other works on the Ricci-harmonic flow including gradient estimates, eigenvalues, entropies, functionals, and solitons, etc., see \cite{Cao-Guo-Tran2015, CZ2013, Guo-Huang-Phong2015, Guo-Philipowski-Thalmaier2013, Guo-Philipowski-Thalmaier2014, Guo-Philipowski-Thalmaier2015, LY2, LY3, Muller2010}.

\begin{table}[ht]
\caption{Ricci-harmonic flow (RHF) via Ricci flow (RH)}
\begin{minipage}{\textwidth}
\begin{center}
\begin{tabular}{c||c|c}
\hline
   Known properties for $T_{\max}<\infty$    & RF & RHF\\
\hline
 & \\
$\lim_{t\to T_{\max}}\max_{M}|{\rm Rm}_{g(t)}|^{2}_{g(t)}=\infty
$ & \cite{H1982} & \cite{List2005}\rlap{${}^1$} \\
& \\
\hline
& \\
$\lim_{t\to T_{\max}}
\max_{M}|{\rm Ric}_{g(t)}|^{2}_{g(t)}=\infty$
& \cite{Sesum2005}\rlap{${}^2$} & \cite{CZ2013} \\
& \\
\hline
& \\
$\begin{array}{cc}
{\it either} \ \lim_{t\to T_{\max}}\max_{R_{g(t)}}
=\infty \ {\it or}\\
\max_{M}R_{g(t)}\lesssim 1 \ \text{and}\\
\lim_{t\to T_{\max}}\frac{|W_{g(t)}|_{g(t)}+|\nabla^{2}_{g(t)}u(t)|_{g(t)}}{R_{g(t)}}
=\infty
\end{array}
$            & \cite{Cao2011} & \cite{LY2}\rlap{${}^3$} \\
&\\
\hline
& \\
$\begin{array}{cc}
T<\infty, \ n=4, \ |R_{g(t)}|_{g(t)}\lesssim1\Rightarrow\\
\int_{M}|{\rm Ric}_{g(t)}|^{2}_{g(t)}dV_{g(t)}
\lesssim1 \ \text{and}\\
\int_{M}|{\rm Rm}_{g(t)}|^{2}_{g(t)}
dV_{g(t)}\lesssim1
\end{array}
$ & \cite{Bamler-Zhang2015, Simon2015} &
\cite{LY2}\\
&\\
\hline
& \\
$\begin{array}{cc}
{\bf Conjecture:}\\
\lim_{t\to T_{\max}}\max_{M}R_{g(t)}
=\infty\end{array}$ & see \cite{Cao2011}\rlap{${}^4$} &
\cite{LY3}\\
&\\
\hline
& \\
Pseudo-locality theorem & \cite{P1} &
\cite{Guo-Huang-Phong2015}\\
&\\
\hline
\end{tabular}
\end{center}
\vspace*{1pc}

{\Small
${}^1$ For the general Ricci-harmonic flow, this result was proved by
Muller \cite{Muller2009}.

${}^2$ Recently, a new and elementary proof was given by Kotschwer, Munteaun, and Wang \cite{KMW2016}.

${}^3$ Here $W_{g(t)}$ is the Weyl tensor of $g(t)$, and for RF we let $u(t)\equiv0$. According to the evolution equation for $R_{g(t)}$
(see {\color{blue}{(\ref{A.8})}} -- {\color{blue}{(\ref{A.9})}}) and the maximum principle, we can assume that $R_{g(t)}>0$
for all $t$.

${}^4$ This conjecture is due to Hamilton and was verified for the K\"ahler-Ricci flow by Zhang \cite{Zhang2010} and the Type-I Ricci flow by Enders, M\"uller and Topping \cite{Enders-Muller-Topping2011}.
\par
}
\end{minipage}
\vspace{.25\baselineskip}
\hrule width \textwidth
\vspace{.5\baselineskip}
\end{table}

In particular, we mention a result on the gradient estimates of $u(t)$.
Under the curvature condition
\begin{equation}
\sup_{M\times[0,T]}|{\rm Ric}_{g(t)}|_{g(t)}\leq K,\label{1.2}
\end{equation}
{\color{red}{Theorem \ref{tB.2}}} shows that
\begin{equation}
|\nabla_{g(t)}u(t)|^{2}_{g(t)}\lesssim K\label{1.3}
\end{equation}
on $M\times[0,T]$, where $\lesssim$ depends only on $n$. Moreover, under a stronger condition
\begin{equation}
\sup_{M\times[0,T]}|{\rm Rm}_{g(t)}|_{g(t)}
\leq K,\label{1.4}
\end{equation}
{\color{red}{Theorem \ref{tB.2}}} also gives
\begin{equation}
|\nabla^{2}_{g(t)}u(t)|^{2}_{g(t)}\lesssim K
\end{equation}
on $M\times[0,T]$, where $\lesssim$ depends only on $n$.

However, Cheng and Zhu \cite{CZ2013} proved that under the condition
{\color{blue}{(\ref{1.2})}} the Riemannian curvature remains bounded\footnote{Their
bound is
implicit, however, following the argument of \cite{KMW2016}, we can give an
explicit bound.} and hence the
Hessian $\nabla^{2}_{g(t)}u(t)$ is bounded along the flow.
\\

{\bf (A) Gradient estimates under the condition (\ref{1.2}).} In this section, we assume that $M$ is {\it closed}. It is clear from {\color{blue}{(\ref{A.8})}} that the gradient of $u(t)$ along the flow {\color{blue}{(\ref{1.1})}} is bounded\footnote{When $M$ is complete and noncompact,
the same result is still true for List's solution of the Ricci-harmonic flow, see {\color{red}{Theorem \ref{tA.3}}} and {\color{red}{Theorem \ref{tB.1}}}.}, without any curvature condition,
in terms of an initial data. According
to {\color{red}{Lemma \ref{lA.1}}}, integrating {\color{blue}{(\ref{A.8})}} over $M\times[0, T]$, one can prove that the total space-time $L^{2}$-norm of $\nabla^{2}_{g(t)}u(t)$ is bounded
\footnote{A general estimate is obtained in (\ref{3.28}), where a generalized Ricci flow is considered.}, i.e.,
\begin{equation*}
\int^{T}_{0}\int_{M}\left|\nabla^{2}_{g(t)}u(t)\right|^{2}_{g(t)}dV_{t}dt\lesssim e^{T}
\end{equation*}
for any $T<T_{\max}$, without any curvature conditions such like {\color{blue}{(\ref{1.2})}}
or {\color{blue}{(\ref{1.4})}}. For fixed $T\in(0,T_{\max})$, one can get an upper bound, without any curvature conditions, for $\nabla^{2}_{g(t)}u(t)$ and then $\Delta_{g(t)}u(t)$ on $M\times[0,T]$, however this bound may depend on time $T$ and even tends towards to infinity when $T\to T_{\max}$.
\\

The first result in this paper is to obtain the time-independent
(i.e., depends on $T_{\max}$) bound of $\Delta_{g(t)}u(t)$ under the
curvature condition {\color{blue}{(\ref{1.2})}}. We have mentioned that this time-independent bound has been implicitly obtained in \cite{CZ2013}, and our contribution is to obtain an {\it explicit} bound of $\Delta_{g(t)}u(t)$. Under the condition
{\color{blue}{(\ref{1.2})}}, one can prove (see {\color{red}{Proposition \ref{p2.2}}})
\begin{equation}
\int_{M}|\Delta_{g(t)}u(t)|^{2}dV_{g(t)}
\leq C(1+K)e^{C(1+K)T}\leq C(1+K) e^{C(1+K)T_{\max}}, \ t\in[0,T],\label{1.6}
\end{equation}
for some uniform constant $C$. This $L^{2}$-norm bound together with the non-collapsing result (see {\color{red}{Corollary \ref{c2.5}}}) implies that
\begin{eqnarray}
\left|\Delta_{g(t)}u(t)\right|
&\leq&\frac{C(1+K)}{(1+T)^{n/2}}\exp\left[C\left(1+T+KT
+e^{C\sqrt{K}}\right)\right]\nonumber\\
&\leq&C(1+K)\exp\left[C\left(1+T_{\max}
+K T_{\max}+e^{C\sqrt{K}}\right)\right]\label{1.7}
\end{eqnarray}
over $M\times[0,T]$, where $C$ is a uniform constant.

\begin{theorem}\label{t1.1} {\bf (See also Theorem \ref{t2.6})} If the
Ricci-harmonic flow {\color{blue}{(\ref{1.1})}} satisfies the curvature condition
{\color{blue}{(\ref{1.2})}}, then $\Delta_{g(t)}u(t)$ is bounded and an explicit bound is given by {\color{blue}{(\ref{1.7})}}.
\end{theorem}

{\bf (B) Results for a generalized Ricci flow.} The author \cite{LY2} introduced a class of generalized Ricci flow (For motivation see {\bf Section} \ref{section3}),
called {\it $(\alpha_{1}, 0,\beta_{1}, \beta_{2})$-Ricci flow}:
\begin{eqnarray}
\partial_{t}g(t)&=&-2\!\ {\rm Ric}_{g(t)}+2\alpha_{1}\nabla_{g(t)}u(t)
\otimes \nabla_{g(t)}u(t),\label{1.8}\\
\partial_{t}u(t)&=&\Delta_{g(t)}u(t)+\beta_{1}|\nabla_{g(t)}u(t)|^{2}_{g(t)}
+\beta_{2}\!\ u(t).\label{1.9}
\end{eqnarray}
Here $\alpha_{1},\beta_{1},\beta_{2}$ are given constants. Under a technical condition ``{\it regular}'' (for definition, see subsection \ref{subsection3.1}) the system {\color{blue}{(\ref{1.8})}} -- {\color{blue}{(\ref{1.9})}} has the following estimate:
\begin{equation}
|\nabla_{g(t)}u(t)|_{g(t)}\leq C\label{1.10}
\end{equation}
for some uniform constant $C$ depending only on $\alpha_{1},\beta_{1},\beta_{2}, n$. Thus, roughly speaking, the regular condition on constants $\alpha_{1},\beta_{1},\beta_{2}$ guarantees the boundedness of $\nabla_{g(t)}u(t)$.
\\

The flow {\color{blue}{(\ref{1.8})}} -- {\color{blue}{(\ref{1.9})}} is connected with the $(K, N)$-super Ricci flow introduced in \cite{LL2014}, along which the $W$-entropy is constant. For
more detail, see {\bf Section} \ref{section3}.

\begin{theorem}\label{t1.2}{\bf (See also Theorem \ref{t3.2})} Let $(g(t),u(t))_{t\in[0,T)}$ be a solution to the regular $(\alpha_{1},0,\beta_{1},\beta_{2})$-Ricci flow {\color{blue}{(\ref{1.8})}} -- {\color{blue}{(\ref{1.9})}} on a closed $n$-dimensional manifold $M$ with $T\leq\infty$ and the initial data $(g_{0},u_{0})$. Assume that $S_{g(t)}+C\geq C_{0}>0$ along the flow for some uniform constants $C, C_{0}>0$. Then
\begin{equation}
\frac{|{\rm Sin}_{g(t)}|_{g(t)}}{S_{g(t)}+C}
\leq C_{1}+C_{2}\max_{M\times[0,t]}
\sqrt{\frac{|W_{g(s)}|_{g(s)}+|\nabla^{2}_{g(s)}u(s)|^{2}_{g(s)}}{S_{g(s)}
+C}}\label{1.11}
\end{equation}
where
\begin{eqnarray*}
{\rm Sic}_{g(t)}&=&{\rm Ric}_{g(t)}
-\alpha_{1}\nabla_{g(t)}u(t)\otimes\nabla_{g(t)}u(t),\\
S_{g(t)}&=&{\rm tr}_{g(t)}{\rm Sic}_{g(t)} \ \ = \ \
R_{g(t)}-\alpha_{1}|\nabla_{g(t)}u(t)|^{2}_{g(t)},\\
{\rm Sin}_{g(t)}&:=&{\rm Sic}_{g(t)}
-\frac{S_{g(t)}}{n}g(t)
\end{eqnarray*}
and $W_{g(t)}$ is the Weyl tensor field of $g(t)$.
\end{theorem}

Here the technical assumption $S_{g(t)}+C\geq C_{0}>0$ is necessary in the theorem, since, due to the undermined signs of $\alpha_{1}, \beta_{1}, \beta_{2}$, we can not in general deduce any bounds for $S_{g(t)}$ from the evolution equations of {\color{blue}{(\ref{1.8})}} -- {\color{blue}{(\ref{1.9})}} (see, for example, the evolution equation {\color{blue}{(\ref{3.11})}}). In the simplest case, $\alpha_{1}\geq0$ and $\beta_{1}
=\beta_{2}=0$ (i.e., Ricci-harmonic flow), we have a lower bound from {\color{red}{Lemma \ref{lA.1}}}.

This result is obtained by applying Hamilton-Cao's method (see for example \cite{Cao2011, LY2}) on {\color{blue}{(\ref{1.8})}} -- {\color{blue}{(\ref{1.9})}}. More precisely, we consider the quantity
\begin{equation}
f_{1}:=\frac{|{\rm Sic}_{g(t)}+\frac{C}{n}g(t)|^{2}_{g(t)}}{(S_{g(t)}
+C)^{2}},\label{1.12}
\end{equation}
and deduce the following evolution inequality
\begin{eqnarray}
\Box f&\leq& 2\langle\nabla f,\nabla \ln(S+C)\rangle\nonumber\\
&&+ \ 4(S+C)f
\left(-f+\frac{2}{n-2}f^{1/2}
+C+C\frac{|W|^{2}+|\nabla^{2}u|^{2}}{S+C}\right).\label{1.13}
\end{eqnarray}
A direct consequence of the maximum principle applied on {\color{blue}{(\ref{1.13})}} yields {\color{blue}{(\ref{1.11})}}. In conclusion, we can derive the long-time existence of
{\color{blue}{(\ref{1.8})}} -- {\color{blue}{(\ref{1.9})}}.

\begin{corollary}\label{c1.3}{\bf (See also Corollary \ref{c3.3})} Let $(g(t),u(t))_{t\in[0,T)}$ be a solution to the regular $(\alpha_{1},0,\beta_{1},\beta_{2})$-Ricci flow {\color{blue}{(\ref{1.8})}} -- {\color{blue}{(\ref{1.9})}} on a closed $n$-dimensional manifold $M$ with $T\leq\infty$ and the initial data $(g_{0},u_{0})$. Then only one of the followings cases occurs:

\begin{itemize}

\item[(a)] $T=\infty$;

\item[(b)] $T<\infty$ and $|{\rm Ric}_{g(t)}|_{g(t)}\leq C$ for some uniform constant $C$;

\item[(c)] $T<\infty$ and $|{\rm Ric}_{g(t)}|_{g(t)}\to\infty$ as $t\to T$. In this case, there are only two subcases:

    \begin{itemize}

    \item[(c1)] $|R_{g(t)}|_{g(t)}\to\infty$,

    \item[(c2)] $|R_{g(t)}|_{g(t)}\leq C$ for some uniform constant $C$ and there exist some uniform constants $C_{1}, C_{2}>0$ such that $S_{g(t)}+C_{1}\geq C_{2}>0$ and
        \begin{equation*}
        \frac{|W_{g(t)}|_{g(t)}+|\nabla^{2}_{g(t)}u(t)|^{2}_{g(t)}}{S_{g(t)}
        +C_{1}}\to\infty
        \end{equation*}
        as $t\to T$.

    \end{itemize}

\end{itemize}
\end{corollary}

This is a general picture of the long time existence for regular $(\alpha_{1},0,\beta_{1},\beta_{2})$-Ricci
flows, containing results in \cite{Cao2011, LY2}. Since the signs of $\alpha_{1},\beta_{1},\beta_{2}$ are not determined, we can not discard the case (b) which is of course true for Ricci-harmonic flow and Ricci flow.
\\

In the case of dimension $n=4$, we can consider another quantity
\begin{equation}
f_{2}:=\frac{|{\rm Sic}_{g(t)}|^{2}_{g(t)}}{S_{g(t)}+C}.\label{1.14}
\end{equation}
Since the sign of $\alpha_{1}$ is undermined, some term in the evolution inequality (see {\color{blue}{(\ref{3.37})}}) for $f_{2}$ will contain the $L^{4}$-norm of $\nabla^{2}_{g(t)}
u(t)$. Though we can prove
\begin{equation}
\left|\left|\nabla^{2}_{g(t)}u(t)\right|\right|_{L^{1}_{[0,T]}
L^{2}(M)}\leq Ce^{CT}\leq C e^{C T_{\max}}\label{1.15}
\end{equation}
for any $T\in(0, T_{\max})$, it is not clear whether we can get its (spatial) $L^{4}$-norm\footnote{Under the curvature condition (\ref{1.4}), we can always get a bound for the $L^{4}$-norm of $\nabla^{2}_{g(t)}u(t)$ by {\color{red}{Theorem \ref{tB.2}}}.}.

\begin{corollary}\label{c1.4}{\bf (See also Corollary \ref{c3.6})} Let $(g(t),u(t))_{t\in[0,T)}$ be a solution to the regular $(\alpha_{1},0,\beta_{1},\beta_{2})$-Ricci flow {\color{blue}{(\ref{1.8})}} -- {\color{blue}{(\ref{1.9})}} on a closed $4$-manifold $M$ with $T\leq\infty$ and the initial data $(g_{0},u_{0})$. Assume that $S_{g(t)}+C\geq C_{0}>0$ and $S^{2}_{g(t)}
\leq C_{1}<\infty$ along the flow for some uniform constants $C, C_{0},C_{1}>0$. Then
\begin{eqnarray}
\int^{s}_{0}\int_{M}|{\rm Sic}_{g(t)}|^{2}_{g(t)}
dV_{g(t)}dt&\leq&C'(1+s) e^{C's}+\nonumber\\
&&C'[1-{\rm sgn}(\alpha_{1},0)]e^{C's}\int^{s}_{0}\int_{M}|\nabla^{2}_{g(t)}u(t)|^{4}
dV_{g(t)}dt,
\label{1.16}
\end{eqnarray}
\begin{eqnarray}
\int^{s}_{0}\int_{M}|{\rm Rm}_{g(t)}|^{2}_{g(t)}dV_{g(t)}dt&\leq&C'(1+s) e^{C's}+\nonumber\\
&&C'[1-{\rm sgn}(\alpha_{1},0)]e^{C's}\int^{s}_{0}\int_{M}|\nabla^{2}_{g(t)}u(t)
|^{4}dV_{g(t)}dt,
\label{1.17}
\end{eqnarray}
for all $s\in[0,T)$, where $C'=C'(g_{0},u_{0},\alpha_{1},
\beta_{1},\beta_{2},C,C_{0},C_{1}, A_{1},\chi(M))$ is a uniform constant. Here $|\nabla_{g(t)}u(t)|^{2}_{g(t)}\leq A_{1}$ holds along the flow (by the regularity) for some uniform constant $A_{1}>0$ (which depends only on $g_{0},u_{0}$ and $\alpha_{0},\beta_{1},\beta_{2}$).
\end{corollary}

In the corollary, the shorthand notion ${\rm sgn}(\alpha_{1},0)$ is defined to be $1$ if $\alpha_{1}\geq0$, and $0$ otherwise. When $\alpha_{1}$ is nonnegative, the inequalities {\color{blue}{(\ref{1.16})}} and {\color{blue}{(\ref{1.17})}} give us uniform $L^{1}_{[0,T)}
L^{2}(M)$-norms to ${\rm Sic}_{g(t)}$ and ${\rm Rm}_{g(t)}$. Moreover, using
an equality (above {\color{blue}{(\ref{3.43})}}) for ${\rm Vol}_{t}$ and $|\nabla_{g(t)}
u(t)|^{2}_{g(t)}\leq A_{1}$, the integral of ${\rm Sic}_{g(t)}$ can be
replaced by the integral of ${\rm Ric}_{g(t)}$. Thus, in the case that $\alpha_{1}\geq0$, {\color{red}{Corollary \ref{c1.4}}} gives uniform $L^{1}_{[0,T)}L^{2}(M)$-norms to ${\rm Ric}_{g(t)}$ and ${\rm Rm}_{g(t)}$.
\\

To derive $L^{1}_{[0,T)}L^{p}(M)$-norms for ${\rm Sic}_{g(t)}$ and $
{\rm Rm}_{g(t)}$, for simplicity, introduce the basic assumption {\bf BA}, original defined in \cite{Bamler-Zhang2015, Simon2015} (see also \cite{LY2}), for a solution $(g(t),u(t))_{t\in[0,T)}$ to the regular $(\alpha_{1},0,\beta_{1},\beta_{2})$-Ricci flow:

\begin{itemize}

\item[(a)] $M$ is a closed $4$-manifold;

\item[(b)] $T<\infty$;

\item[(c)] $-1\leq S_{g(t)}\leq 1$ along the flow;

\item[(d)] $|\nabla_{g(t)}u(t)|^{2}_{g(t)}\leq A_{1}$ along the flow.

\end{itemize}
The last condition is obtained from the regularity of the flow and the third condition implies $S_{g(t)}+C\geq C_{0}>0$, where $C=2$ and $C_{0}=1$.

\begin{theorem}\label{t1.5} {\bf (See also Theorem \ref{t3.7})} Suppose that $(g(t),u(t))_{t\in[0,T)}$ satisfies {\bf BA}. Then
\begin{eqnarray}
\int^{s}_{0}\int_{M}|{\rm Sic}_{g(t)}|^{4}_{g(t)}dV_{g(t)}dt
&\leq&\widetilde{C}(1+s)e^{\widetilde{C}s}\nonumber\\
&&+ \ \widetilde{C}[1-{\rm sgn}(\alpha_{1},0)]e^{\widetilde{C}s}
\int^{s}_{0}\int_{M}|\nabla^{2}_{g(t)}u(t)|^{4}dV_{g(t)}dt,\label{1.18}\\
\int^{T}_{s}\int_{M}|{\rm Sic}_{g(t)}|^{2}_{g(t)}dV_{g(t)}dt
&\leq&\left[(T-s)e^{T}{\rm Vol}_{0}\right]^{\frac{4-p}{4}}e^{\widetilde{C}T}
\bigg[\widetilde{C}(1+T)\nonumber\\
&&+ \ \widetilde{C}[1-{\rm sgn}(\alpha_{1},0)]
\int^{s}_{0}\int_{M}|\nabla^{2}_{g(t)}u(t)|^{4}dV_{g(t)}dt\bigg]^{\frac{p}{4}},\label{1.19}
\end{eqnarray}
for any $s\in[0,T)$. Here $\widetilde{C}$ is a uniform constant which depends only on $g_{0},u_{0}, \alpha_{1},\beta_{1},\beta_{2}$, $A_{1}, \chi(M)$.
\end{theorem}

{\bf (C) A conjecture for the Einstein scalar field equations.} This conjecture is based on the following, where $M$ is a complete manifold,

\begin{theorem}\label{t1.6}{\bf (See also Theorem \ref{t2.7})} Let $(g(t),u(t))_{t\in[0,T]}$ be a solution to the Ricci-haronic
flow {\color{blue}{(\ref{1.1})}}. Suppose there exist constants $\rho, K, L, P>0$ and $x_{0}\in M$ such that $B_{g(0)}(x_{0},\rho/\sqrt{K})$ is compactly contained on $M$
and the following conditions\footnote{The second assumption $|\nabla_{g(t)}u(t)|_{g(t)}
\leq L$ can not be derived by {\color{red}{Theorem \ref{tB.1}}}. Actually, by {\color{red}{Theorem \ref{tB.1}}}, one has the estimate
\begin{equation*}
|\nabla_{g(t)}u(t)|_{g(t)}
\leq\rho^{2}C_{n}\left(\frac{K}{\rho^{2}}+\frac{1}{t}\right)
=C_{n}\left(K+\frac{\rho^{2}}{t}\right)
\end{equation*}
which depends on time $t$.}
\begin{equation}
|{\rm Ric}_{g(t)}|_{g(t)}
\leq K, \ \ \ |\nabla_{g(t)}u(t)|_{g(t)}
\leq L, \ \ \ |\nabla^{2}_{g(t)}u(t)|_{g(t)}
\leq P\label{1.20}
\end{equation}
hold on $B_{g(0)}(x_{0},\rho/\sqrt{K})\times[0,T]$. For any $p\geq3$, there is a
constant $C$, depending only on $n$ and $p$, such that
\begin{equation*}
\int_{B_{g(0)}(x_{0},\rho/2\sqrt{K})}|{\rm Rm}_{g(t)}|^{p}_{g(t)}
dV_{g(t)} \ \ \leq \ \ C\Lambda_{1} e^{C\Lambda_{2} T}\int_{B_{g(0)}(x_{0},\rho/\sqrt{K})}
|{\rm Rm}_{g(0)}|^{p}_{g(0)}dV_{g(0)}
\end{equation*}
\begin{equation}
+ \ CK^{p}\left(1+\rho^{-2p}\right)
e^{C\Lambda_{2} T}{\rm Vol}_{g(t)}
\left(B_{g(0)}\left(x_{0},\frac{\rho}{\sqrt{K}}\right)
\right).\label{1.21}
\end{equation}
Here $\Lambda_{1}:=1+K$ and $\Lambda_{2}:=K+L+L^{2}+P^{2}(1+K^{-1})$.
\end{theorem}

This theorem extends \cite{KMW2016} to the Ricci-harmonic flow.
Along the argument in \cite{KMW2016} we consider the quantity
\begin{equation*}
\frac{d}{dt}\int_{M}|{\rm Rm}|^{p}\phi^{2p}dV_{t}
\end{equation*}
where $\phi$ is a Lipschitz function with support in $B_{g(0)}
(x_{0}, \rho/\sqrt{K})$. For the Ricci-harmonic flow, extra integrals
involving derivatives of $u(t)$ appear in the computation, however,
these integrals can be treated with the help of the last assumptions in {\color{blue}{(\ref{1.20})}}.

By H\"older's inequality we can get a similar upper bound for $p=2$. From {\color{red}{Theorem \ref{t1.6}}}, we can get an upper bound for the $L^{2}$-norm of ${\rm Rm}_{g(t)}$. Motivated by the inequality {\color{blue}{(\ref{1.21})}}, we in this section impose a conjecture for the Einstein scalar field equations, which is analogous to the corresponding conjecture for the Einstein vacuum equations proved by Klainerman, Rodnianski, and Szeftel \cite{KRS2015, S1, S2, S3, S4, S5}.
\\

Consider Einstein's scalar field equation or Einstein Klein-Gordon
system
\begin{equation}
\boldsymbol{R}_{\alpha\beta}-\frac{1}{2}\boldsymbol{R}{\bf g}_{\alpha\beta}=\boldsymbol{T}_{\alpha\beta}, \ \ \ \boldsymbol{T}_{\alpha\beta}
=2\partial_{\alpha}u\partial_{\beta}u
-\frac{1}{2}|{\bf D}u|^{2}{\bf g},\label{1.22}
\end{equation}
where $u$ is a smooth function on a four dimensional Lorentzian space-time $({\bf M},
{\bf g})$, $\boldsymbol{R}_{\alpha\beta}$, $\boldsymbol{R}$, and ${\bf D}$ denote, respectively, the Ricci curvature tensor, scalar curvature, and the Levi-Civita connection of ${\bf g}$. In this case, the Einstein equation {\color{blue}{(\ref{1.15})}} can be written as
\begin{equation*}
\boldsymbol{R}_{\alpha\beta}-2\partial_{\alpha}u
\partial_{\beta}u=0.
\end{equation*}
As discussed in \cite{R2009}, we should impose a matter equation
\begin{equation*}
\boldsymbol{\Delta}u=0
\end{equation*}
for $u$, where $\boldsymbol{\Delta}:={\bf D}^{\alpha}{\bf D}_{\alpha}$. Hence we should consider a system of PDEs
\begin{equation}
\boldsymbol{R}_{\alpha\beta}
-2\partial_{\alpha}u\partial_{\beta}u=0, \ \ \
\boldsymbol{\Delta}u=0.\label{1.23}
\end{equation}
An initial data set $(\Sigma, g,k,u_{0},u_{1})$ for {\color{blue}{(\ref{1.23})}} consists of a three dimensional manifold $\Sigma$, a Riemannian metric $g$, a symmetric $2$-tensor $k$, together with two functions $u_{0}$ and $u_{1}$ on $\Sigma$, all assumed to be smooth, verifying the constraint equations,
\begin{eqnarray}
\nabla^{j}k_{ij}-\nabla_{i}{\rm tr}\!\ k&=&u_{1}\nabla_{i}u_{0},\label{1.24}\\
R-|k|^{2}+({\rm tr}\!\ k)^{2}&=&u^{2}_{1}+|\nabla u_{0}|^{2},\label{1.25}
\end{eqnarray}
where $\nabla$ is the Levi-Civita connection of $g$.
\\

Given an initial data set $(\Sigma,g,k, u_{0},u_{1})$, the Cauchy problem consists in finding a four-dimensional Lorentzian manifold $({\bf M},
{\bf g})$ and a smooth function $u$ on ${\bf M}$ satisfying {\color{blue}{(\ref{1.22})}}, and also an embedding $\iota: \Sigma\to{\bf M}$ such that
\begin{equation}
\iota^{\ast}{\bf g}=g, \ \ \ \iota^{\ast}u=u_{0}, \ \ \
\iota^{\ast}K=k,\ \ \ \iota^{\ast}(\boldsymbol{N}u)=u_{1},\label{1.26}
\end{equation}
where $\boldsymbol{N}$ is the future-directed unit normal to $\iota(\Sigma)$ and $K$ is the second fundamental form of $\iota(\Sigma)$. The local existence and uniqueness result for globally hyperbolic developments can be found
in, for example, \cite{R2009}, Theorem 14.2. For stability and 
instability for Einstein's scalar field equation, we refer to
\cite{D2003, DW2018, LM2016, LM2017, V2018, Wang2016, Wang, WJH2018}.
\\

For Einstein's equations (i.e., $u=0$ in {\color{blue}{(\ref{1.22})}}, and the corresponding initial data set is denoted by $(\Sigma, g,k)$), Klainerman \cite{K1999} proposed the following conjecture:
\begin{quote}
{\it The Einstein vacuum equations admit local Cauchy developments for initial data sets $(\Sigma,g,k)$ with locally finite $L^{2}$-curvature and locally finite $L^{2}$-norm of the first covariant derivatives of $k$.}
\end{quote}

This conjecture was recently solved by Klainerman, Rodnianski and
Szeftel \cite{KRS2015}. Motivated by {\color{red}{Theorem \ref{t1.6}}} and Klainerman's
conjecture, we propose the following

\begin{conjecture}\label{c1.7}{\bf (See also Conjecture \ref{c4.2})} The Einstein scalar field equations admit local Cauchy developments for initial data sets $(\Sigma, g, k, u_{0},u_{1})$ with locally finite $L^{2}$-curvature, locally finite $L^{2}$-norm of the first covariant derivatives of $k$, locally finite $L^{2}$-norm of the covariant derivatives (up to second order) of $u_{0}$, and locally finite $L^{2}$-norm of the covariant derivatives (up to first order) of $u_{1}$.
\end{conjecture}

{\bf (D) Some notions on Riemann curvatures of Bakry-\'Emery Ricci curvature.} We now compare our curvature ${\rm Sm}$ with $\alpha_{1}=2$ (see {\color{blue}{(\ref{3.10})}}) with a notion of curvature introduced recently by Wylie and Yeroshkin \cite{WY2016}. Let $(M,g)$ be a Riemannian manifold with a smooth function $u$. Wylie and Yeroshkin introduced the following weighted connection
\begin{equation}
\nabla^{u}_{X}Y:=\nabla_{X}Y-(Yu)X-(Xu)Y.\label{1.27}
\end{equation}
By Proposition 3.3 in \cite{WY2016}, we have
\begin{equation}
R^{u}_{ijk\ell}=R_{ijk\ell}
+\nabla_{j}\nabla_{k}u\!\ g_{i\ell}
-\nabla_{i}\nabla_{k}u\!\ g_{j\ell}
+\nabla_{j}u\nabla_{k}u\!\ g_{i\ell}
-\nabla_{i}u\nabla_{k}u\!\ g_{j\ell},\label{1.28}
\end{equation}
where $R^{u}_{ijk\ell}:=\langle{\rm Rm}^{\alpha}(\partial_{i},
\partial_{j})\partial_{k},\partial_{\ell}\rangle$ and ${\rm Rm}^{\alpha}$ is the induced Riemann curvature tensor associated to the connection $\nabla^{u}$. The Ricci curvature associated to $\nabla^{u}$ is defined by
\begin{equation}
R^{u}_{jk}:=g^{i\ell}R^{u}_{ijk\ell}
=R_{ik}+(n-1)\nabla_{j}\nabla_{k}u+(n-1)\nabla_{j}u\nabla_{k}u.\label{1.29}
\end{equation}
Here the last formula also follows from Proposition 3.3 in \cite{WY2016}. Recall from {\color{blue}{(\ref{3.10})}} that (with $\alpha_{1}=2$)
\begin{equation}
S_{ijk\ell}=R_{ijk\ell}-\nabla_{i}u\nabla_{k}u\!\ g_{j\ell}-\nabla_{i}u\nabla_{j}u\!\ g_{k\ell}.\label{1.30}
\end{equation}
From now on, we are given a smooth function $u$ on $M$ and write
\begin{eqnarray}
R^{{\bf L}}_{ijk\ell}&:=&S_{ijk\ell}, \ \ \
R^{{\bf WY}}_{ijk\ell} \ \ := \ \ R^{u}_{ijk\ell},\nonumber\\
R^{{\bf L}}_{jk}&:=&g^{i\ell}R^{{\bf L}}_{ijk\ell}, \ \ \
R^{{\bf WY}}_{jk} \ \ := \ \ R^{u}_{jk} \ \ = \ \ g^{i\ell}R^{{\bf WY}}_{ijk\ell},\label{1.31}\\
R^{{\bf L}}&:=&g^{jk}R^{{\bf L}}_{jk}, \ \ \
R^{{\bf WY}} \ \ := \ \ g^{jk}R^{{\bf WY}}_{jk}.\nonumber
\end{eqnarray}
From {\color{blue}{(\ref{1.29})}} and {\color{blue}{(\ref{1.30})}}, we have
\begin{equation*}
{\rm Ric}^{{\bf L}}={\rm Ric}-2du\otimes du, \ \ \
{\rm Ric}^{{\bf WY}}={\rm Ric}+(n-1)du\otimes du
+(n-1)\nabla^{2}u.
\end{equation*}
Consider another Ricci curvature of $R^{{\bf WY}}_{ijk\ell}$:
\begin{equation*}
\widehat{R}^{{\bf WY}}_{jk}:=g^{i\ell}R^{{\bf WY}}_{ji\ell k}
=R_{jk}+\left(\Delta u+|\nabla u|^{2}\right) g_{jk}
-\nabla_{j}\nabla_{k}u-\nabla_{j}u\nabla_{k}u.
\end{equation*}
There are some relations between those two notions on ``Riemann curvature tensors, e.g.,
\begin{equation}
R^{{\bf WY}}_{ijk\ell}-R^{{\bf L}}_{ijk\ell}
=\nabla_{i}u\nabla_{j}u\!\ g_{k\ell}
+\nabla_{k}u\nabla_{j}u\!\ g_{i\ell}
+\nabla_{j}\nabla_{k}u\!\ g_{i\ell}-\nabla_{i}\nabla_{k}u\!\ g_{j\ell},
\label{1.32}
\end{equation}
and
\begin{equation}
\widehat{{\rm Ric}}{}^{{\bf WY}}
-{\rm Ric}^{{\bf WY}}=\left(\Delta u+|\nabla u|^{2}\right)g
-n\left(\nabla^{2}u+du\otimes du\right).\label{1.33}
\end{equation}
We note that ${\rm Ric}^{{\bf L}}$ and ${\rm Ric}^{{\bf WY}}$ are
actually the Ricci curvatures in the sense of Bakey-\'Emery \cite{BE1985}. We here use our notions to keep the paper smoothly.
\\

We now have four different types of Ricci curvatures, ${\rm Ric}$, ${\rm Ric}^{{\bf L}}$, ${\rm Ric}^{{\bf WY}}$,
and $\widehat{{\rm Ric}}^{{\bf WY}}$, and three different
types of scalar curvatures, $R$, $R^{{\bf L}}$, and $R^{{\bf WY}}$. In order to compare those quantities, we introduce a notation $\mathcal{P}\leq_{{\bf I},\mu}\mathcal{Q}$, which is an integral inequality with respect to the measure $\mu$.

Given two scalar quantities $\mathcal{P}, \mathcal{Q}$
on $(M,g)$, and a measure $\mu$, we write $\mathcal{P}\leq_{{\bf I},\mu}
\mathcal{Q}$ if the following inequality
\begin{equation}
\int_{M}\mathcal{P}\!\ d\mu\leq\int_{\mathcal{M}}\mathcal{Q}\!\ d\mu
\label{1.34}
\end{equation}
holds. When $d\mu$ is the volume form $dV$, we simply write {\color{blue}{(\ref{1.34})}} as $\mathcal{P}\leq_{{\bf I}}
\mathcal{Q}$. When $d\mu$ is the measured volume form $e^{f}dV$, we write {\color{blue}{(\ref{1.34})}} as $
\mathcal{P}_{\leq{\bf I},f}\mathcal{Q}$. Similarly, we can define
$\mathcal{P}_{{\bf I},\mu}\mathcal{Q}$.

\begin{proposition}\label{p1.8}{\bf (See also Proposition \ref{p5.4})} For any measure $\mu$ on $M$ and smooth function $u$ on $M$, we have
\begin{equation}
R^{{\bf L}}\leq_{{\bf I},\mu} R, \ \ \ R\leq_{{\bf I}}
R^{{\bf WY}}, \ \ \ R=_{{\bf I},u}R^{{\bf WY}}.
\label{1.35}
\end{equation}
\end{proposition}

This proposition shows that $R^{{\bf L}}\leq_{{\bf I}}
R\leq_{{\bf I}} R^{{\bf WY}}$ and $R^{{\bf L}}\leq_{{\bf I},u} R=_{{\bf I},u}R^{{\bf WY}}$. Thus, in the sense of integrals, $R_{{\bf L}}$ is
weaker and $R^{{\bf WY}}$ is stronger than $R$, respectively.
\\

Next we consider the similar question on Ricci curvatures. Let $(M,g)$ be a closed Riemannian manifold with a smooth function $u$, and $\mu$ be a given measure on $M$. Given two Ricci curvatures ${\rm Ric}^{\clubsuit}, {\rm Ric}^{\diamondsuit}\in\mathfrak{Ric}_{4}
:=\{{\rm Ric}, {\rm Ric}^{{\bf L}}, {\rm Ric}^{{\bf WY}},
\widehat{{\rm Ric}}^{{\bf WY}}\}$, we say
\begin{equation}
{\rm Ric}^{\clubsuit}\leq_{{\bf I},\mu}{\rm Ric}^{\diamondsuit}
\label{1.36}
\end{equation}
if ${\rm Ric}^{\clubsuit}(X,X)\leq_{{\bf I},\mu}{\rm Ric}^{\diamondsuit}
(X,X)$ holds for all vector fields $X
\in\mathfrak{X}(M)$. Similarly we can define ${\rm Ric}^{\clubsuit}\leq_{{\bf I}}{\rm Ric}^{\diamondsuit}$ and ${\rm Ric}^{\clubsuit}\leq_{{\bf I},f}{\rm Ric}^{\diamondsuit}$. We say
\begin{equation}
{\rm Ric}^{\clubsuit}\leq_{{\bf IK},\mu}{\rm Ric}^{\diamondsuit}\label{1.37}
\end{equation}
if ${\rm Ric}^{\clubsuit}(X,X)\leq_{{\bf I},\mu}{\rm Ric}^{\diamondsuit}
(X,X)$ holds for all Killing vector fields $X
\in\mathfrak{X}_{{\bf K}}(M)$, where $\mathfrak{X}_{{\bf K}}(M)$ is the space of all Killing vector
fields on $M$. Similarly we can define ${\rm Ric}^{\clubsuit}\leq_{{\bf IK}}{\rm Ric}^{\diamondsuit}$ and ${\rm Ric}^{\clubsuit}\leq_{{\bf IK},f}{\rm Ric}^{\diamondsuit}$.

Consider the subset $\mathfrak{X}_{{\bf KC}}(M)$ of $\mathfrak{X}_{{\bf K}}
(M)$, which consists of Killing vector fields on $M$ with constant norm. we say
\begin{equation}
{\rm Ric}^{\clubsuit}\leq_{{\bf IKC},\mu}{\rm Ric}^{\diamondsuit}
\label{1.38}
\end{equation}
if ${\rm Ric}^{\clubsuit}(X,X)\leq_{{\bf I},\mu}{\rm Ric}^{\diamondsuit}
(X,X)$ holds for all $X
\in\mathfrak{X}_{{\bf KC}}(M)$. Similarly we can define ${\rm Ric}^{\clubsuit}\leq_{{\bf IKC}}{\rm Ric}^{\diamondsuit}$ and ${\rm Ric}^{\clubsuit}\leq_{{\bf IKC},f}{\rm Ric}^{\diamondsuit}$. We then obtain the following two results.

\begin{theorem}\label{t1.9}{\bf (see also Theorem \ref{t5.6})} Let $(M,g)$ be a closed Riemannian manifold with a smooth function $u$ and $\mu$ be a given measure on $M$. Then we have
\begin{itemize}

\item[(i)] ${\rm Ric}^{{\bf L}}\leq_{{\bf I}, \mu}{\rm Ric}$.

\item[(ii)] ${\rm Ric}\leq_{{\bf IKC}}{\rm Ric}^{{\bf
WY}}$.

\item[(iii)] ${\rm Ric}\leq_{{\bf IKC}}\widehat{{\rm Ric}}{}^{{\bf WY}}$.

\end{itemize}

\end{theorem}

A consequence of Theorem \ref{t1.9} indicates
\begin{equation}
{\rm Ric}^{{\bf L}}\leq_{{\bf IKC}}{\rm Ric}\leq_{{\bf IKC}}{\rm Ric}^{{\bf WY}} \ \ \ \text{and} \ \ \
{\rm Ric}^{{\bf L}}\leq_{{\bf IKC}}{\rm Ric}\leq_{{\bf IKC}}\widehat{{\rm Ric}}{}^{{\bf WY}}.
\label{1.39}
\end{equation}

\begin{theorem}\label{t1.10}{\bf (See also Theorem \ref{t5.8})} Let $(M,g)$ be a closed Riemannian manifold with a smooth function $u$ and $\mu$ be a given measure on $M$. Then we have
\begin{itemize}

\item[(i)] ${\rm Ric}\leq_{{\bf IKC},\widetilde{f}}{\rm Ric}^{{\bf
WY}}$, and

\item[(ii)] ${\rm Ric}\leq_{{\bf IKC},\widetilde{f}}\widehat{{\rm Ric}}{}^{{\bf WY}}$.

\end{itemize}
where $\widetilde{f}:=u-u_{\min}+c_{0}$ and $c_{0}\geq1/e$.
\end{theorem}

For a given odd-dimensional sphere, we can always find a Riemannian metric $g$ and a Killing
vector field $X$ of constant length with respect to $g$.

\begin{proposition}\label{p1.11} {\bf (See also Proposition \ref{p5.9})} On each of $28$ homotopical seven-dimensional spheres $M$, there exist a Riemannian metric $g$ and a nonzero vector field $X$, such that

\begin{itemize}

\item ${\rm Ric}^{{\bf L}}(X,X)\leq_{{\bf I}}{\rm Ric}(X,X)\leq_{\bf I}{\rm Ric}^{{\bf WY}}(X,X)$ and ${\rm Ric}^{{\bf L}}(X,X)\leq_{{\bf I}}{\rm Ric}(X,X)
    \leq_{{\bf I}}\widehat{{\rm Ric}}{}^{{\bf WY}}(X,X)$ hold.

\item for any smooth function $u$ on $M$, ${\rm Ric}^{{\bf L}}(X,X)\leq_{{\bf I},\widetilde{f}}{\rm Ric}(X,X)\leq_{{\bf I},\widetilde{f}}{\rm Ric}^{{\bf WY}}(X,X)$ and ${\rm Ric}^{{\bf L}}(X,X)\leq{\rm Ric}(X,X)
    \leq_{{\bf I},\widetilde{f}}\widehat{{\rm Ric}}{}^{{\bf WY}}(X,X)$ hold, where $\tilde{f}:=u-u_{\min}+c_{0}$ with $c_{0}\geq1/e$.

\end{itemize}

\end{proposition}

We say that a Riemannian metric $g$ on $M$ is of {\it cohomogeneity 1} if some compact Lie group $G$ acts smoothly and isometrically on $M$ and the space of orbits $M/G$ with respect to this action is one-dimensional.

\begin{proposition}\label{p1.12}{\bf (See also Proposition \ref{p5.10})} Let $n\geq2$ and $\epsilon>0$. On the sphere $\mathbb{S}^{2n-1}$, there are a (real-analytic) Riemannian
metric $g_{\epsilon}$, of cohomogeneity $1$, with the property that all section curvatures of $g_{\epsilon}$ differ from $1$ at most by $\epsilon$, and a (real-analytic) nonzero vector field $X_{\epsilon}$, such that

\begin{itemize}

\item ${\rm Ric}^{{\bf L}}_{g_{\epsilon}}(X_{\epsilon},X_{\epsilon})\leq_{{\bf I}}{\rm Ric}_{g_{\epsilon}}(X_{\epsilon},X_{\epsilon})\leq_{\bf I}{\rm Ric}^{{\bf WY}}_{g_{\epsilon}}(X_{\epsilon},X_{\epsilon})$ and ${\rm Ric}^{{\bf L}}_{g_{\epsilon}}(X_{\epsilon},X_{\epsilon})\leq_{{\bf I}}{\rm Ric}_{g_{\epsilon}}(X_{\epsilon},X_{\epsilon})
    \leq_{{\bf I}}\widehat{{\rm Ric}}{}^{{\bf WY}}_{g_{\epsilon}}
    (X_{\epsilon},X_{\epsilon})$ hold.

\item for any smooth function $u$ on $M$, ${\rm Ric}^{{\bf L}}_{g_{\epsilon}}(X_{\epsilon},X_{\epsilon})\leq_{{\bf I},\widetilde{f}}{\rm Ric}_{g_{\epsilon}}(X_{\epsilon},X_{\epsilon})\leq_{{\bf I},\widetilde{f}}{\rm Ric}^{{\bf WY}}_{g_{\epsilon}}$ $(X_{\epsilon},X_{\epsilon})$ and ${\rm Ric}^{{\bf L}}_{g_{\epsilon}}(X_{\epsilon},X_{\epsilon})\leq{\rm Ric}_{g_{\epsilon}}(X_{\epsilon},X_{\epsilon})
    \leq_{{\bf I},\widetilde{f}}\widehat{{\rm Ric}}{}^{{\bf WY}}_{g_{\epsilon}}(X_{\epsilon},X_{\epsilon})$ hold, where $\tilde{f}:=u-u_{\min}+c_{0}$ with $c_{0}\geq1/e$.

\end{itemize}

\end{proposition}

Berestovskii and Nikonorov \cite{BN2006} observed that (see {\color{red}{Remark
\ref{r5.11}}}) that if the ${\bf S}^{1}$-action obtain by Tuschmann \cite{Tuschmann1997} is free, then we can find for $\epsilon>0$ a Killing vector field $X_{\epsilon}$ of unit length with respect to $g_{\epsilon}$.
\\

The nonnegativity of $R^{{\bf L}}_{ijk\ell}$ was used in \cite{Wu2018}
to prove the compactness for gradient shrinking Ricci harmonic solitons. There is no useful relation between ${\rm Rm}^{{\bf L}}$ and ${\rm Rm}^{{\bf WY}}$. More precisely, we can find (see {\color{red}{Example \ref{e5.12}}}) a Riemannian manifold $(M,g)$ so that ${\rm Rm}^{{\bf L}}(X,Y,Y,X)< {\rm Rm}^{{\bf WY}}(X,Y,Y,X)$
for some triple $(X, Y, u)$ of smooth vector fields $X, Y$ and smooth function $u$, and ${\rm Rm}^{{\bf L}}(X,Y,Y,X)> {\rm Rm}^{
{\bf WY}}(X,$ $Y,Y,X)$ for another such triple $(X', Y', u')$.
\\

{\bf (E) Uniqueness for the Ricci-harmonic flow.} The uniqueness
problem for the Ricci flow was proved by Hamilton \cite{H1982} in the
compact setting, Chen-Zhu \cite{CZ2006} for forward uniqueness of complete solutions, and Kotschwar \cite{K2010, K2014} for forward and backward uniqueness of complete solutions.

It is a natural problem to prove the forward and
backward uniqueness for the Ricci-harmonic flow. Actually,
the forward uniqueness of the Ricci-harmonic flow when the underlying
manifold is compact, was proved by List \cite{List2005}. In this paper, we use the strategy of Kotschwar to prove the uniqueness for complete solutions
of the Ricci-harmonic flow. List \cite{List2005} has proved that (see
{\color{red}{Theorem \ref{tA.3}}}) if $(M, g)$ is a complete and
non-compact Riemannian manifold satisfying
\begin{equation*}
|{\rm Rm}_{g}|_{g}+
|u|+|\nabla_{g} u|_{g}+|\nabla^{2}_{g}u|_{g}\lesssim1
\end{equation*}
then a local time existence holds for the Ricci-harmonic flow with
the initial data $(g,u)$, and moreover
\begin{equation*}
|{\rm Rm}_{g(t)}|_{g(t)}
+|u(t)|+|\nabla_{g(t)}u(t)|_{g(t)}+|\nabla^{2}_{g(t)}u(t)|_{g(t)}\lesssim1.
\end{equation*}
Hence one may expect that the uniqueness of the Ricci-harmonic flow holds in the
class $\{(g,u): {\rm Rm}_{g}, \nabla_{g}u, \nabla^{2}_{g}u \
\text{bounded}\}$. Surprisingly we prove that the uniqueness of
the Ricci-harmonic flow holds in a larger class $\{(g,u): {\rm Rm}_{g} \
\text{bounded}\}$. The reason is that if the Riemann curvature is bounded
along the flow {\color{blue}{(\ref{1.1})}} then all derivatives of $u(t)$ is still bounded by {\color{red}{Theorem \ref{tB.2}}}. To state the results, consider the following curvature condition
\begin{equation}
\sup_{M\times[0,T]}\left(|{\rm Rm}_{g(t)}|_{g(t)}+|{\rm Rm}_{\tilde{g}(t)}|_{\tilde{g}(t)}\right)\leq K,\label{1.40}
\end{equation}
where $K$ is some uniform constant.

\begin{theorem}\label{t1.13}{\bf (See also Theorem \ref{t6.1})} Suppose that $(g(t),u(t))$ and $(\tilde{g}(t),\tilde{u}(t))$ are two smooth
complete solutions of {\color{blue}{(\ref{1.1})}} satisfying {\color{blue}{(\ref{1.40})}}. If $(g(0), u(0))=(\tilde{g}(0),
\tilde{u}(0))$, then $(g(t),u(t))\equiv(\tilde{g}(t),\tilde{u}(t))$ for each $t\in[0,T]$.
\end{theorem}

The Basic idea on proving {\color{red}{Theorem \ref{t1.13}}} follows from the approach of Kotschwar \cite{K2014} who considered the quantity
\begin{eqnarray}
\mathcal{E}(t)&:=&\int_{M}\bigg[t^{-1}|g(t)-\tilde{g}(t)|^{2}_{g(t)}
+t^{-\beta}\left|\Gamma_{g(t)}-\Gamma_{\tilde{g}(t)}\right|^{2}_{g(t)}\nonumber\\
&&+ \ \left|{\rm Rm}_{g(t)}-{\rm Rm}_{\tilde{g}(t)}
\right|^{2}_{g(t)}\bigg]e^{-\eta}dV_{g(t)}\label{1.41}
\end{eqnarray}
for the Ricci flow, where $\beta\in(0,1)$ and $\eta$ is a cutoff function (so that the integral is well-defined as $t$ tends to zero). In our setting, the corresponding quantity for the Ricc-harmonic flow takes the form
\begin{eqnarray}
\mathcal{E}(t)&:=&\int_{M}\bigg[t^{-1}|g(t)-\tilde{g}(t)|^{2}_{g(t)}
+t^{-\beta}\left|\Gamma_{g(t)}-\Gamma_{\tilde{g}(t)}\right|^{2}_{g(t)}
+\left|{\rm Rm}_{g(t)}-{\rm Rm}_{\tilde{g}(t)}
\right|^{2}_{g(t)}\nonumber\\
&&+ \ |u(t)-\tilde{u}(t)|^{2}_{g(t)}
+\left|\nabla_{g(t)} u(t)-\nabla_{\tilde{g}(t)}\tilde{u}(t)\right|^{2}_{g(t)}
\bigg]e^{-\eta}dV_{g(t)}.\label{1.42}
\end{eqnarray}
It can be showed
\begin{equation}
\mathcal{E}'(t)\leq N\mathcal{E}(t)\label{1.43}
\end{equation}
on $[0,T_{0}]$, for some $T_{0}\ll1$ and $N>0$. From {\color{blue}{(\ref{1.43})}} together
with the initial data $\mathcal{E}(0)=0$, we get $\mathcal{E}(t)
\equiv0$ on $[0,T_{0}]$ and then on $[0,T]$.

\begin{theorem}\label{t1.14}{\bf (See also Theorem \ref{t6.6})} Suppose that $(g(t), u(t))$ and $(\tilde{g}(t), \tilde{u}(t))$ are two smooth complete solutions of {\color{blue}{(\ref{1.1})}} satisfying {\color{blue}{(\ref{1.40})}}. If $(g(T), u(T))=(\tilde{g}(T),
\tilde{u}(T))$, then $(g(t), u(t))\equiv(\tilde{g}(t), \tilde{u}(t))$
for each $t\in[0,T]$.
\end{theorem}

To prove {\color{red}{Theorem \ref{t1.14}}}, we use an idea of Kotschwar \cite{K2010} and set
\begin{eqnarray*}
T&:=&{\rm Rm}-\widetilde{{\rm Rm}}, \ \  \
U \ \ := \ \ \nabla{\rm Rm}
-\widetilde{\nabla}\widetilde{{\rm Rm}},\\
h&:=&g-\tilde{g}, \ \ \ A \ \ := \ \ \nabla-\widetilde{\nabla}, \ \ \
B \ \ := \ \ \nabla A\\
y&:=&\nabla^{2}u-\widetilde{\nabla}{}^{2}
\tilde{u}, \ \ \ z \ \ := \ \
\nabla^{3}u-\widetilde{\nabla}{}^{3}\tilde{u},\\
v&:=&u-\tilde{u}, \ \ \ w \ \ := \ \
\nabla u-\widetilde{\nabla}\tilde{u}, \ \ \ x \ \ := \ \ \nabla w,\\
{\bf X}&:=&T\oplus U\oplus y\oplus z, \ \ \ {\bf Y} \ \ := \ \ h\oplus A\oplus B\oplus v\oplus w\oplus x.
\end{eqnarray*}
As the same argument of \cite{K2010}, we can prove
\begin{eqnarray}
|\Box_{g(t)}{\bf X}|^{2}_{g(t)}&\lesssim&|{\bf X}|^{2}_{g(t)}
+|{\bf Y}|^{2}_{g(t)},\nonumber\\
|\partial_{t}{\bf Y}|^{2}_{g(t)}&\lesssim&
|{\bf X}|^{2}_{g(t)}+|{\bf Y}|^{2}_{g(t)}
+|\nabla{\bf X}|^{2}_{g(t)}\label{1.44}
\end{eqnarray}
on any $[\delta, T]$, where $\delta\in(0,T)$. A result of Kotschwar (see
{\color{red}{Theorem \ref{t6.5}}} below) implies
${\bf X}={\bf Y}\equiv0$ on $[\delta, T]$, and then on $[0,T]$.
\\

{\bf Acknowledgments.} The main results were announced in the Conference GeoProb 2017 in the University of Luxembourg from July 10 to 14, and 2018 Mini-Workshop about Function Theory on Riemannian Manifolds in the University of Science and Technology of China from October 5 to 6. The author thanks Professor Anton Thalmaier and Zuoqin Wang, respectively, for his invitation. Some part was done when the author visited Tsinghua University invited by Professor Guoyi Xu with whom I discussed the boundedness of potentials in the Ricci-harmonic flow, and Chinese Academy of Sciences invited by Professor Xiang-Dong Li with whom I discussed the super-Ricci flow.

\section{Gradient and local curvature estimates}\label{section2}

In this section we assume that $(g(t),u(t))_{[0,T]}$ is a solution of {\color{blue}{(\ref{1.1})}} on a closed $n$-dimensional manifold $M$ and use the convention
in {\bf Section} \ref{section1}. From the equation
{\color{blue}{(\ref{A.8})}} in {\color{red}{Lemma \ref{lA.1}}}, we see that
$|\nabla u|^{2}$ is uniformly bounded, i.e.,
\begin{equation}
|\nabla u|\leq L\label{2.1}
\end{equation}
on $M\times[0,T]$ for some uniform constant $L$ depending only on
the initial data $(g_{0},u_{0})=(g(0), u(0))$. We also notice from {\color{blue}{(\ref{A.9})}} that
\begin{equation}
R-2|\nabla u|^{2}\gtrsim -1\Longrightarrow R\gtrsim-1.\label{2.2}
\end{equation}
Moreover, integrating over the space-time $M\times[0,T]$, we get
\begin{eqnarray}
\frac{d}{dt}\int_{M}|\nabla u|^{2}dV_{t}&=&\int_{M}
\partial_{t}|\nabla u|^{2}dV_{t}+\int_{M}|\nabla u|^{2}\partial_{t}dV_{t}
\nonumber\\
&=&\int_{M}\left[-2|\nabla^{2}u|^{2}
-4|\nabla u|^{4}-\left(R-2|\nabla u|^{2}\right)\right]dV_{t}
\label{2.3}
\end{eqnarray}
and then
\begin{equation*}
\frac{d}{dt}\int_{M}|\nabla u|^{2}dV_{t}
+2\int_{M}|\nabla^{2}u|^{2}dV_{t}
=-4\int_{M}|\nabla u|^{4}dV_{t}
-\int_{M}\left(R-2|\nabla u|^{2}\right)dV_{t}.
\end{equation*}
Denoting ${\rm Vol}_{t}$ the volume of $(M, g(t))$, we
have
\begin{equation}
\int^{t}_{0}\int_{M}|\nabla^{2}u|^{2}dV_{t}dt
\leq L^{2}{\rm Vol}_{0}+C\int^{t}_{0}{\rm Vol}_{s}\!\ ds\label{2.4}
\end{equation}
from {\color{blue}{(\ref{2.1})}} and {\color{blue}{(\ref{2.2})}}, where $C$ is a uniform
constant. Using {\color{blue}{(\ref{A.10})}} and {\color{blue}{(\ref{2.2})}}, we have
\begin{equation*}
\partial_{t}{\rm Vol}_{t}=\int_{M}\partial_{t}dV_{t}
=\int_{M}\left(-R+2|\nabla u|^{2}\right)dV_{t}\leq C\!\ {\rm Vol}_{t};
\end{equation*}
consequently,
\begin{equation}
{\rm Vol}_{t}\leq e^{Ct}{\rm Vol}_{0}.\label{2.5}
\end{equation}
Plugging {\color{blue}{(\ref{2.5})}} into {\color{blue}{(\ref{2.4})}} we conclude that
\begin{equation}
\int^{t}_{0}\int_{M}|\nabla^{2}u|^{2}dV_{t}dt
\leq\left(L^{2}+e^{Ct}\right){\rm Vol}_{0}\leq C(1+L^{2})e^{Ct}
\lesssim e^{Ct}.\label{2.6}
\end{equation}
According to {\color{blue}{(\ref{A.5})}} we obtain
\begin{equation}
\Box\Delta u=-4|\nabla u|^{2}\Delta u
+2R_{ij}\nabla^{i}\nabla^{j}u-4\nabla_{i}u\nabla_{j}u
\nabla^{i}\nabla^{j}u.\label{2.7}
\end{equation}
In particular, the square of $\Delta u$ satisfies
\begin{eqnarray*}
\partial_{t}|\Delta u|^{2}&=&2\Delta u\!\ \partial_{t}\Delta u \ \ = \ \ \Delta|\Delta u|^{2}-2|\nabla\Delta u|^{2}
-8|\nabla u|^{2}|\Delta u|^{2}+4\left(R_{ij}\nabla^{i}\nabla^{j}u\right)
\Delta u\\
&&- \ 8\left(\nabla_{i}u\nabla_{j}u\nabla^{i}\nabla^{j}u\right)
\Delta u\\
&\leq&\Delta|\Delta u|^{2}-2|\nabla\Delta u|^{2}
-8|\nabla u|^{2}|\Delta u|^{2}
+4\left(|{\rm Ric}|+2|\nabla u|^{2}\right)
|\nabla^{2}u||\Delta u|\\
&\leq&\Delta|\Delta u|^{2}+2\left(|{\rm Ric}|^{2}+2|\nabla u|^{2}
\right)|\Delta u|^{2}+2\left(|{\rm Ric}|+2|\nabla u|^{2}
\right)|\nabla^{2}u|^{2}.
\end{eqnarray*}
Taking integrations on both sides yields
\begin{equation}
\frac{d}{dt}\int_{M}|\Delta u|^{2}dV_{t} \ \ = \ \ \int_{M}\partial_{t}|\Delta u|^{2}\!\ dV_{t}
+\int_{M}|\Delta u|^{2}\left(-R+2|\nabla u|^{2}\right)dV_{t}\label{2.8}
\end{equation}
$$
\leq \ \ \int_{M}\left(2|{\rm Ric}|-R+6|\nabla u|^{2}\right)|\Delta
u|^{2}dV_{t}+\int_{M}\left(2|{\rm Ric}|+4|\nabla u|^{2}\right)|
\nabla^{2}u|^{2}dV_{t}.
$$
When the Ricci curvature is uniformly bounded, together with {\color{blue}{(\ref{2.6})}},
we can prove that the $L^{2}$-norm of $\Delta u$ is finite.

\begin{proposition}\label{p2.1} If the curvature condition {\color{blue}{(\ref{1.2})}} holds, then
\begin{equation}
\int_{M}|\Delta u|^{2}dV_{t}\leq C(1+K) e^{C(1+K)T}\label{2.9}
\end{equation}
for some uniform constant $C>0$.
\end{proposition}

\begin{proof} Compute, using {\color{blue}{(\ref{2.8})}} and {\color{blue}{(\ref{2.2})}},
\begin{equation*}
\frac{d}{dt}\int_{M}|\Delta u|^{2}dV_{t}
\leq\left(2K+C+6L^{2}\right)\int_{M}|\Delta u|^{2}dV_{t}
+\left(2K+4L^{2}\right)\int_{M}|\nabla^{2}u|^{2}dV_{t}.
\end{equation*}
Therefore
\begin{equation*}
\frac{d}{dt}\left[e^{-(2K+C+6K^{2})t}
\int_{M}|\Delta u|^{2}dV_{t}\right]
\leq\left(2K+4L^{2}\right)e^{-(2K+C+6L^{2})t}
\int_{M}|\nabla^{2}u|^{2}dV_{t}.
\end{equation*}
From {\color{blue}{(\ref{2.6})}}, we obtain
\begin{eqnarray*}
\int_{M}|\Delta u|^{2}dV_{t}
&\leq& Ce^{(2K+C+6L^{2})t}+\left(2K+4 L^{2}\right)
\int^{t}_{0}\int_{M}|\nabla^{2}u|^{2}dt\\
&\leq& C e^{(2K+C+6 L^{2})t}+C\left(2K+4L^{2}\right)
\left(1+L^{2}\right) e^{Ct}
\end{eqnarray*}
which implies {\color{blue}{(\ref{2.9})}}.
\end{proof}

\subsection{The boundedness of $\Delta_{g(t)}u(t)$}\label{subsection2.1}

According to \cite{CZ2013}, Chen and Zhu proved an analog of Sesum's theorem for the Ricci-harmonic flow. As a consequence, we see that $\Delta_{g(t)}u(t)$ is uniformly bounded. Our contribution in this paper is to give an {\it explicit} bound for $\Delta_{g(t)}u(t)$.

We first review a non-collapsing theorem for the Ricci-harmonic flow. Suppose that $M$ is a closed manifold. For any Riemannian metric $g$,
any smooth functions $u, f$, and any positive number $\tau$, define (see \cite{List2005})
\begin{equation}
\mathcal{W}(g,u,f,\tau):=\int_{M}\left[\tau\left(S_{g}
+|\nabla_{g}f|^{2}_{g}\right)+f-n\right]\frac{e^{-f}}{(4\pi\tau)^{n/2}}
dV_{g}\label{2.10}
\end{equation}
with $S_{g}:=R_{g}-2|\nabla_{g}u|^{2}_{g}$, and
\begin{equation}
\mu(g,u,\tau):=\inf\left\{\mathcal{W}(g,u,f,\tau): f\in C^{\infty}(M) \ \text{and} \ \int_{M}\frac{e^{-f}}{(4\pi\tau)^{n/2}}dV_{g}=1\right\}.\label{2.11}
\end{equation}
Observe that
\begin{equation}
\mu(\tau\!\ g,u,\tau)=\mu(g,u,1), \ \ \ \tau>0.\label{2.12}
\end{equation}

\begin{proposition}\label{p2.2} If $(g(t),u(t),\tau(t))_{t\in[0,T]}$ solves
\begin{eqnarray}
\partial_{t}g(t)&=&-2{\rm Ric}_{g(t)}+4\nabla_{g(t)}u(t)
\otimes \nabla_{g(t)}u(t),\nonumber\\
\partial_{g}u(t)&=&\Delta_{g(t)}u(t),\label{2.13}\\
\frac{d}{dt}\tau(t)&=&-1,\nonumber
\end{eqnarray}
then $\mu(g(t),u(t),\tau(t))$ is monotone nondecreasing in time $t$.
\end{proposition}

\begin{proof} See \cite{List2005, List2008, Muller2009, Muller2012}.
\end{proof}

In the definition {\color{blue}{(\ref{2.10})}}, introduce the function
\begin{equation}
w:=\left[\frac{e^{-f}}{(4\pi\tau)^{n/2}}\right]^{1/2}\label{2.14}
\end{equation}
so that we can rewrite the functional $\mathcal{W}$ as
\begin{eqnarray}
\mathcal{W}(g,u,f,\tau)&=&\int_{M}\left[\tau
\left(w^{2}S_{g}+4|\nabla_{g}w|^{2}_{g}\right)
-\left(2\ln w+\frac{n}{2}\ln(4\pi\tau)+n\right)w^{2}\right]dV_{g}
\nonumber\\
&=&\tau\int_{M}S_{g}w^{2}dV_{g}-\left[\frac{n}{2}\ln(4\pi\tau)
+n\right]\int_{M}w^{2}dV_{g}\label{2.15}\\
&&- \ 2\left[\int_{M}w^{2}\ln w\!\ dV_{g}-2\tau|\nabla_{g}w|^{2}_{g}dV_{g}
\right].\nonumber
\end{eqnarray}
The last term in {\color{blue}{(\ref{2.15})}} can be handed by the logarithmic Sobolev
inequality (see \cite{CCGGIIKLLN3}, Lemma 17.1): for any $a>0$ there exists a constant $C(a,g)$ such that if $\varphi>0$ satisfies $\int_{M}\varphi^{2}dV_{g}=1$, then
\begin{equation}
\int_{M}\varphi^{2}\ln\varphi\!\ dV_{g}-
a\int_{M}|\nabla_{g}\varphi|^{2}_{g}dV_{g}\leq C(a,g),\label{2.16}
\end{equation}
where
\begin{equation}
C(a,g)=a{\rm Vol}(M,g)^{-2/n}+\frac{n^{2}}{4ae^{2}C_{s}(M,g)}\label{2.17}
\end{equation}
and $C_{s}(M,g)$ denotes the $L^{2}$-Sobolev constant.

\begin{lemma}\label{l2.3} Let $M$ be a closed $n$-dimensional manifold. For any Riemannian metric $g$,
any smooth function $u$, and any $\tau>0$, we have
\begin{eqnarray}
\mu(g,u,\tau)&\geq&\tau\!\ S_{g,\min}-2C(2\tau,g)-\frac{n}{2}\ln(4\pi
\tau)-n,\label{2.18}\\
\mu(g,u,\tau)&\leq&\tau S_{g,{\rm avg}}+\ln{\rm Vol}(M,g)-\frac{n}{2}
\ln(4\pi\tau)-n.\label{2.19}
\end{eqnarray}
Here $S_{g,\min}:=\min_{M}S_{g}$ and $S_{g,{\rm avg}}$ denotes the average of $S_{g}$ over $(M, g)$.
\end{lemma}

\begin{proof} Taking $f=\ln{\rm Vol}(M,g)-\frac{n}{2}\ln(4\pi\tau)$
in {\color{blue}{(\ref{2.11})}} gives the first inequality. The second estimate follows
from {\color{blue}{(\ref{2.15})}} and {\color{blue}{(\ref{2.16})}}.
\end{proof}

\begin{lemma}\label{l2.4} For every $n\geq2$, $\rho\in(0,\infty)$,
and $D>0$, there exists $c=c(n,\rho,D)>0$ such that if $(M,g)$ is a closed
$n$-dimensional Riemannian manifold, $u$ is a smooth function on $M$, and if for some $r\in(0,\rho]$
and $A<\infty$ we have $\mu(g,u,r^{2})>-A$, then for any $p\in M$
with ${\rm Ric}_{g}\geq- D r^{-2}$ on $B_{g}(p,r)$ and $R_{g}
\leq D r^{-2}$ on $B_{g}(p,r)$, we have
\begin{equation}
{\rm Vol}_{g}(B_{g}(p,r))\geq\kappa\!\ r^{n}\label{2.20}
\end{equation}
where $\kappa:=ce^{-A}$.
\end{lemma}

\begin{proof} Since $S_{g}=R_{g}-2|\nabla_{g}u|^{2}_{g}
\leq R_{g}$, the proof is almost the same as the proof of Proposition
5.37 in \cite{CLN2006}.
\end{proof}

Actually, the constant $c$ in {\color{red}{Lemma \ref{l2.4}}} can be explicitly
determined. We can take the constant $c$ in such a way that it depends
only on $n$ and $C$. From the proof of Proposition 5.37 in \cite{CLN2006}, we have
\begin{equation*}
\mu(g,u,r^{2})\leq \ln\frac{{\rm Vol}_{g}(B_{g}(p,r))}{r^{n}}
+C'(n,r)+\frac{1}{e}(4\pi)^{-n/2}e^{C_{1}(n,r)}
\end{equation*}
where
\begin{equation*}
C'(n,r):=36(4\pi)^{-n/2}e^{C_{1}(n,r)}+D, \ \ \
C_{1}(n,r):=\frac{n}{2}\ln(4\pi)+\ln C(n,r)
\end{equation*}
and
\begin{equation*}
C(n,r)=\int^{r}_{0}\frac{\sinh(\sqrt{K'}t)}{\sqrt{K'}}dt\bigg/\int^{r/2}_{0}
\frac{\sinh(\sqrt{K'}t)}{\sqrt{K'}}dt, \ \ \ K':=\frac{D}{(n-1)r^{2}}.
\end{equation*}
The last quantity $C(n,r)$ can be bounded as
\begin{eqnarray*}
C(n,r)&=&\int^{r}_{0}\left(e^{\sqrt{K'}t}-e^{-\sqrt{K'}t}\right)dt\bigg/
\int^{r/2}_{0}\left(e^{\sqrt{K'}t}-e^{-\sqrt{K't}}\right)dt\\
&=&\frac{e^{\sqrt{k'}t}+e^{-\sqrt{K'}t}\big|^{r}_{0}}{e^{\sqrt{K'}t}
+e^{-\sqrt{K'}t}\big|^{r/2}_{0}} \ \ = \ \
\frac{e^{\sqrt{K'}r}+e^{-\sqrt{K'}r}-2}{e^{\sqrt{K'}r/2}
+e^{-\sqrt{K'}r/2}-2}\\
&=&\frac{(e^{\sqrt{K'}r/2}-e^{-\sqrt{K'}r/2})^{2}}{(e^{\sqrt{K'}r/4}
-e^{-\sqrt{K'}r/4})^{2}} \ \ = \ \ \frac{(e^{\sqrt{K'}r}-1)^{2}}{e^{\sqrt{K'}r/2}
(e^{\sqrt{K'}r/2}-1)^{2}}\\
&=&\frac{(e^{\sqrt{K'}r/2}+1)^{2}}{e^{\sqrt{K'}r/2}} \ \ = \ \
e^{\sqrt{K'}r/2}+2+e^{-\sqrt{K'}r/2} \ \ \leq \ \ 3+e^{\sqrt{K'}r/2}.
\end{eqnarray*}
Hence
\begin{equation}
C(n,r)\leq 3+e^{\sqrt{D/4(n-1)}}\label{2.21}
\end{equation}
and the constant $c$ in {\color{red}{Lemma \ref{l2.4}}} can be taken to be
\begin{equation}
c=C(n)\exp
\left[-C(n)\exp\left(C(n)\sqrt{D}\right)\right],\label{2.22}
\end{equation}
where $C(n)$ is a uniform constant depending only on $n$.
\\

From the rescaling
\begin{equation}
\left|{\rm Ric}_{r^{2}g}\right|_{r^{2}g}=
r^{-2}|{\rm Ric}_{g}|_{g}, \ \ \
B_{r^{2}g}(p,r)=B_{g}(p,1),\label{2.23}
\end{equation}
we can conclude from {\color{red}{Lemma \ref{l2.4}}} that

\begin{corollary}\label{c2.5} For every $n\geq2$ and $D>0$, there
exists $c=c(n,D)>0$ such that if $(M,g)$ is a closed $n$-dimensional Riemannian
manifold, $u$ is a smooth function on $M$, and if for some $A<\infty$
we have $\mu(g,u,1)>-A$, then for any $p\in M$ with $|{\rm Ric}_{g}|_{g}
\leq D$ on $B_{g}(p,1)$, we have
\begin{equation}
{\rm Vol}_{g}(B_{g}(p,1))\geq\kappa\label{2.24}
\end{equation}
where $\kappa=ce^{-A}$. Moreover, $c=c(n,D)$ can be taken to be given in {\color{blue}{(\ref{2.22})}} for some constant $C(n)$ depending only on $n$.
\end{corollary}

\begin{proof} Write $\hat{g}:=r^{2}g$ for any given $r>0$. Then the conditions
$\mu(g,u,1)>-A$ and $|{\rm Ric}_{g}|_{g}\leq C$ on $B_{h}(p,1)$ become
\begin{equation*}
\mu(\hat{g},u,r^{2})=\mu(g,u,1)>-A, \ \ \
|{\rm Ric}_{\hat{g}}|_{\hat{g}}
=\frac{1}{r^{2}}|{\rm Ric}_{g}|_{g}\leq\frac{C}{r^{2}} \
\text{on} \ B_{\hat{g}}(p,r).
\end{equation*}
We obtain from {\color{red}{Lemma \ref{l2.4}}} that $\kappa r^{n}\leq{\rm Vol}_{\hat{g}}
(B_{\hat{g}}(p,r))=r^{n}{\rm Vol}_{g}(B_{g}(p,1))$.
\end{proof}

To prove the boundedness of $\Delta_{g(t)}u(t)$, we first
verify that $\mu(g(t),u(t),1)$ is always bounded from below by some
uniform constant. Set, for each real number $\tau$,
\begin{equation*}
C^{\infty}_{\tau}(M):=\left\{f\in C^{\infty}(M): \int_{M}\frac{e^{-f}}{(4\pi
\tau)^{n/2}}dV_{g}=1\right\}.
\end{equation*}
The mapping
\begin{equation*}
C^{\infty}_{\tau_{2}}(M)\ni f\longmapsto\tilde{f}:=f+\frac{n}{2}\ln\frac{\tau_{2}}{\tau_{1}}
\in C^{\infty}_{\tau_{1}}(M)
\end{equation*}
is one-to-one and onto. Choose $\tilde{f}\in C^{\infty}_{\tau_{1}}(M)$
so that $\mu(g,u,\tau_{1})=\mathcal{W}(g,u,\tilde{f},\tau_{1})$ and define $
f:=\tilde{f}-\frac{n}{2}\ln\frac{\tau_{2}}{\tau_{1}}\in
C^{\infty}_{\tau_{2}}(M)$. Hence, for $\tau_{1}, \tau_{2}>0$,
\begin{equation*}
\mu(g,u,\tau_{2}) \ \ \leq \ \ \mathcal{W}(g,u,f,\tau_{2})
\end{equation*}
\begin{equation*}
= \ \ \int_{M}\left[\tau_{2}\left(S_{g}+|\nabla_{g}f|^{2}_{g}\right)+
f-n\right]\frac{e^{-f}}{(4\pi\tau_{2})^{n/2}}dV_{g}
\end{equation*}
\begin{equation*}
= \ \ \int_{M}\left[\tau_{2}\left(S_{g}+|\nabla\tilde{f}|^{2}_{g}\right)
+\tilde{f}-n-\frac{n}{2}\ln\frac{\tau_{2}}{\tau_{1}}\right]\frac{e^{-\tilde{f}}}{(4
\pi\tau_{1})^{n/2}}dV_{g}
\end{equation*}
\begin{equation*}
= \ \ \int_{M}\left[\tau_{1}\left(S_{g}|\nabla_{g}\tilde{f}|^{2}_{g}\right)
+\tilde{f}-n\right]\frac{e^{-\tilde{f}}}{(4\pi\tau_{1})^{n/2}}dV_{g}
\end{equation*}
\begin{equation*}
-\ \ \frac{n}{2}\ln\frac{\tau_{2}}{\tau_{1}}
+(\tau_{2}-\tau_{1})\int_{M}\left(S_{g}+|\nabla_{g}\tilde{f}|^{2}_{g}
\right)\frac{e^{-\tilde{f}}}{(4\pi\tau_{1})^{n/2}}dV_{g}
\end{equation*}
\begin{equation*}
=\ \ \mu(g,u,\tau_{1})-\frac{n}{2}\ln\frac{\tau_{2}}{\tau_{1}}
+(\tau_{2}-\tau_{1})\int_{M}\left(S_{g}+|\nabla_{g}\tilde{f}|^{2}_{g}
\right)\frac{e^{-\tilde{f}}}{(4\pi\tau_{1})^{n/2}}dV_{g}
\end{equation*}
When $\tau_{1}\geq\tau_{2}>0$, one has
\begin{equation}
\mu(g,u,\tau_{2})\leq
\mu(g,u,\tau_{1})-\frac{n}{2}\ln\frac{\tau_{2}}{\tau_{1}}
+(\tau_{2}-\tau_{1})S_{g,\min}.\label{2.25}
\end{equation}
In particular, if $0<\tau(t)\leq1$, then the inequality {\color{blue}{(\ref{2.25})}} implies
\begin{equation}
\mu(g(t),u(t),\tau(t))\leq\mu(g(t),u(t),1)-\frac{n}{2}\ln\tau(t)
+[\tau(t)-1]S_{g(t),\min}.\label{2.26}
\end{equation}
By the monotonicity of the Ricci-harmonic flow, {\color{red}{Proposition \ref{p2.2}}},
we obtain from {\color{blue}{(\ref{2.26})}} that
\begin{equation*}
\mu(g(t),u(t),1)\geq\mu(g(0),u(0),\tau(0))
+\frac{n}{2}\ln[\tau(0)-t]+[1+t-\tau(0)]S_{g(t),\min}
\end{equation*}
when $1+t-\tau(0)\geq0$ and $\tau(0)-t>0$. In particular, together
with {\color{red}{Lemma \ref{l2.4}}} and {\color{blue}{(\ref{2.4})}},
\begin{equation}
\mu(g(t),\mu(t),1)\geq\frac{n}{2}\ln\frac{\tau(0)-t}{4\pi\tau(0)}
+(1+t)S_{g(0),\min}-2C(2\tau(0),g(0))-n\label{2.27}
\end{equation}
whenever $\tau(0)-1\leq t<\tau(0)$.

\begin{theorem}\label{t2.6} There exists a uniform constant $C$ depending only on $n,
g(0)$, and $u(0)$ such that the following statement is true: If $|{\rm Ric}_{g(t)}|_{g(t)}\leq K$ on $M\times[0,T]$, then
\begin{equation}
|\Delta_{g(t)}u(t)|_{g(t)}\leq\frac{C(1+K)}{(1+T)^{n/2}}
\exp\left[C\left(1+T+1+KT+e^{C\sqrt{K}}\right)\right]\label{2.28}
\end{equation}
over any geodesic ball $B_{g(t)}(p,\sqrt{1+T})$. In particular, the estimate
{\color{blue}{(\ref{2.28})}} holds on $M\times[0,T]$.
\end{theorem}

\begin{proof} Let
\begin{equation*}
\tilde{t}:=\frac{t}{T+1}, \ \ \ \widetilde{T}:=\frac{T}{T+1}, \ \ \
\tilde{g}(\tilde{t}):=\frac{1}{T+1}g\left((T+1)\tilde{t}\right), \ \ \
\tilde{u}(\tilde{t}):=u\left((T+1)\tilde{t}\right).
\end{equation*}
Then $(\tilde{g}(\tilde{t}),\tilde{u}(\tilde{t}))_{\tilde{t}\in[0,
\widetilde{T}]}$ is a solution of the Ricci flow with $\widetilde{T}
\in(0,1)$. In this case we choose $\tilde{\tau}(0):=
(1+\widetilde{T})/2$ so that $\tilde{\tau}(0)
-\tilde{t}\geq(1+\widetilde{T})/2-\widetilde{T}=(1-\widetilde{T})/2>0$ and $
1+\tilde{t}-\tilde{\tau}(0)\geq1-(1+\widetilde{T})/2
=(1-\widetilde{T})/2>0$. Therefore the estimate {\color{blue}{(\ref{2.27})}}
applied to the rescaling Ricci-harmonic flow holds for all $
\tilde{t}\in[0,\widetilde{T}]$, i.e.,
\begin{equation*}
\mu(\tilde{g}(\tilde{t}),\tilde{u}(\tilde{t}),1)\geq\frac{n}{2}
\ln\frac{1-\widetilde{T}}{1+\widetilde{T}}
-(1+\widetilde{T})|\tilde{S}_{\tilde{g}(0),\min}|-2C(1+\widetilde{T},\tilde{g}(0))-\ln4\pi-n.
\end{equation*}
Since the $L^{2}$-Sobolev constant $C_{s}(M,g)$ is invariant under scaling
the metric, it follows from {\color{blue}{(\ref{2.22})}} and {\color{blue}{(\ref{A.9})}} that
\begin{equation*}
\mu(\tilde{g}(\tilde{t}),\tilde{u}(\tilde{t}),1) \ \ \geq \ \
\frac{n}{2}\ln\frac{1-\frac{T}{1+T}}{1+\frac{T}{1+T}}
-\left(1+\frac{T}{1+T}\right)\left|(T+1)S_{g(0),\min}\right|
-\ln 4\pi-n
\end{equation*}
\begin{equation*}
- \ \ 2\left[\frac{1+2T}{1+T}\left[(1+T)^{-n/2}{\rm Vol}(M,g(0))\right]^{-2/n}
+\frac{n^{2}}{4 e^{2}\frac{1+2T}{1+T}C_{s}(M,g(0))}\right]
\end{equation*}
\begin{eqnarray*}
&=&-\frac{n}{2}\ln(1+2T)
-(1+2T)\left[|S_{g(0),\min}|+\frac{2}{{\rm Vol}(M,g(0))^{2/n}}
\right]\\
&&- \ \ln4\pi-n
-\frac{n^{2}(1+T)}{2e^{2}(1+2T)C_{s}(M,g(0))}
\end{eqnarray*}
\begin{equation*}
\geq \ \ -(1+2T)\left[\frac{n}{2}+|S_{g(0),\min}|
+\frac{2}{{\rm Vol}(M,g(0))^{n/2}}\right]-\ln4\pi-n-\frac{n^{2}}{2 e^{2}C_{s}(M,g(0))}.
\end{equation*}
Consequently,
\begin{equation}
\mu(\tilde{g}(\tilde{t}),\tilde{u}(\tilde{t}),1)
\geq-C(1+2T)\label{2.29}
\end{equation}
for some uniform constant $C$ depending only on $g(0)$ and $u(0)$. Because
$|\widetilde{{\rm Ric}}_{\tilde{g}(\tilde{t})}|_{\tilde{g
}(\tilde{t})}$ $=|{\rm Ric}_{g(t)}|_{g(t)}/(1+T)\leq K/(1+T)$ on $\tilde{B}_{\tilde{g}(
\tilde{t})}(p,1)=B_{g(t)}(p,\sqrt{1+T})$, we have
\begin{equation*}
{\rm Vol}_{\frac{1}{1+T}g(t)}\left(B_{\frac{1}{1+T}g(t)}(p,1)
\right)\geq\kappa
\end{equation*}
where $\kappa=C(n)\exp[-C(n)\exp(C(n)\sqrt{K/(1+T)})]e^{-C(1+2T)}$. Thus
\begin{eqnarray}
{\rm Vol}_{g(t)}
\left(B_{g(t)}(p,\sqrt{1+T})\right)
&\geq&C(1+T)^{n/2}\exp\left[-C\left(1+2T+e^{C\sqrt{K/(1+T)}}\right)\right]\nonumber\\
&\geq&C(1+T)^{n/2}\exp\left[-C\left(1+2T+e^{C\sqrt{K}}\right)\right].
\label{2.30}
\end{eqnarray}

We now can prove the estimate {\color{blue}{(\ref{2.28})}}. Suppose otherwise that
\begin{equation*}
|\Delta_{g(t)}
u(t)|>\frac{C(1+K)}{(1+T)^{n/2}}\exp\left\{C\left[(1+K)T
+1+2T+e^{C\sqrt{K}}\right]\right\}
\end{equation*}
over some geodesic ball $B_{g(t)}(p,\sqrt{1+T})$ and for some time $t
\in[0,T]$. On the other hand, from {\color{red}{Proposition \ref{p2.2}}} and {\color{blue}{(\ref{2.30})}},
we get
\begin{equation*}
C(1+K)e^{C(1+K)T} \ \ \geq \ \ \int_{{\rm Vol}_{g(t)}(B_{g(t)}(p,\sqrt{1+T}))}
|\Delta_{g(t)}u(t)|^{2}dV_{g(t)}
\end{equation*}
\begin{equation*}
\geq \ \ \frac{2C(1+K)}{(1+T)^{n/2}}\exp\left\{2C\left[(1+K)T
+1+2T+e^{C\sqrt{K}}\right]\right\}\cdot{\rm Vol}_{g(t)}
\left(B_{g(t)}(p,\sqrt{1+T})\right)
\end{equation*}
\begin{equation*}
\geq \ \ 2C(1+K) e^{C(1+K)T}\exp\left[C\left(1+2T+e^{C\sqrt{K}}\right)
\right] \ \ \geq \ \ 2C(1+K)e^{C(1+K)T}.
\end{equation*}
This contradiction shows that we must have {\color{blue}{(\ref{2.28})}}.
\end{proof}

\subsection{Local curvature estimates}\label{subsection2.2}

In this subsection we assume that
\begin{equation*}
|{\rm Ric}|\leq K, \ \ \ |\nabla u|\leq L, \ \ \
|\nabla^{2}u|\leq P
\end{equation*}
over an open subset $\Omega$ in $M$ and $\phi$ is a Lipschitz
function with support in $\Omega$.
\\

From {\color{red}{Lemma \ref{lA.1}}}, we can deduce that
\begin{eqnarray}
\Box|{\rm Ric}|^{2}&=&-2|\nabla{\rm Ric}|^{2}
+4R_{pijq}R^{pq}R^{ij}-8R_{pijq}R^{ij}\nabla^{p}u
\nabla^{q}u\nonumber\\
&&+ \ 8\Delta uR^{ij}\nabla_{i}\nabla_{j}u
-8R^{ij}\nabla_{i}\nabla_{k}u\nabla^{k}\nabla_{j}u
-8R_{ij}R_{k}{}^{j}\nabla^{i}u\nabla u.\label{2.31}
\end{eqnarray}
In particular
\begin{eqnarray}
|\nabla{\rm Ric}|^{2}&\leq&-\frac{1}{2}\Box|{\rm Ric}|^{2}
+CK^{2}|{\rm Rm}|+CKL^{2}|{\rm Rm}|\nonumber\\
&&+ \ CK|\nabla^{2}u||\Delta u|+CK|\nabla^{2}u|^{2}
+CK^{2}L^{2}\label{2.32}\\
&\leq&-\frac{1}{2}\Box|{\rm Ric}|^{2}+CK(L^{2}+K)|{\rm Rm}|
+CKP^{2}+CK^{2}L^{2},\nonumber
\end{eqnarray}
by the fact at $|\Delta u|\leq \sqrt{n}|\nabla^{2}u|$. Similarly,
from {\color{blue}{(\ref{A.6})}}, we have
\begin{equation}
|\nabla{\rm Rm}|^{2}\leq-\frac{1}{2}\Box|{\rm Rm}|^{2}
+C|{\rm Rm}|^{3}+C|{\rm Rm}|P^{2}+CL^{2}|{\rm Rm}|^{2}.\label{2.33}
\end{equation}
Moreover, we can prove that, see {\color{red}{Lemma \ref{lA.2}}},
\begin{eqnarray}
\partial_{t}|{\rm Rm}|^{2}
&=&\nabla^{2}{\rm Ric}\ast{\rm Rm}
+{\rm Ric}\ast{\rm Rm}\ast{\rm Rm}\nonumber\\
&&+ \ {\rm Rm}\ast\nabla^{2}u\ast\nabla^{2}u
+{\rm Rm}\ast{\rm Rm}\ast\nabla u\ast\nabla u.\label{2.34}
\end{eqnarray}
As in \cite{KMW2016}, we consider the quantity
\begin{equation*}
\frac{d}{dt}\int_{M}|{\rm Rm}|^{p}\phi^{2p}dV_{t}
\end{equation*}
which can be rewritten as, using {\color{blue}{(\ref{2.34})}},
$$
\frac{d}{dt}\int_{M}|{\rm Rm}|^{p}\phi^{2p}dV_{t} \ \ = \ \ \int_{M}\left(\partial_{t}|{\rm Rm}|^{p}\right)\phi^{2p}dV
+\int_{M}|{\rm Rm}|^{p}\phi^{2p}\left(-R+2|\nabla u|^{2}\right)
dV_{t}
$$
$$
= \ \ \frac{p}{2}\int_{M}|{\rm Rm}|^{p-2}
\bigg[\nabla^{2}{\rm Ric}\ast{\rm Rm}+{\rm Ric}\ast{\rm Rm}
\ast{\rm Rm}+{\rm Rm}\ast\nabla^{2}u\ast\nabla^{2}u
$$
$$
+ \ {\rm Rm}
\ast{\rm Rm}\ast\nabla u\ast\nabla u\bigg]\phi^{2p}dV_{t}-\int_{M}R|{\rm Rm}|^{p}\phi^{2p}dV_{t}+2\int_{M}|{\rm Rm}|^{p}
|\nabla u|^{2}\phi^{2p}dV_{t}
$$
$$
\leq \ \ C\int_{M}|{\rm Rm}|^{p-2}\left(\nabla^{2}{\rm Ric}\ast{\rm Rm}
\right)\phi^{2p}dV_{t}+CK\int_{M}|{\rm Rm}|^{p}
\phi^{2p}dV_{t}
$$
$$
+ \ C P^{2}\int_{M}|{\rm Rm}|^{p-1}\phi^{2p}dV+CL^{2}\int_{M}
|{\rm Rm}|^{p}\phi^{2p}dV_{t}
$$
$$
+ \ C\int_{M}|{\rm Rm}|^{p-2}
\left({\rm Rm}\ast\nabla u\ast\nabla^{3}u\right)\phi^{2p}dV_{t}.
$$
From (2.5), (2.6), and (2.7) in \cite{KMW2016}, we know that
$$
C\int_{M}|{\rm Rm}|^{p-2}\left(\nabla^{2}{\rm Ric}\ast{\rm Rm}\right)
\phi^{2p}dV_{t} \ \ \leq\ \ \frac{1}{K}\int_{M}|\nabla{\rm Ric}|^{2}|{\rm Rm}|^{p-1}
\phi^{2p}dV_{t}
$$
$$
+ \ \ CK\int_{M}|\nabla{\rm Rm}|^{2}|{\rm Rm}|^{p-3}\phi^{2p}dV_{t}
+CK\int_{M}|{\rm Rm}|^{p-1}|\nabla\phi|^{2}\phi^{2p-2}dV_{t}.
$$
Combining all terms yields
\begin{equation*}
\frac{d}{dt}\int_{M}|{\rm Rm}|^{p}\phi^{2p}dV_{t}\leq\frac{1}{K}\int_{M}|\nabla{\rm Ric}|^{2}|{\rm Rm}|^{p-1}\phi^{2p}dV_{t}
\end{equation*}
\begin{equation}
+ \ CK\int_{M}|\nabla{\rm Rm}|^{2}|{\rm Rm}|^{p-3}\phi^{2p}dV_{t}
+ CK\int_{M}|{\rm Rm}|^{p-1}|\nabla\phi|^{2}\phi^{2p-2}dV_{t}\label{2.35}
\end{equation}
\begin{equation*}
+ \ C(K+L^{2})\int_{M}|{\rm Rm}|^{p}\phi^{2p}dV+CP^{2}\int_{M}
|{\rm Rm}|^{p-1}\phi^{2p}dV_{t}
\end{equation*}
In {\color{blue}{(\ref{2.35})}} the first two terms are ``bad terms'', since these contain
derivatives of curvature. As in \cite{KMW2016} we set
\begin{equation*}
B_{1}:=\frac{1}{K}\int_{M}|\nabla{\rm Ric}|^{2}|{\rm Rm}|^{p-1}
\phi^{2p}dV_{t}, \ \ \ B_{2}:=\int_{M}|\nabla{\rm Rm}|^{2}|{\rm Rm}|^{p-3}
\phi^{2p}dV_{t}.
\end{equation*}
We also introduce
\begin{eqnarray*}
A_{1}&:=&\int_{M}|{\rm Rm}|^{p}|\phi^{2p}dV_{t}, \ \ \
A_{2} \ \ := \ \ \int_{M}|{\rm Rm}|^{p-1}\phi^{2p}dV_{t},\\
A_{3}&:=&\int_{M}|{\rm Rm}|^{p-1}|\nabla\phi|^{2}\phi^{2p-1}dV_{t}, \ \ \
A_{4} \ \ := \ \ \int_{M}|{\rm Rm}|^{p-1}|\nabla\phi|^{2}\phi^{2p-2}dV_{t}.
\end{eqnarray*}
Then the estimate {\color{blue}{(\ref{2.35})}} can be rewritten as
\begin{equation}
\frac{d}{dt}\int_{M}|{\rm Rm}|^{p}\phi^{2p}dV_{t}
\leq B_{1}+CK B_{2}
+CKA_{4}+C(K+L^{2})A_{1}+CP^{2}A_{2}.\label{2.36}
\end{equation}
Using {\color{blue}{(\ref{2.32})}} yields
$$
B_{1} \ \ \leq \ \ \int_{M}
\left[\frac{1}{2K}\left(\Delta-\partial_{t}\right)
|{\rm Ric}|^{2}+C(L^{2}+K)|{\rm Rm}|+C(P^{2}+KL^{2})\right]
|{\rm Rm}|^{p-1}\phi^{2p}dV_{t}
$$
$$
= \ \frac{1}{2K}\int_{M}\left[\left(\Delta-\partial_{t}\right)|{\rm Ric}|^{2}
\right]|{\rm Rm}|^{p-1}\phi^{2p}dV_{t}+C(L^{2}+K)A_{1}
+C(P^{2}+KL^{2})A_{2}
$$
$$
= \ \ \frac{1}{2K}\int_{M}\left(\Delta|{\rm Ric}|^{2}\right)|{\rm Rm}|^{p-1}
\phi^{2p}dV_{t}+C(L^{2}+K)A_{1}+C(P^{2}+KL^{2})A_{2}
$$
$$
- \ \frac{1}{2K}\int_{M}\bigg[\partial_{t}\left(|{\rm Ric}|^{2}|{\rm Rm}|^{p-1}
\phi^{2p}dV_{t}\right)-|{\rm Ric}|^{2}\left(\partial_{t}|{\rm Rm}|^{p-1}
\right)\phi^{2p}dV_{t}
$$
$$
- \ |{\rm Ric}|^{2}|{\rm Rm}|^{p-1}\phi^{2p}\left(-R+2|\nabla u|^{2}
\right)dV_{t}\bigg]
$$
$$
= \ \ -\frac{1}{2K}\left[\int_{M}\left\langle\nabla|{\rm Ric}|^{2},\nabla|{\rm Rm}|^{p-1}
\right\rangle\phi^{2p}dV_{t}+\int_{M}\left\langle\nabla|{\rm Ric}|^{2},
\nabla\phi^{2p}\right\rangle|{\rm Rm}|^{p-1}dV_{t}\right]
$$
$$
- \ \frac{1}{2K}\frac{d}{dt}\int_{M}|{\rm Ric}|^{2}|{\rm Rm}|^{p-1}
\phi^{2p}dV_{t}+C(L^{2}+K)A_{1}+C(P^{2}+KL^{2})A_{2}
$$
$$
+ \ \frac{1}{2K}\int_{M}|{\rm Ric}|^{2}
\left(\partial_{t}|{\rm Rm}|^{p-1}\right)\phi^{2p}dV_{t}+CKL^{2}A_{2}+CKA_{1}
$$
$$
\leq \ \ -\frac{1}{2K}\left[\int_{M}\left\langle\nabla|{\rm Ric}|^{2},\nabla|{\rm Rm}|^{p-1}
\right\rangle\phi^{2p}dV_{t}+\int_{M}\left\langle\nabla|{\rm Ric}|^{2},
\nabla\phi^{2p}\right\rangle|{\rm Rm}|^{p-1}dV_{t}\right]
$$
$$
- \ \frac{1}{2K}\frac{d}{dt}\int_{M}|{\rm Ric}|^{2}|{\rm Rm}|^{p-1}
\phi^{2p}dV_{t}+\frac{1}{2K}\int_{M}|{\rm Ric}|^{2}
\left(\partial_{t}|{\rm Rm}|^{p-1}\right)\phi^{2p}dV_{t}
$$
$$
+ \ C(L^{2}+K)A_{1}+C(P^{2}+KL^{2})A_{2}.
$$
From (2.10) and (2.11) in \cite{KMW2016}, one has
\begin{eqnarray}
-\frac{1}{2K}\int_{M}\left\langle\nabla|{\rm Ric}|^{2},
\nabla|{\rm Rm}|^{p-1}\right\rangle\phi^{2p}dV_{t}
&\leq&\frac{1}{10}B_{1}+CK B_{2},\label{2.37}\\
-\frac{1}{2K}\int_{M}\left\langle\nabla|{\rm Ric}|^{2},\nabla
\phi^{2p}\right\rangle|{\rm Rm}|^{p-1}dV_{t}
&\leq&\frac{1}{10}B_{1}+CKA_{4}.\label{2.38}
\end{eqnarray}
According to {\color{blue}{(\ref{2.34})}}, we have
\begin{equation*}
\frac{1}{2K}\int_{M}|{\rm Ric}|^{2}\left(\partial_{t}|{\rm Rm}|^{p-1}
\right)\phi^{2p}dV_{t} \ \ = \ \ \frac{p-1}{4K}\int_{M}|{\rm Ric}|^{2}
\left(|{\rm Rm}|^{p-3}\partial_{t}|{\rm Rm}|^{2}\right)
\phi^{2p}dV_{t}
\end{equation*}
\begin{eqnarray*}
&=&\frac{C}{K}\int_{M}|{\rm Ric}|^{2}|{\rm Rm}|^{p-3}
\phi^{2p}\bigg[\nabla^{2}{\rm Ric}\ast{\rm Rm}+{\rm Ric}\ast{\rm Rm}
\ast{\rm Rm}\\
&&+ \ {\rm Rm}\ast\nabla^{2}u\ast\nabla^{2}u+{\rm Rm}
\ast{\rm Rm}\ast\nabla u\ast\nabla u\bigg]dV_{t}\\
&\leq&\frac{C}{K}\int_{M}|{\rm Ric}|^{2}|{\rm Rm}|^{p-3}
\phi^{2p}\left(\nabla^{2}{\rm Ric}\ast{\rm Rm}\right)dV_{t}+CKA_{1}+C(P^{2}+KL^{2})
A_{2}
\end{eqnarray*}
From the proof of (\ref{2.13}) -- (\ref{2.15}) in \cite{KMW2016}, we can deduce
that
\begin{equation*}
\frac{C}{K}\int_{M}|{\rm Ric}|^{2}|{\rm Rm}|^{p-3}
\phi^{2p}\left(\nabla^{2}{\rm Ric}\ast{\rm Rm}\right)dV
\leq\frac{1}{5}B_{1}+CKB_{2}+CKA_{4}.
\end{equation*}
In summary, we arrive at
\begin{eqnarray}
\frac{1}{2K}\int_{M}|{\rm Ric}|^{2}\left(\partial_{t}|{\rm Rm}|^{p-1}
\right)\phi^{2p}dV_{t}&\leq&\frac{1}{5}B_{1}+CKB_{2}
+CKA_{1}\nonumber\\
&&+ \ C(P^{2}+KL^{2})A_{2}+CKA_{4}.\label{2.39}
\end{eqnarray}
Plugging {\color{blue}{(\ref{2.37})}}, {\color{blue}{(\ref{2.38})}}, and {\color{blue}{(\ref{2.39})}} into the inequality for
$B_{1}$, we get
\begin{eqnarray}
B_{1}&\leq&CKB_{2}+C(K+L^{2})A_{1}+CKA_{4}\nonumber\\
&&+ \ C(P^{2}+KL^{2})A_{2}
-\frac{1}{2K}\frac{d}{dt}\int_{M}|{\rm Ric}|^{2}|{\rm Rm}|^{p-1}\phi^{2p}
dV_{t}.\label{2.40}
\end{eqnarray}
To deal with the term $B_{2}$, we use the evolution equation {\color{blue}{(\ref{2.33})}} and
then obtain
\begin{eqnarray*}
B_{2}&\leq&\int_{M}\left[\frac{1}{2}\left(\Delta-\partial_{t}\right)
|{\rm Rm}|^{2}+C|{\rm Rm}|^{3}+C(L^{2}+P^{2})|{\rm Rm}|^{2}
\right]|{\rm Rm}|^{p-3}\phi^{2p}dV_{t}\\
&=&\frac{1}{2}\int_{M}\left(\Delta|{\rm Rm}|^{2}\right)
|{\rm Rm}|^{p-3}\phi^{2p}dV_{t}+CA_{1}+C(L^{2}+P^{2})A_{2}\\
&&- \ \frac{1}{2}\int_{M}\left(\partial_{t}|{\rm Rm}|^{2}\right)
|{\rm Rm}|^{p-3}\phi^{2p}dV_{t}\\
&\leq&C\int_{M}|\nabla{\rm Rm}||\nabla\phi||{\rm Rm}|^{p-2}
\phi^{2p-1}dV_{t}+CA_{1}+C(L^{2}+P^{2})A_{2}\\
&&- \ \frac{1}{2}\int_{M}\left(\partial_{t}|{\rm Rm}|^{2}\right)
|{\rm Rm}|^{p-3}\phi^{2p}dV_{t}\\
&\leq&\frac{1}{2}B_{2}+CA_{4}+CA_{1}+C(L^{2}+P^{2})A_{2}
-\frac{1}{2}\int_{M}\left(\partial_{t}|{\rm Rm}|^{2}\right)
|{\rm Rm}|^{p-3}\phi^{2p}dV_{t}.
\end{eqnarray*}
As the proof of (2.18) -- (2.19) in \cite{KMW2016}, we have
\begin{equation*}
-\frac{1}{2}\int_{M}\left(\partial_{t}|{\rm Rm}|^{2}\right)
|{\rm Rm}|^{p-3}\phi^{2p}dV_{t} \ \ = \ \ -\frac{1}{2}
\int_{M}\bigg[\partial_{t}\left(|{\rm Rm}|^{2}|{\rm Rm}|^{p-3}\phi^{2p}dV_{t}
\right)
\end{equation*}
\begin{equation*}
-|{\rm Rm}|^{2}\left(\partial_{t}|{\rm Rm}|^{p-3}\right)
\phi^{2p}dV_{t}-|{\rm Rm}|^{p-1}
\phi^{2p}\partial_{t}dV\bigg]
\end{equation*}
\begin{equation*}
= \ \ -\frac{1}{2}\frac{d}{dt}\int_{M}|{\rm Rm}|^{p-1}\phi^{2p}dV_{t}
+\frac{p-3}{4}\int_{M}|{\rm Rm}|^{p-3}\left(\partial_{t}|{\rm Rm}|^{2}
\right)\phi^{2p}dV_{t}
\end{equation*}
\begin{equation*}
-\frac{1}{2}\int_{M}R|{\rm Rm}|^{p-1}\phi^{2p}dV_{t}
+\int_{M}|{\rm Rm}|^{p-1}|\nabla u|^{2}\phi^{2p}dV_{t}
\end{equation*}
and therefore
\begin{equation*}
-\frac{1}{2}\int_{M}\left(\partial_{t}|{\rm Rm}|^{2}\right)
|{\rm Rm}|^{p-3}\phi^{2p}dV_{t}
\leq-\frac{1}{p-1}\frac{d}{dt}\int_{M}|{\rm Rm}|^{p-1}\phi^{2p}dV_{t}
+CA_{1}+CL^{2}A_{2}.
\end{equation*}
In summary,
\begin{equation}
B_{2}\leq-\frac{1}{p-1}\frac{d}{dt}\int_{M}|{\rm Rm}|^{p-1}
\phi^{2p}dV_{t}+CA_{1}+CA_{4}+C(L^{2}+P^{2})A_{2}.\label{2.41}
\end{equation}
From {\color{blue}{(\ref{2.36})}}, {\color{blue}{(\ref{2.40})}}, and {\color{blue}{(\ref{2.41})}}, we finally obtain
\begin{equation*}
\frac{d}{dt}\left[A_{1}+CKA_{2}
+\frac{1}{2K}\int_{M}|{\rm Ric}|^{2}|{\rm Rm}|^{p-1}\phi^{2p}dV_{t}\right] \ \
\leq \ \ C(K+L^{2})A_{1}
\end{equation*}
\begin{equation*}
+ \ CKA_{4}
+C(KL^{2}+KP^{2}+P^{2}+KL)A_{2}.
\end{equation*}

\begin{theorem}\label{t2.7} Let $(g(t),u(t))_{t\in[0,T]}$ be a solution to the Ricci-haronic
flow. Suppose there exist constants $\rho, K, L, P>0$ and $x_{0}\in M$ such that $
B_{g(0)}(x_{0},\rho/\sqrt{K})$ is compactly contained on $M$ and
\begin{equation*}
|{\rm Ric}_{g(t)}|_{g(t)}
\leq K, \ \ \ |\nabla_{g(t)}u(t)|_{g(t)}
\leq L, \ \ \ |\nabla^{2}_{g(t)}u(t)|_{g(t)}
\leq P
\end{equation*}
on $B_{g(0)}(x_{0},\rho/\sqrt{K})\times[0,T]$. For any $p\geq3$, there is a
constant $C$, depending only on $n$ and $p$, such that
\begin{equation*}
\int_{B_{g(0)}(x_{0},\rho/2\sqrt{K})}|{\rm Rm}_{g(t)}|^{p}_{g(t)}
dV_{g(t)} \ \ \leq \ \ C\Lambda_{1} e^{C\Lambda_{2} T}\int_{B_{g(0)}(x_{0},\rho/\sqrt{K})}
|{\rm Rm}_{g(0)}|^{p}_{g(0)}dV_{g(0)}
\end{equation*}
\begin{equation*}
+ \ C K^{p}\left(1+\rho^{-2p}\right)
e^{C\Lambda_{2} T}{\rm Vol}_{g(t)}
\left(B_{g(0)}\left(x_{0},\frac{\rho}{\sqrt{K}}\right)
\right).
\end{equation*}
Here $\Lambda_{1}:=1+K$ and $\Lambda_{2}:=K+L+L^{2}+P^{2}(1+K^{-1})$.
\end{theorem}

\begin{proof} Choose $\Omega:=B_{g(0)}(x_{0},\rho/\sqrt{K})$ and
\begin{equation*}
\phi:=\left(\frac{\rho/\sqrt{K}-d_{g(0)}(x_{0},\cdot)}{\rho/\sqrt{K}}\right)_{+}.
\end{equation*}
Then $e^{-2Kt}g(0)\leq g(t)\leq e^{2Kt}g(0)$ and $|\nabla_{g(t)}
\phi|_{g(t)}\leq e^{KT}|\nabla_{g(0)}\phi|_{g(0)}\leq\sqrt{K} e^{KT}/\rho$
for any $t\in[0,T]$. Let
\begin{equation*}
U:=\int_{M}|{\rm Rm}|^{p}\phi^{2p}dV_{t}
+CK\int_{M}|{\rm Rm}|^{p-1}\phi^{2p}dV_{t}+\frac{1}{2K}\int_{M}|{\rm Ric}|^{2}|{\rm Rm}|^{p-1}\phi^{2p}dV_{t}.
\end{equation*}
Then
\begin{equation*}
U' \ \ \leq \ \ C(K+L^{2})U+CKA_{4}+C(KL^{2}+KP^{2}+P^{2}+KL)\frac{U}{K}.
\end{equation*}
For $A_{4}$, we can estimate it as follows:
\begin{eqnarray*}
A_{4}&=&\int_{M}|{\rm Rm}|^{p-1}|\nabla\phi|^{2}
\phi^{2p-2}dV_{t} \ \ \leq \ \ \int_{B_{g(0)}(x_{0},\rho/\sqrt{K})}
|{\rm Rm}|^{p-1}\phi^{2p-2} K\rho^{-2}e^{2KT}dV_{t}\\
&\leq&\int_{B_{g(0)}(x_{0},\rho/\sqrt{K})}
\left[\frac{(|{\rm Rm}|^{p-1}\phi^{2p-2})^{\frac{p}{p-1}}}{\frac{p}{p-1}}
+\frac{(K\rho^{-2}e^{2KT})^{p}}{p}\right]dV_{t}\\
&\leq&A_{1}+K^{p}e^{2KpT}{p\rho^{-2p}}{\rm Vol}_{g(t)}
\left(B_{g(0)}\left(x_{0},\frac{\rho}{\sqrt{K}}\right)\right)\\
&\leq& U+C K^{p}\rho^{-2p}e^{2KpT}{\rm Vol}_{g(t)}
\left(B_{g(0)}\left(x_{0},\frac{\rho}{\sqrt{K}}\right)\right).
\end{eqnarray*}
Hence
\begin{equation*}
U' \ \ \leq \ \ C\left[K+L^{2}+L+P^{2}\left(1+\frac{1}{K}
\right)\right]U
\end{equation*}
\begin{equation*}
+ \ CK^{p+1}\rho^{-2p}e^{2KpT}{\rm Vol}_{g(t)}
\left(B_{g(0)}\left(x_{0},\frac{\rho}{\sqrt{K}}\right)\right).
\end{equation*}
Since, for each $\tau\in[0,T]$,
\begin{equation*}
{\rm Vol}_{g(t)}\left(B_{g(0)}\left(x_{0},
\frac{\rho}{\sqrt{K}}\right)\right)
\leq e^{CKT}{\rm Vol}_{g(\tau)}
\left(B_{g(0)}\left(x_{0},\frac{\rho}{\sqrt{K}}\right)\right),
\end{equation*}
as argued in (2.27) of \cite{KMW2016}, we deduce that
\begin{equation}
U(\tau)\leq e^{C\Lambda_{2} T}
\left[U(0)+C\rho^{-2p}K^{p}{\rm Vol}_{g(t)}
\left(B_{g(0)}\left(x_{0},\frac{\rho}{\sqrt{K}}\right)\right)\right],
\label{2.42}
\end{equation}
According to the Young inequality
\begin{eqnarray*}
\int_{M}|{\rm Rm}_{g(0)}|^{p-1}\phi^{2p}dV_{g(0)}&=&\int_{M}\left(|{\rm Rm}|^{p-1}\phi^{2p-2}\right)\phi^{2}dV_{g(0)}\\
&\leq&\frac{p-1}{p}\int_{M}|{\rm Rm}_{g(0)}|^{p}\phi^{2p}dV_{g(t)}
+\frac{1}{p}\int_{M}\phi^{2p}dV_{g(0)},
\end{eqnarray*}
we obtain
\begin{eqnarray*}
U(0)&\leq&C(1+K)\int_{M}|{\rm Rm}_{g(0)}|^{p}\phi^{2p}dV_{g(0)}+CK{\rm Vol}_{g(0)}\left(B_{g(0)}\left(x_{0},\frac{\rho}{\sqrt{K}}
\right)\right)\\
&\leq&C(1+K)\int_{M}|{\rm Rm}_{g(0)}|^{2p}_{g(0)}\phi^{2p}dV_{g(0)}\\
&&+ \ CKe^{CKT}{\rm Vol}_{g(\tau)}\left(B_{g(0)}
\left(x_{0},\frac{\rho}{\sqrt{K}}\right)\right)\\
&\leq&C(1+K)\int_{B_{g(0)}(x_{0},\rho/\sqrt{K})}
|{\rm Rm}_{g(0)}|^{2p}_{g(0)}dV_{g(0)}\\
&&+ \ CKe^{CKT}{\rm Vol}_{g(\tau)}
\left(B_{g(0)}\left(x_{0},\frac{\rho}{\sqrt{K}}\right)\right),
\end{eqnarray*}
and $\phi\geq1/2$ on $B_{g(0)}(x_{0},\rho/2\sqrt{K})$, we complete the proof.
\end{proof}

The same method can be applied to the regular Ricci flow (see {\bf Section} \ref{section3}), since all computations only involve the evolution equations for the metrics, which take the same forms in the Ricci-harmonic flow.

\section{Results for a generalized Ricci flow}\label{section3}

In this section we extend the main estimates in \cite{LY2} to a generalized Ricci flow introduced in \cite{LY1}, where I proved that for any complete $n$-manifold $M$, the following two conditions are equivalent:
\begin{itemize}

\item[(i)] there exists a Ricci-flat Riemannian metric on $M$;

\item[(ii)] there exists real numbers $\alpha,\beta$, a smooth function $u$ on $M$, and a Riemannian metric $g$ on $M$ such that
    \begin{equation}
    0=-R_{ij}+\alpha\nabla_{i}\nabla_{j}u, \ \ \ 0=\Delta_{g}
    u+\beta|\nabla_{g}u|^{2}_{g}.\label{3.1}
    \end{equation}

\end{itemize}
The main ingredient in the proof is Chen's result \cite{Chen2009} which says that any complete noncompact steady gradient Ricci soliton has nonnegative scalar curvature.

Observe that the first equation in {\color{blue}{(\ref{3.1})}} is actual the vanishing of the $\infty$-Bakry-\'Emery Ricci tensor. Indeed, the $N$-Bakry-\'Emery Ricci tensor is defined by
\begin{equation}
{\rm Ric}_{g,N,f}:={\rm Ric}_{g}+\nabla^{2}f-\frac{df\otimes df}{N-n}\label{3.2}
\end{equation}
for $N$ finite, and
\begin{equation}
{\rm Ric}_{g,\infty,f}:={\rm Ric}_{g}+\nabla^{2}f\label{3.3}
\end{equation}
for $N$ infinite. Thus the first equation in {\color{blue}{(\ref{3.1})}} is equivalent to ${\rm Ric}_{g,\infty,-\alpha u}=0$.
\\

Motivated by the above equivalence conditions I introduced the following generalized Ricci flow \cite{LY1}:
\begin{eqnarray}
\partial_{t}g(t)&=&-2\!\ {\rm Ric}_{g(t)}
+2\alpha_{1}\nabla_{g(t)} u(t)\otimes \nabla_{g(t)}u(t)+2\alpha_{2}\nabla^{2}_{g(t)}u(t),\label{3.4}\\
\partial_{t}u(t)&=&\Delta_{g(t)}u(t)+\beta_{1}|\nabla_{g(t)}
u(t)|^{2}_{g(t)}
+\beta_{2}u(t).\label{3.5}
\end{eqnarray}
Here $\alpha_{1},\alpha_{2},\beta_{1},\beta_{2}$ are given constants. This system is called {\it $(\alpha_{1},\alpha_{2},\beta_{1},\beta_{2})$-Ricci flow}. In particular, when $(\alpha_{1},\alpha_{2},\beta_{1},\beta_{2})
=(2,0,0,0)$, we get the Ricci-harmonic flow {\color{blue}{(\ref{1.1})}}. In view of {\color{blue}{(\ref{3.2})}}, we see that {\color{blue}{(\ref{3.4})}} can be written as
\begin{equation*}
\partial_{t}g(t)=-2\!\ {\rm Ric}_{g(t),N,-\alpha_{2}u(t)},
\end{equation*}
where $N=n+\alpha^{2}_{2}/\alpha_{1}$ if $\alpha_{1}\neq0$, and $N=\infty$ if $\alpha_{1}=0$.
\\

According to Proposition 2.12 in \cite{LY1}, we know that an $(\alpha_{1},\alpha_{2},\beta_{1},\beta_{2})$-Ricci flow is equivalent to the $(\alpha_{1},0,\beta_{1}-\alpha_{2},\beta_{2})$-flow. By this reduction, in the following we main study the {\it $(\alpha_{1},0,\beta_{1},\beta_{2})$-Ricci flow}:
\begin{eqnarray}
\partial_{t}g(t)&=&-2\!\ {\rm Ric}_{g(t)}
+2\alpha_{1}\nabla_{g(t)}u(t)\otimes\nabla_{g(t)}
u(t),\label{3.6}\\
\partial_{t}u(t)&=&\Delta_{g(t)}u(t)+\beta_{1}|\nabla_{g(t)}
u(t)|^{2}_{g(t)}
+\beta_{2}u(t).\label{3.7}
\end{eqnarray}
Here $\alpha_{1},\beta_{1},\beta_{2}$ are given constants. Recall the notion $\Box_{g(t)}=\partial_{t}-\Delta_{g(t)}$ and introduce as in \cite{LY2},
\begin{eqnarray}
{\rm Sic}_{g(t)}&:=&{\rm Ric}_{g(t)}
-\alpha_{1}\nabla_{g(t)}u(t)\otimes\nabla_{g(t)}
u(t),\label{3.8}\\
S_{g(t)}&:=&{\rm tr}_{g(t)}{\rm Sic}_{g(t)} \ \ = \ \
R_{g(t)}-\alpha_{1}|\nabla_{g(t)}u(t)|^{2}_{g(t)}
.\label{3.9}
\end{eqnarray}

Another interesting flow involving $(g(t),u(t))$ is the so-called {\it super-Ricci flow} introduced by X. D. Li and S. Z. Li \cite{LL2014}
and in terms of our notions {\color{blue}{(\ref{3.2})}} and {\color{blue}{(\ref{3.3})}} can be written as, according to whether or not $N$ is infinity
\begin{equation*}
\frac{1}{2}\partial_{t}g(t)+{\rm Ric}_{g(t),N,u(t)}\geq K g(t)
\end{equation*}
which is called a {\it $(K,N)$-super Ricci flow}, where $N\in\mathbb{R}$, and, respectively,
\begin{equation*}
\frac{1}{2}\partial_{t}g(t)+{\rm Ric}_{g(t),\infty,u(t)}
\geq Kg(t)
\end{equation*}
which is called a {\it $K$-super Perelman Ricci flow}. Under the super-Ricci flow, the authors studied Harnack inequalities \cite{LL2014, LL20171, LL20173}, $W$-entropy formulas \cite{LL2014, LL2015, LL20172, LL20173, LL20174}, $(K,N)$-Ricci solitons \cite{LL20172}, etc. For example, they proved that if $(g(t), u(t))_{t\in[0,T]}$ satisfies
\begin{equation*}
\frac{1}{2}\partial_{t}g(t)+{\rm Ric}_{g(t),N, u(t)}
=0, \ \ \ \partial_{t}u(t)=\frac{1}{2}{\rm tr}_{g(t)}
\left(\partial_{t}g(t)\right),
\end{equation*}
then the $W$-entropy $\mathcal{W}_{N}(f(t)):=\frac{d}{dt}\left[t\!\ \mathcal{H}_{N}(f(t))\right]$ with
\begin{equation*}
\mathcal{H}_{N}(f(t)):=-\int_{M}f(t)
\ln f(t)\!\ dV_{g(t)}-\frac{N}{2}
\left[1+\ln(4\pi t)\right]
\end{equation*}
is constant along the flow, where $f(t)$ is the fundamental solution to the heat equation $\partial_{t}f(t)=\Delta_{g(t)}f(t)-\langle\nabla_{g(t)}u(t),\nabla_{g(t)}
f(t)\rangle_{g(t)}$. However, we can {\it not} apply this result to our flow {\color{blue}{(\ref{3.4})}} or
\begin{equation*}
\partial_{t}g(t)=-2{\rm Ric}_{g(t),N,-\alpha_{2}u(t)},
\end{equation*}
since the second equation {\color{blue}{(\ref{3.5})}} may not satisfy the constraint equation $\partial_{t}u(t)
=\frac{1}{2}{\rm tr}_{g(t)}\partial_{t}g(t)$.
\\

In \cite{LY2} we also introduced a ``Riemann curvature" type for RHF
\begin{equation}
S_{ijk\ell}:=R_{ijk\ell}-\frac{\alpha_{1}}{2}
\left(g_{j\ell}\nabla_{i}u\nabla_{k}u
+g_{k\ell}\nabla_{i}u\nabla_{j}u\right)\label{3.10}
\end{equation}
so that $S_{ij}=g^{k\ell}S_{ik\ell j}=g^{k\ell}S_{kij\ell}=R_{ij}
-\alpha_{1}\nabla_{i}u\nabla_{j}u$. A related construction of ``Riemann curvature" type for {\color{blue}{(\ref{3.2})}} is given in \cite{WY2016}. A detailed
discussion of these two notions of ``Riemann curvature'' will be given in {\bf Section} \ref{section5}.
\\

According to {\color{red}{Lemma \ref{lC.1}}}, we have

\begin{lemma}\label{l3.1} Under the flow {\color{blue}{(\ref{3.6})}} -- {\color{blue}{(\ref{3.7})}},
\begin{eqnarray}
\Box S&=&2|{\rm Sic}|^{2}
+2\alpha_{1}|\Delta u|^{2}-2\alpha_{1}\beta_{2}|\nabla u|^{2}
-4\alpha_{1}\beta_{1}\nabla^{i}u\nabla^{j}u\nabla_{i}\nabla_{j}u\nonumber\\
&=&2|{\rm Sic}|^{2}+2\alpha_{1}|\Delta u|^{2}
-2\alpha_{1}\beta_{2}|\nabla u|^{2}-4\alpha_{1}\beta_{1}
\langle\nabla^{2}u,\nabla u\otimes\nabla u\rangle,\label{3.11}\\
\Box S_{ij}&=&2S_{kij\ell}S^{k\ell}-2S_{ik}S^{k}{}_{j}+2\alpha_{1}\Delta u
\nabla_{i}\nabla_{j}u\nonumber\\
&&- \ 2\alpha_{1}\beta_{2}\nabla_{i}u\nabla_{j}u
-2\alpha_{1}\beta_{1}\nabla^{k}u(\nabla_{i}u\nabla_{j}\nabla_{k}u
+\nabla_{j}u\nabla_{i}\nabla_{k}u).\label{3.12}
\end{eqnarray}

\end{lemma}

\subsection{Long time
existence}\label{subsection3.1}

Given an initial data $(g_{0},u_{0})$, define $c_{0}:=|\nabla_{g_{0}}u_{0}|^{2}_{g_{0}}$. We always assume that $c_{0}$ is a positive number. According to Corollary 2.10 and Definition 2.11 in \cite{LY1}, we say the $(\alpha_{1},0,\beta_{1},\beta_{2})$-Ricci flow is {\it regular}, if $\alpha_{1},\beta_{1},\beta_{2}$ satisfy one of the following conditions:

\begin{itemize}

\item[(i)] $\beta_{2}\leq0$ and $\alpha_{1}\geq\beta^{2}_{1}$;

\item[(ii)] $\beta_{2}>0$ and $c^{-1}_{0}\beta_{2}+\beta^{2}_{1}\geq\alpha_{1}>\beta^{2}_{1}$.

\end{itemize}
Then Corollary 2.10 in \cite{LY1} tells us that
\begin{equation}
|\nabla u|^{2}\lesssim1\label{3.13}
\end{equation}
along the $(\alpha_{1},0,\beta_{1},\beta_{2})$-Ricci flow equations
{\color{blue}{(\ref{3.6})}} -- {\color{blue}{(\ref{3.7})}}, where $\lesssim$ depends only on $\alpha_{1},\beta_{1},\beta_{2}$ and $c_{0}$.

\begin{theorem}\label{t3.2} Let $(g(t),u(t))_{t\in[0,T)}$ be a solution to the regular $(\alpha_{1},0,\beta_{1},\beta_{2})$-Ricci flow on a closed $n$-dimensional manifold $M$ with $T\leq\infty$ and the initial data $(g_{0},u_{0})$. Assume that $S_{g(t)}+C\geq C_{0}>0$ along the flow for some uniform constants $C, C_{0}>0$. Then
\begin{equation}
\frac{|{\rm Sin}_{g(t)}|_{g(t)}}{S_{g(t)}+C}
\leq C_{1}+C_{2}\max_{M\times[0,t]}
\sqrt{\frac{|W_{g(s)}|_{g(s)}+|\nabla^{2}_{g(s)}u(s)|^{2}_{g(s)}}{S_{g(s)}
+C}}
\end{equation}
where ${\rm Sin}_{g(t)}:={\rm Sic}_{g(t)}
-\frac{S_{g(t)}}{n}g(t)$ is the trace-free part of ${\rm Sic}_{g(t)}$ and $W_{g(t)}$ is the Weyl tensor field of $g(t)$.
\end{theorem}

Here the assumption $S_{g(t)}+C>0$ is necessary in the theorem, since, due to the undermined sign of $\alpha_{1}, \beta_{1}, \beta_{2}$, we can not in general deduce any bounds for $S_{g(t)}$ from the evolution equation {\color{blue}{(\ref{3.11})}}. In the simplest case, $\alpha_{1}\geq0$ and $\beta_{1}
=\beta_{2}=0$ (i.e., Ricci-harmonic flow), we have a lower bound from {\color{blue}{(\ref{3.11})}}.

\begin{proof} As in \cite{LY2}, consider the quantity
\begin{equation}
f:=\frac{|{\rm Sin}_{g(t)}|^{2}_{g(t)}}{(S_{g(t)}+C)^{\gamma}}
=\frac{|{\rm Sic}_{g(t)}+\frac{C}{n}g(t)|^{2}_{g(t)}}{(S_{g(t)}
+C)^{\gamma}}-\frac{1}{n}(S_{g(t)}+C)^{2-\gamma}, \ \ \
\gamma>0\label{3.15}
\end{equation}
and set
\begin{equation}
{\rm Sic}'_{g(t)}:={\rm Sic}_{g(t)}+\frac{C}{n}g(t), \ \ \
S'_{g(t)}:=S_{g(t)}+C.\label{3.16}
\end{equation}
From the identity (3.21) in \cite{CLN2006}, we have
\begin{eqnarray*}
\Box\frac{|{\rm Sic}'|^{2}}{(S')^{\gamma}}
&=&\frac{1}{(S')^{\gamma}}\Box|{\rm Sic}'|^{2}
-\gamma\frac{|{\rm Sic}'|^{2}}{(S')^{\gamma+1}}
\Box S-\gamma(\gamma+1)\frac{|{\rm Sic}'|^{2}}{(S')^{\gamma+2}}
|\nabla S'|^{2}\\
&&+ \ \frac{2\gamma}{(S')^{\gamma+1}}\left\langle\nabla|{\rm Sic}'|^{2},
\nabla S'\right\rangle.
\end{eqnarray*}
Using {\color{blue}{(\ref{3.12})}}, we get
\begin{eqnarray}
\Box |{\rm Sic}|^{2}&=&-2|\nabla{\rm Sic}|^{2}
+4{\rm Sm}({\rm Sic},{\rm Sic})+2\left\langle{\rm Sic},
2\alpha_{1}\Delta u\nabla^{2}u-2\alpha_{1}\beta_{2}\nabla u
\otimes \nabla u\right\rangle\nonumber\\
&&- \ 4\alpha_{1}\beta_{1}\left\langle{\rm Sic},\nabla u
\otimes\nabla|\nabla u|^{2}\right\rangle\label{3.17}
\end{eqnarray}
where ${\rm Sm}({\rm Sic},{\rm Sic})=S_{kij\ell}S^{ij}S^{k\ell}$. As in \cite{LY2}, we can prove
\begin{eqnarray*}
\Box|{\rm Sic}'|^{2}&=&\Box|{\rm Sic}|^{2}+\frac{2C}{n}\Box S\\
&=&-2|\nabla{\rm Sic}'|^{2}
+4{\rm Sm}({\rm Sic},{\rm Sic})
+\frac{4C}{n}|{\rm Sic}|^{2}\\
&&+ \ 4\alpha_{1}\left\langle{\rm Sic}',\Delta u\nabla^{2}u
-\beta_{2}\nabla u\otimes\nabla u\right\rangle
-4\alpha_{1}\beta_{1}\left\langle{\rm Sic}',\nabla u\otimes\nabla|\nabla u|^{2}\right\rangle\\
&=&-2|\nabla{\rm Sic}'|^{2}
+4{\rm Sm}({\rm Sic}',{\rm Sic}')-\frac{4C}{n}|{\rm Sic}'|^{2}
+\frac{4C^{2}}{n^{2}}S'\\
&&+ \ 4\alpha_{1}\left\langle{\rm Sic}',\Delta u\nabla^{2}u
-\beta_{2}\nabla u\otimes\nabla u\right\rangle
-4\alpha_{1}\beta_{1}\left\langle{\rm Sic}',\nabla u\otimes\nabla|\nabla u|^{2}\right\rangle
\end{eqnarray*}
and
\begin{eqnarray*}
\Box\frac{|{\rm Sic}'|^{2}}{(S')^{\gamma}}
&=&-\frac{2}{(S')^{\gamma}}
|\nabla{\rm Sic}'|^{2}
-\frac{2\gamma|{\rm Sic}'|^{4}}{(S')^{\gamma+1}}
+\frac{4}{(S')^{\gamma}}{\rm Sm}({\rm Sic}',{\rm Sic}')\\
&&- \ \gamma(\gamma+1)\frac{|{\rm Sic}'|^{2}|\nabla S'|^{2}}{(S')^{\gamma+2}}+\frac{2\gamma}{(S')^{\gamma+1}}
\langle\nabla|{\rm Sic}'|^{2},\nabla S'\rangle+\frac{4C^{2}}{n^{2}}\frac{S'}{(S')^{\gamma}}\\
&&- \ \frac{2C}{n}\frac{2(1-\gamma)S'+\gamma C}{(S')^{\gamma+1}}
|{\rm Sic}'|^{2}+\frac{4\alpha_{1}}{(S')^{\gamma}}\langle{\rm Sic}',\Xi\rangle
-\frac{2\alpha_{1}\gamma|{\rm Sic}'|^{2}}{(S')^{\gamma+1}}{\rm tr}\Xi
\end{eqnarray*}
where ${\rm tr}\Xi$ is the trace of $\Xi$ with respect to $g(t)$, and
\begin{equation}
\Xi:=\Delta u\nabla^{2}u
-\beta_{2}\nabla u\otimes\nabla u
-\beta_{1}\nabla u\otimes\nabla|\nabla u|^{2}.\label{3.18}
\end{equation}
From the identities
\begin{eqnarray*}
\left\langle\nabla\frac{|{\rm Sic}'|^{2}}{(S')^{\gamma}},
\nabla S'\right\rangle&=&\frac{1}{(S')^{\gamma}}
\langle\nabla|{\rm Sic}'|^{2},\nabla S'\rangle
-\frac{\gamma}{(S')^{\gamma+1}}
|\nabla S'|^{2}|{\rm Sic}'|^{2},\\
|S'\nabla{\rm Sic}'|^{2}&=&|Z'|^{2}-|{\rm Sic}'|^{2}|\nabla S'|^{2}
+S'\langle\nabla |{\rm Sic}'|^{2},\nabla S'\rangle,
\end{eqnarray*}
where $Z'$ is a $3$-tensor with components $Z'_{ijk}=S'\nabla_{i}S'_{jk}
-S'_{jk}\nabla_{i}S'$, we have
\begin{eqnarray*}
\Box\frac{|{\rm Sic}'|^{2}}{(S')^{\gamma}}
&=&\frac{2(\gamma-1)}{S'}\left\langle
\nabla\frac{|{\rm Sic}'|^{2}}{(S')^{\gamma}},
\nabla S'\right\rangle
-\frac{2}{(S')^{\gamma+2}}|Z'|^{2}-\frac{2\gamma|{\rm Sic}'|^{4}}{(S')^{(\gamma+1)}}\\
&&- \ \frac{(2-\gamma)(\gamma-1)}{(S')^{(\gamma+1)}}
|{\rm Sic}'|^{2}|\nabla S'|^{2}
+\frac{4}{(S')^{\gamma}}{\rm Sm}({\rm Sic}',{\rm Sic}')+\frac{4C^{2}}{n^{2}}\frac{S'}{(S')^{\gamma}}\\
&&- \ \frac{2C}{n}\frac{2(1-\gamma)S'+\gamma C}{(S')^{\gamma+1}}
|{\rm Sic}'|^{2}+\frac{4\alpha_{1}}{(S')^{\gamma}}\langle{\rm Sic}',\Xi\rangle
-\frac{2\alpha_{1}\gamma|{\rm Sic}'|^{2}}{(S')^{\gamma+1}}{\rm tr}\Xi.
\end{eqnarray*}
The identity
\begin{equation*}
\Box(S')^{2-\gamma}=(2-\gamma)(S')^{1-\gamma}\Box S'
-(2-\gamma)(1-\gamma)(S')^{-\gamma}|\nabla S'|^{2}
\end{equation*}
implies
\begin{eqnarray}
\Box f&=&2(\gamma-1)\langle \nabla f,\nabla\ln S'\rangle
-\frac{2}{(S')^{\gamma+2}}|Z'|^{2}-(2-\gamma)(\gamma-1)
|\nabla\ln S'|^{2}f\nonumber\\
&&+ \ \mathscr{D}_{1}+\mathscr{D}_{2}+\mathscr{D}_{3}.\label{3.19}
\end{eqnarray}
where
\begin{eqnarray*}
\mathscr{D}_{1}&:=&-\frac{2(2-\gamma)}{n}(S')^{1-\gamma}
|{\rm Sic}'|^{2}
+\frac{4}{(S')^{\gamma}}{\rm Sm}({\rm Sic}',{\rm Sic}')
-\frac{2\gamma|{\rm Sic}'|^{4}}{(S')^{1+\gamma}},\nonumber\\
&=&\frac{2}{(S')^{\gamma+1}}
\left[(2-\gamma)|{\rm Sic}'|^{2}|{\rm Sin}|^{2}
-2\left(|{\rm Sic}'|^{4}-S'{\rm Sm}({\rm Sic}',{\rm Sic}')\right)
\right],\\
&=&\frac{2}{(S')^{\gamma+1}}
\big[-\gamma(S')^{2\gamma}f^{2}
+\left(\frac{2n-4}{n(n-1)}-\frac{\gamma}{n}\right)(S')^{\gamma+2}
f-\frac{4(S')^{4}}{n-2}\frac{{\rm Sin}^{3}}{(S')^{3}}\\
&&+ \ 2(S')^{3}W\left(\frac{{\rm Sin}}{S'},\frac{\rm Sin}{S'}
\right)+\frac{2\alpha_{1}}{n-2}\left\langle (S')^{2}{\rm Sic}'
-\frac{n}{2}S'{\rm Sic}'{}^{2},\nabla u\otimes\nabla u\right\rangle\\
&&- \ \frac{2}{n-1}\left(\frac{C}{n}+\frac{\alpha|\nabla u|^{2}}{n-1}
\right)\left(\frac{n-1}{n}(S')^{3}-(S')^{\gamma+1}f\right)\bigg],\\
\mathscr{D}_{2}&:=&\frac{4C}{n}
\left[\frac{CS'}{n(S')^{2}}
-\frac{(1-\gamma)S'+\frac{1}{2}\gamma C}{(S')^{\gamma+1}}|{\rm Sic}'|^{2}
-\frac{2-\gamma}{2n}\frac{C-2S'}{(S')^{\gamma-1}}\right],\nonumber\\
&=&\frac{4C}{n^{2}}\left[\frac{C}{S'}
+\frac{C}{(S')^{\gamma-1}}+\frac{1}{(S')^{\gamma-2}}
+nf\left(\gamma-1-\frac{\gamma C}{2S'}\right)\right],\\
\mathscr{D}_{3}&:=&\frac{4\alpha_{1}}{(S')^{\gamma}}
\langle{\rm Sic}',\Xi\rangle
-\frac{2\alpha_{1}{\rm tr}\Xi}{(S')^{\gamma+1}}
\left(\gamma|{\rm Sic}'|^{2}+\frac{2-\gamma}{n}|S'|^{2}\right)
\end{eqnarray*}
and ${\rm Sin}^{3}={\rm Sin}_{ij}{\rm Sin}^{j}{}_{k}{\rm Sin}^{ki}$ and
$({\rm Sic}'{}^{2})_{ij}=S'_{ik}S'_{j}{}^{k}$. Some detailed computations can be found in \cite{LY2}.

In particular, for $\gamma=2$, we have from {\color{blue}{(\ref{3.19})}} that
\begin{equation}
\Box f=2\langle\nabla f,\nabla\ln S'\rangle
-2\left|\nabla\left(\frac{{\rm Sin}}{S'}\right)\right|^{2}
+\mathscr{D}_{1}+\mathscr{D}_{2}+\mathscr{D}_{3},\label{3.20}
\end{equation}
where
\begin{eqnarray*}
\mathscr{D}_{1}&=&4S'\bigg[-f^{2}-\frac{f}{n(n-1)}
-\frac{2}{n-2}\frac{{\rm Sin}^{3}}{(S')^{3}}
+\frac{1}{S'}W\left(\frac{{\rm Sin}}{S'},
\frac{{\rm Sin}}{S'}\right)\\
&&- \ \frac{1}{S'}
\left(\frac{C}{n}+\frac{\alpha|\nabla u|^{2}}{n-2}\right)
\left(\frac{1}{n}-\frac{f}{n-1}\right)
+\frac{1}{S'}\frac{\alpha}{n-2}
\left\langle\frac{{\rm Sin}^{2}}{(S')^{2}},\nabla u\otimes \nabla u
\right\rangle\\
&&+ \ \frac{1}{S'}\frac{\alpha|\nabla u|^{2}}{2n(n-2)}
\bigg],\\
\mathscr{D}_{2}&=&\frac{4C}{n^{2}}
\left[\frac{C}{S'}+\frac{C}{(S')^{3}}
+nf\left(1-\frac{C}{S'}\right)\right],\\
\mathscr{D}_{3}&=&\frac{4\alpha_{1}}{S'}
\left[\left\langle\frac{{\rm Sic}'}{S'},\Xi\right\rangle
-f{\rm tr}\Xi \right].
\end{eqnarray*}
Since the flow is regular, we have a uniform upper bound for $|\nabla u|$, together with $S'\geq C_{0}>0$, and then
\begin{eqnarray*}
\mathscr{D}_{1}&\leq&4S'
\left[-f^{2}-\frac{f}{n(n-1)}
+\frac{2}{n-2}f^{3/2}+\tilde{C}\frac{|W|}{S'}f
+\tilde{C}+\tilde{C}f\right],\\
\mathscr{D}_{2}&\leq&4S'\left(\tilde{C}+\tilde{C}f\right),\\
\mathscr{D}_{3}&\leq&4\tilde{C}S'\left(f+f^{1/2}\right)\frac{|\nabla^{2}u|^{2}}{S'}
\end{eqnarray*}
for some uniform constant $\tilde{C}$ depending only on on $n, C_{0}, C,
\alpha_{1},\beta_{1},\beta_{2}, g_{0}$, and $u_{0}$. Without loss of generality, we may assume that $f\geq1$. In this case we have
\begin{equation*}
\Box f\leq 2\langle\nabla f,\nabla\ln S'\rangle
+4S'f\left(-f+\frac{2}{n-2}f^{1/2}+\tilde{C}
+\tilde{C}\frac{|W|+|\nabla^{2}u|^{2}}{S'}\right).
\end{equation*}
Now the maximum principle yields the desired estimate.
\end{proof}

As immediate consequence, we have the following

\begin{corollary}\label{c3.3} Let $(g(t),u(t))_{t\in[0,T)}$ be a solution to the regular $(\alpha_{1},0,\beta_{1},\beta_{2})$-Ricci flow on a closed $n$-dimensional manifold $M$ with $T\leq0$ and the initial data $(g_{0},u_{0})$. Then only one of the followings cases occurs:

\begin{itemize}

\item[(a)] $T=\infty$;

\item[(b)] $T<\infty$ and $|{\rm Ric}_{g(t)}|_{g(t)}\lesssim1$;

\item[(c)] $T<\infty$ and ${\rm Ric}_{g(t)}|_{g(t)}\to\infty$ as $t\to T$. In this case, there are only two subcases:

    \begin{itemize}

    \item[(c1)] $|R_{g(t)}|_{g(t)}\to\infty$,

    \item[(c2)] $|R_{g(t)}|_{g(t)}\lesssim 1$ and there exist some uniform constants $C_{1}, C_{2}>0$ such that $S_{g(t)}+C_{1}\geq C_{2}>0$ and
        \begin{equation*}
        \frac{|W_{g(t)}|_{g(t)}+|\nabla^{2}_{g(t)}u(t)|^{2}_{g(t)}}{S_{g(t)}
        +C_{1}}\to\infty
        \end{equation*}
        as $t\to T$.

    \end{itemize}

\end{itemize}
\end{corollary}

This is a general property of the long time existence for a regular Ricci
flow, generalizing results in \cite{Cao2011, LY2}. Since the signs of $\alpha_{1},\beta_{1},\beta_{2}$ are not determined, we can not discard the case (b) which is true for Ricci-harmonic flow and Ricci flow.

\subsection{Bounded scalar
curvature}\label{subsection3.2}

We assume that $(g(t),u(t))_{t\in[0,T)}$ is a solution to the regular $(\alpha_{1},0,\beta_{1},\beta_{2})$-Ricci flow on a closed $4$-dimensional manifold $M$ with $T\leq\infty$ and the initial data $(g_{0},u_{0})$, and also assume that $S_{g(t)}+C\geq C_{0}>0$ along the flow for some uniform constants $C, C_{0}>0$. According to {\color{blue}{(\ref{3.13})}} one has
\begin{equation}
|\nabla_{g(t)} u(t)|^{2}_{g(t)}\leq A_{1}\label{3.21}
\end{equation}
along the flow.
\\

In the proof of {\color{red}{Theorem \ref{t3.2}}} we actually proved the following identity
\begin{eqnarray}
\Box\frac{|{\rm Sic}|^{2}}{S+C}
&=&-2\frac{|Z|^{2}}{(S+C)^{3}}
-2\frac{|{\rm Sic}|^{4}}{(S+C)^{2}}
+4\frac{{\rm Sm}({\rm Sic},{\rm Sic})}{S+C}\nonumber\\
&&- \ \frac{2\alpha_{1}}{(S+C)^{2}}
\left[{\rm tr}\Xi|{\rm Sic}|^{2}-2(S+C)\langle{\rm Sic},
\Xi\rangle\right],\label{3.22}
\end{eqnarray}
where $Z$ is the $3$-tensor with components
\begin{equation}
Z_{ijk}=S'\nabla_{i}S_{jk}-S_{jk}\nabla_{i}S'=(S+C)\nabla_{i}S_{jk}
-S_{jk}\nabla_{i}S,\label{3.23}
\end{equation}
and $\Xi$ is given in {\color{blue}{(\ref{3.18})}}. Observe that the last term in {\color{blue}{(\ref{3.21})}} can be written as
$$
{\rm tr}\Xi|{\rm Sic}|^{2}-2(S+C)\langle{\rm Sic},
\Xi\rangle
$$
$$
= \ \ \left[(S+C)\left|\Delta u\frac{{\rm Sic}}{\sqrt{S+C}}
-\sqrt{S+C}\nabla^{2}u\right|^{2}
-(S+C)^{2}|\nabla^{2}u|^{2}\right]
$$
$$
- \ \beta_{2}\left[(S+C)\left(|\nabla u|\frac{{\rm Sic}}{\sqrt{S+C}}
-\frac{\sqrt{S+C}}{|\nabla u|}\nabla u\otimes\nabla u\right)^{2}
-(S+C)^{2}|\nabla u|^{2}\right]
$$
$$
- \ 2\beta_{1}(S+C)
\left[\langle\nabla^{2}u,\nabla u\otimes\nabla u\rangle
\frac{|{\rm Sic}|^{2}}{S+C}
-2\left\langle{\rm Sic},\frac{1}{2}\nabla u\otimes\nabla|\nabla u|^{2}
\right\rangle\right].
$$

To further analysis, we need to know an estimate for $|\nabla^{2}u|^{2}$. Recall that
\begin{equation}
\Box|\nabla u|^{2}=2\beta_{2}|\nabla u|^{2}
-2|\nabla^{2}u|^{2}-2\alpha_{1}|\nabla u|^{4}
+4\beta_{1}\langle\nabla u\otimes\nabla u,\nabla^{2}u\rangle.\label{3.24}
\end{equation}
Given any $\epsilon>0$, we have from {\color{blue}{(\ref{3.24})}} that
\begin{eqnarray*}
\Box|\nabla u|^{2}&\leq&2\beta_{2}|\nabla u|^{2}
-2|\nabla^{2}u|^{2}-2\alpha_{1}|\nabla u|^{4}
+4|\beta_{1}|\left(\epsilon|\nabla^{2}u|^{2}
+\frac{1}{4\epsilon}|\nabla u|^{4}\right)\\
&=&2\beta_{2}|\nabla u|^{2}-2\left(1-2\epsilon|\beta_{1}|\right)|\nabla^{2}u|^{2}
-2\left(\alpha_{1}-\frac{|\beta_{1}|}{2\epsilon}\right)|\nabla u|^{4}.
\end{eqnarray*}
From the evolution equation $\partial_{t}dV_{t}=-S\!\ dV_{t}$, we arrive at
\begin{eqnarray*}
\frac{d}{dt}\int_{M}|\nabla u|^{2}dV_{t}
&=&\int_{M}\partial_{t}|\nabla u|^{2}dV_{t}+\int_{M}|\nabla u|^{2}\partial_{t}dV_{t}\\
&=&\int_{M}\Box|\nabla u|^{2}dV_{t}-\int_{M}S|\nabla u|^{2}dV_{t}\\
&\leq&2\beta_{2}\int_{M}|\nabla u|^{2}dV_{t}
-2\left(1-2\epsilon|\beta_{1}|\right)\int_{M}|\nabla^{2}u|^{2}dV_{t}\\
&&- \ \int_{M}S|\nabla u|^{2}dV_{t}-2\left(\alpha_{1}-\frac{|\beta_{1}|}{2\epsilon}\right)
\int_{M}|\nabla u|^{4}dV_{t}
\end{eqnarray*}
and then
\begin{eqnarray}
\frac{d}{dt}\int_{M}|\nabla u|^{2}dV_{t}&\leq&
-2\left(1-2\epsilon|\beta_{1}|\right)\int_{M}|\nabla^{2}u|^{2}dV_{t}\nonumber\\
&&+ \ (2|\beta_{2}|+C)\int_{M}|\nabla u|^{2}
-2\left(\alpha_{1}-\frac{|\beta_{1}|}{2\epsilon}\right)
\int_{M}|\nabla u|^{4}dV_{t},\label{3.25}
\end{eqnarray}
because of $S+C\geq C_{0}>0$.

\begin{itemize}

\item[(1)] $\beta_{1}=0$. In this case, the inequality becomes
\begin{equation*}
\frac{d}{dt}\int_{M}|\nabla u|^{2}dV_{t}\leq-2\int_{M}
|\nabla^{2}u|^{2}dV_{t}+(2|\beta_{2}|+C)\int_{M}|\nabla u|^{2}
dV_{t}-2\alpha_{1}\int_{M}|\nabla u|^{4}dV_{t}.
\end{equation*}
When $\alpha_{1}\geq0$, we furthermore have
\begin{equation*}
\frac{d}{dt}\int_{M}|\nabla u|^{2}dV_{t}
\leq-2\int_{M}|\nabla^{2}u|^{2}dV_{t}
+(2|\beta_{2}|+C)\int_{M}|\nabla u|^{2}dV_{t}
\end{equation*}
or in this form
\begin{equation*}
\frac{d}{dt}\left[e^{-(2|\beta_{2}|+C)t}
\int_{M}|\nabla u|^{2}dV_{t}\right]
\leq-2e^{-(2|\beta_{2}|+C)t}\int_{M}|\nabla^{2}u|^{2}dV_{t}.
\end{equation*}
Integrating over the interval $[0,t]$ and using {\color{blue}{(\ref{3.21})}}, we obtain
\begin{equation}
2\int^{t}_{0}\int_{M}|\nabla^{2}u|^{2}dV_{t}dt+\int_{M}
|\nabla u|^{2}dV_{t}\leq e^{(2|\beta_{2}|+C)t}
A_{1}{\rm Vol}_{0}\label{3.26}
\end{equation}
where ${\rm Vol}_{0}$ is the volume of the initial metric $g_{0}$. When $\alpha_{1}<0$, we similar have
\begin{equation*}
\frac{d}{dt}\int_{M}|\nabla u|^{2}dV_{t}\leq-2\int_{M}|\nabla^{2}u|^{2}dV_{t}
+(2|\beta_{2}|+2|\alpha_{1}|A_{1}+C)\int_{M}|\nabla u|^{2}dV_{t}.
\end{equation*}
Replacing $|\beta_{2}|$ by $|\beta_{2}|+|\alpha_{1}|A_{1}$ in {\color{blue}{(\ref{3.26})}}, the case that $\alpha_{1}$ is negative implies
\begin{equation}
2\int^{t}_{0}\int_{M}|\nabla^{2}u|^{2}dV_{t}dt+\int_{M}
|\nabla u|^{2}dV_{t}\leq e^{(2|\beta_{2}|+2|\alpha_{1}|A_{1}+C)t}
A_{1}{\rm Vol}_{0}.\label{3.27}
\end{equation}

\item[(2)] $\beta_{1}\neq0$. In this case we choose $\epsilon=1/4|\beta_{1}|$ in {\color{blue}{(\ref{3.25})}} and obtain
\begin{eqnarray*}
\frac{d}{dt}\int_{M}|\nabla u|^{2}dV_{t}&\leq&-\int_{M}|\nabla^{2}u|^{2}dV_{t}
+(2|\beta_{2}|+C)\int_{M}|\nabla u|^{2}dV_{t}\\
&&- \ 2(\alpha_{1}-2\beta^{2}_{1})\int_{M}|\nabla u|^{4}dV_{t}.
\end{eqnarray*}
A similar argument used to obtain equations {\color{blue}{(\ref{3.26})}} and {\color{blue}{(\ref{3.27})}} we get
\begin{equation}
\int^{t}_{0}\int_{M}|\nabla^{2}u|^{2}dV_{t}dt
+\int_{M}|\nabla u|^{2}dV_{t}\leq e^{(2|\beta_{2}|+2|\alpha_{1}-2\beta^{2}_{1}|A_{1}+C)t}
A_{1}{\rm Vol}_{0}.\label{3.28}
\end{equation}

\end{itemize}
Finally, from {\color{blue}{(\ref{3.27})}} and {\color{blue}{(\ref{3.28})}}, we have
\begin{equation}
\int^{t}_{0}A_{2}(t)\!\ dt\leq e^{(2|\beta_{2}|+2|\alpha_{1}-2\beta^{2}_{1}|A_{1}+C)t}
A_{1}{\rm Vol}_{0}\label{3.29}
\end{equation}
where
\begin{equation}
A_{2}(t):=\int_{M}|\nabla^{2}u|^{2}dV_{t}.\label{3.30}
\end{equation}
Introduce
\begin{equation}
\Lambda:=\frac{1}{(S+C)^{2}}
\left[{\rm tr}\!\ \Xi|{\rm Sic}|^{2}
-2(S+C)\langle{\rm Sic},\Xi\rangle\right], \ \ \ f:=\frac{|{\rm Sic}|^{2}}{S+C}\label{3.31}
\end{equation}
and rewrite {\color{blue}{(\ref{3.22})}} as
\begin{equation}
\Box f=-2\frac{|Z|^{2}}{(S+C)^{3}}
-2f^{2}+4\frac{{\rm Sm}({\rm Sic},{\rm Sic})}{S+C}
-2\alpha_{1}\Lambda.\label{3.32}
\end{equation}
To determine a lower bound for $\alpha_{1}\Lambda$ we consider the following cases.

\begin{itemize}

\item[(i)] $\alpha_{1}\geq0$. In this case we shall also find a lower bound for $\Lambda$.

    \begin{itemize}

    \item[(ia)] When $\beta_{2}\leq0$, we have
    \begin{eqnarray*}
    \Lambda&\geq&-|\nabla^{2}u|^{2}-|\beta_{2}||\nabla u|^{2}
    -2|\beta_{1}|\frac{|\nabla^{2}u||\nabla u|^{2}}{S+C}f-2|\beta_{1}|\left(f+
    \frac{|\nabla u|^{4}|\nabla^{2}u|^{2}}{S+C}\right)\\
    &\geq&-|\nabla^{2}u|^{2}-|\beta_{2}||\nabla u|^{2}
    -2|\beta_{1}|\frac{|\nabla^{2}u||\nabla u|^{2}}{C_{0}}
    f-2|\beta_{1}|\left(f+\frac{|\nabla^{2}u|^{2}|\nabla u|^{4}}{C_{0}}
    \right).
    \end{eqnarray*}

    \item[(ib)] When $\beta_{2}>0$, we get the same estimate in (ia), where, in the case, the term $-|\beta_{2}||\nabla u|^{2}$ is now replaced by
        \begin{equation*}
        -\frac{\beta_{2}}{S+C}
        \left(|\nabla u|\frac{{\rm Sic}}{\sqrt{S+C}}
        -\frac{\sqrt{S+C}}{|\nabla u|}\nabla u\otimes\nabla u\right)^{2}
        \end{equation*}
        which is bounded below by
        $$
        -\frac{2\beta_{2}}{S+C}
        \left(f|\nabla u|^{2}+\frac{S+C}{|\nabla u|^{2}}|\nabla u|^{4}\right)\geq-2\beta_{2}|\nabla u|^{2}
        \left(1+\frac{f}{C_{0}}\right).
        $$

    \end{itemize}
    From (ia) -- (ib), we obtain for any $\beta_{2}$
    \begin{eqnarray}
    \Lambda&\geq&-|\nabla^{2}u|^{2}-2|\beta_{2}||\nabla u|^{2}
    \left(1+\frac{f}{C_{0}}\right)\nonumber\\
    &&- \ 2|\beta_{1}|\frac{|\nabla^{2}u||\nabla u|^{2}}{C_{0}}f
    -2|\beta_{1}|\left(f+\frac{|\nabla^{2}u|^{2}|\nabla u|^{4}}{C_{0}}
    \right).\label{3.33}
    \end{eqnarray}

\item[(ii)] $\alpha_{1}<0$. In this case we shall find an upper bound for $\Lambda$.

    \begin{itemize}

    \item[(iia)] When $\beta_{2}>0$, we have
    \begin{eqnarray*}
    \Lambda&\leq&\frac{1}{(S+C)^{2}}
    \left[(S+C)\left|\Delta u\frac{{\rm Sic}}{\sqrt{S+C}}
    -\sqrt{S+C}\nabla^{2}u\right|^{2}+\beta_{2}(S+C)^{2}
    |\nabla u|^{2}\right]\\
    &&+ \ 2|\beta_{1}|\frac{|\nabla^{2}u||\nabla u|^{2}}{S+C}
    f+2|\beta_{1}|\left(f+\frac{|\nabla u|^{4}|\nabla^{2}u|^{2}}{S+C}\right)\\
    &=&\frac{1}{S+C}\left|\Delta u\frac{{\rm Sic}}{\sqrt{S+C}}
    -\sqrt{S+C}\nabla^{2}u\right|^{2}+\beta_{2}|\nabla u|^{2}\\
    &&+ \ 2|\beta_{1}|\frac{|\nabla^{2}u||\nabla u|^{2}}{S+C}
    f+2|\beta_{1}|\left(f+\frac{|\nabla u|^{4}|\nabla^{2}u|^{2}}{S+C}\right)\\
    &\leq&\frac{2}{S+C}\left[(\Delta u)^{2}f+(S+C)|\nabla^{2}u|^{2}\right]+\beta_{2}|\nabla u|^{2}\\
    &&+ \ 2|\beta_{1}|\frac{|\nabla^{2}u||\nabla u|^{2}}{S+C}
    f+2|\beta_{1}|\left(f+\frac{|\nabla u|^{4}|\nabla^{2}u|^{2}}{S+C}\right)\\
    &\leq&2|\nabla^{2}u|^{2}\left(1+\frac{4f}{C_{0}}\right)
    +\beta_{2}|\nabla u|^{2}\\
    &&+ \ 2|\beta_{1}|\frac{|\nabla^{2}u||\nabla u|^{2}}{S+C}
    f+2|\beta_{1}|\left(f+\frac{|\nabla u|^{4}|\nabla^{2}u|^{2}}{S+C}\right)\\
    \end{eqnarray*}

    \item[(iib)] When $\beta_{2}\leq0$, we also have
    \begin{eqnarray*}
    \Lambda&\leq&2|\nabla^{2}u|^{2}
    \left(1+\frac{4f}{C_{0}}\right)-2\beta_{2}|\nabla u|^{2}
    \left(1+\frac{f}{C_{0}}\right)\\
    &&+ \ 2|\beta_{1}|\frac{|\nabla^{2}u||\nabla u|^{2}}{S+C}
    f+2|\beta_{1}|\left(f+\frac{|\nabla u|^{4}|\nabla^{2}u|^{2}}{S+C}\right).
    \end{eqnarray*}

    \end{itemize}
From (iia) -- (iib), we obtain for any $\beta_{2}$
    \begin{eqnarray}
    \Lambda&\leq&2|\nabla^{2}u|^{2}
    \left(1+\frac{4f}{C_{0}}\right)+2|\beta_{2}||\nabla u|^{2}
    \left(1+\frac{f}{C_{0}}\right)\nonumber\\
    &&+ \ 2|\beta_{1}|\frac{|\nabla^{2}u||\nabla u|^{2}}{C_{0}}
    f+2|\beta_{1}|\left(f+\frac{|\nabla u|^{4}|\nabla^{2}u|^{2}}{C_{0}}\right).\label{3.34}
    \end{eqnarray}
\end{itemize}

\begin{lemma}\label{l3.4} If $\alpha_{1}\geq0$, one has
\begin{eqnarray}
\frac{d}{dt}\int_{M}f\!\ dV_{t}&\leq&\int_{M}
\left[-2f^{2}+4\frac{{\rm Sm}({\rm Sic},{\rm Sic})}{S+C}
-fS\right]dV_{t}\nonumber\\
&&+ \ \int_{M}4\alpha_{1}\left[|\beta_{1}|+\frac{A_{1}(|\beta_{2}|
+|\beta_{1}||\nabla^{2}u|)}{C_{0}}\right]f\!\ dV_{t}\label{3.35}\\
&&+ \ 2\alpha_{1}\left(1+\frac{2A^{2}_{1}|\beta_{1}|}{C_{0}}\right)
A_{2}+4\alpha_{1}A_{1}|\beta_{2}|{\rm Vol}_{t}.\nonumber
\end{eqnarray}
If $\alpha_{1}<0$, one has
\begin{eqnarray}
\frac{d}{dt}\int_{M}f\!\ dV_{t}&\leq&\int_{M}\left[-2f^{2}
+4\frac{{\rm Sm}({\rm Sic},{\rm Sic})}{S+C}-fS\right]dV_{t}\nonumber\\
&&- \ \int_{M}4\alpha_{1}\left[|\beta_{1}|
+\frac{A_{1}(|\beta_{2}|+|\beta_{1}||\nabla^{2}u|)+4|\nabla^{2}u|^{2}}{C_{0}}
\right]f\!\ dV_{t}\label{3.36}\\
&&- \ 2\alpha_{1}\left(2+\frac{2A^{2}_{1}|\beta_{1}|}{C_{0}}
\right)A_{2}-4\alpha_{1}A_{1}|\beta_{2}|{\rm Vol}_{t}.\nonumber
\end{eqnarray}
Here ${\rm Vol}_{t}$ denotes the volume of $g(t)$.
\end{lemma}

This follows immediately from {\color{blue}{(\ref{3.32})}} -- {\color{blue}{(\ref{3.34})}}. According to {\color{blue}{(\ref{3.35})}} and {\color{blue}{(\ref{3.36})}} we have
\begin{eqnarray}
\frac{d}{dt}\int_{M}f\!\ dV_{t}&\leq&\int_{M}\left[-2f^{2}
+4\frac{{\rm Sm}({\rm Sic},{\rm Sic})}{S+C}-fS\right]dV_{t}+\int_{M}4|\alpha_{1}|\bigg[|\beta_{1}|+\nonumber\\
&& \ \frac{A_{1}(|\beta_{2}|+|\beta_{1}||\nabla^{2}u|)+4(1-{\rm sgn}(\alpha_{1},0))|\nabla^{2}u|^{2}}{C_{0}}
\bigg]f\!\ dV_{t}\label{3.37}\\
&&+ \ 4|\alpha_{1}|\left(1+\frac{A^{2}_{1}|\beta_{1}|}{C_{0}}
\right)A_{2}+4|\alpha_{1}|A_{1}|\beta_{2}|{\rm Vol}_{t},\nonumber
\end{eqnarray}
where ${\rm sgn}(\alpha_{1},0)=1$ if $\alpha_{1}\geq0$, and otherwise $0$. For $\alpha_{1}\geq0$, the above estimates shall imply integrals bounds for $|{\rm Sic}|, |{\rm Sm}|$ as in \cite{LY2}. On the other hand, the case $\alpha_{1}<0$ will prevent us to obtain the previous-type of estimates, since we have no control on $|\nabla^{2}u|^{2}$ even in the integral sense. Hence in the
case that $\alpha_{1}$ is negative, we will impose another
condition\footnote{For the Ricci-harmonic flow or the $(\alpha,0,0,0)$-Ricci
flow, the estimate (\ref{3.38}) is always true provided that the curvature
condition (\ref{1.4}) holds.}:
\begin{equation}
|\nabla^{2}_{g(t)}u(t)|^{2}_{g(t)}
\leq \widetilde{A}_{1}\label{3.38}
\end{equation}
along the flow for some uniform constant $\widetilde{A}_{1}$. Note that the condition {\color{blue}{(\ref{3.38})}} is stronger than {\color{blue}{(\ref{3.29})}}. A weakened condition is
\begin{equation}
||\nabla^{2}_{g(t)}u(t)||_{L^{4}(M,g(t))}\leq\widehat{A}_{1}\label{3.39}
\end{equation}
along the flow for some uniform constant $\widehat{A}_{1}$.
\\

In the the case of dimension $4$, we have the following Gauss-Bonnet-Chern formula
\begin{equation}
32\pi^{2}\chi(M)=\int_{M}\left[|{\rm Rm}|^{2}-4|{\rm Ric}|^{2}
+R^{2}\right]dV_{g}\label{3.40}
\end{equation}
for any Riemannian metric $g$ on $M$, where $\chi(M)$ is the Euler characteristic number of $M$. Applying the formula {\color{blue}{(\ref{3.40})}} to $g(t)$ and noting that (see Lemma 3.1 in \cite{LY2})
\begin{eqnarray*}
|{\rm Rm}|^{2}-4|{\rm Ric}|^{2}+R^{2}
&=&|{\rm Sm}|^{2}-4|{\rm Sic}|^{2}+S^{2}\\
&&- \ \frac{13}{2}\alpha^{2}_{1}|\nabla u|^{4}
-9\alpha_{1}{\rm Sic}(\nabla u,\nabla u)
+2\alpha_{1}S|\nabla u|^{2}
\end{eqnarray*}
we have (see (3.12) in \cite{LY2})
\begin{eqnarray}
\int_{M}\left[|{\rm Sm}|^{2}-4|{\rm Sic}|^{2}
+S^{2}\right]dV_{t} \ \ = \ \ 32\pi^{2}\chi(M)
+ \frac{13}{2}\alpha^{2}_{1}\int_{M}|\nabla u|^{4}dV_{t}\nonumber\\
+ \ 9\alpha_{1}\int_{M}{\rm Sic}(\nabla u,\nabla u)\!\ dV_{t}
-2\alpha_{1}\int_{M}S|\nabla u|^{2}dV_{t}.\label{3.41}
\end{eqnarray}
Moreover we also have (see (3.15) in \cite{LY2})
\begin{eqnarray}
\int_{M}\left[-2f^{2}+4\frac{{\rm Sm}({\rm Sic},{\rm Sic})}{S+C}
-fS\right]dV_{t} \ \ \leq \ \
\int_{M}\left(-f^{2}+36C f+574 S^{2}\right)dV_{t}\nonumber\\
+ \ 8\left[32\pi^{2}\chi(M)
+13\alpha^{2}_{1}A^{2}_{1}{\rm Vol}_{0} e^{(2|\beta_{2}|
+2|\alpha_{1}-2\beta^{2}_{1}|A_{1}+C)t}\right]\label{3.42}
\end{eqnarray}
where we used {\color{blue}{(\ref{3.28})}} to control the integral
\begin{equation*}
\int_{M}|\nabla u|^{4}dV_{t}.
\end{equation*}
Plugging {\color{blue}{(\ref{3.42})}} into {\color{blue}{(\ref{3.37})}} and using {\color{blue}{(\ref{2.5})}} yields
\begin{eqnarray}
\frac{d}{dt}\int_{M}f\!\ dV_{t}&\leq&\int_{M}\left(-\frac{1}{2}f^{2}+C_{1}f
+C_{2}S^{2}\right)dV_{t}\nonumber\\
&&+ \ C_{3}A_{2}+C_{4}\left[1-{\rm sgn}(\alpha_{1},0)\right]
\int_{M}|\nabla^{2}u|^{4}dV_{t}+C_{5}e^{C_{6}t}+C_{7}.\label{3.43}
\end{eqnarray}
where
\begin{eqnarray}
C_{1}&=&36C+4|\alpha_{1}||\beta_{1}|+\frac{4|\alpha_{1}|A_{1}|\beta_{2}|}{
C_{0}} \ \ = \ \ C_{1}(C,C_{0},\alpha_{1},\beta_{1},\beta_{2},A_{1}),\nonumber\\
C_{2}&=&574,\nonumber\\
C_{3}&=&\frac{16|\alpha_{1}|^{2}A^{2}_{1}|\beta_{1}|^{2}}{C^{2}_{0}}
+4|\alpha_{1}|\left(1+\frac{A^{2}_{1}|\beta_{1}|}{C_{0}}\right) \ \ = \ \ C_{3}(C_{0},\alpha_{1},\beta_{1},A_{1}),\nonumber\\
C_{4}&=&\frac{256|\alpha_{1}|^{2}}{C^{2}_{0}} \ \ = \ \ C_{4}(C_{0},\alpha_{1}),\label{3.44}\\
C_{5}&=&(104\alpha^{2}_{1}A^{2}_{1}+4|\alpha_{1}|A_{1}
|\beta_{2}|){\rm Vol}_{0} \ \ = \ \ C_{5}(\alpha_{1},\beta_{2},A_{1},{\rm Vol}_{0}),\nonumber\\
C_{6}&=&2|\beta_{2}|+2|\alpha_{1}
-2\beta^{2}_{1}||A_{1}|+C \ \ = \ \ C_{6}(C,\alpha_{1},\beta_{1},\beta_{2},A_{1}),\nonumber\\
C_{7}&=&256\pi^{2}\chi(M).\nonumber
\end{eqnarray}
Note that $C_{1}, C_{6}$ are linear functions of $A_{1}$ and $C_{3}, C_{5}$ are quadratic functions of $A_{1}$. Furthermore, $C_{2}, C_{7}$ are constants depending only on the topological quantities of $M$. Finally, $C_{4}$ depends only on $\alpha_{1}$, and the term containing $C_{4}$ in {\color{blue}{(\ref{3.43})}} vanishes provided that $\alpha_{1}\geq0$.

\begin{theorem}\label{t3.5} Let $(g(t),u(t))_{t\in[0,T)}$ be a solution to the regular $(\alpha_{1},0,\beta_{1},\beta_{2})$-Ricci flow on a closed $4$-manifold $M$ with $T\leq\infty$ and the initial data $(g_{0},u_{0})$. Assume that $S_{g(t)}+C\geq C_{0}>0$ along the flow for some uniform constants $C, C_{0}>0$. Then
\begin{eqnarray}
&&\int_{M}\frac{|{\rm Sic}_{g(s)}|^{2}_{g(s)}}{S_{g(s)}+C}dV_{g(s)}
+\int^{s}_{0}\int_{M}\frac{|{\rm Sic}_{g(t)}|^{4}_{g(t)}}{(S_{g(t)}
+C)^{2}}dV_{g(t)}dt\nonumber\\
&\leq&C'(1+s) e^{C's}+C'e^{C's}\int^{s}_{0}\int_{M}S^{2}dV_{t}dt\label{3.45}\\
&&+ \ C'[1-{\rm sgn}(\alpha_{1},0)]e^{C's}\int^{s}_{0}\int_{M}|\nabla^{2}u|^{4}dV_{t}dt,
\nonumber
\end{eqnarray}

\begin{eqnarray}
\int_{M}|{\rm Sic}_{g(s)}|_{g(s)}dV_{g(s)}&\leq&C'(1+s) e^{C's}+C'e^{C's}\int^{s}_{0}\int_{M}S^{2}dV_{t}dt
\nonumber\\
&&+ \ C'[1-{\rm sgn}(\alpha_{1},0)]e^{C's}\int^{s}_{0}\int_{M}|\nabla^{2}u|^{4}dV_{t}dt,
\label{3.46}
\end{eqnarray}

\begin{eqnarray}
\int^{s}_{0}\int_{M}|{\rm Sic}_{g(t)}|^{2}_{g(t)}
dV_{g(t)}dt&\leq&C'(1+s) e^{C's}+C'e^{C's}\int^{s}_{0}\int_{M}S^{2}dV_{t}dt
\nonumber\\
&&+ \ C'[1-{\rm sgn}(\alpha_{1},0)]e^{C's}\int^{s}_{0}\int_{M}|\nabla^{2}u|^{4}dV_{t}dt,
\label{3.47}
\end{eqnarray}

\begin{eqnarray}
\int^{s}_{0}\int_{M}|{\rm Sm}_{g(t)}|^{2}_{g(t)}dV_{g(t)}dt&\leq&C'(1+s) e^{C's}+C'e^{C's}\int^{s}_{0}\int_{M}S^{2}dV_{t}dt
\nonumber\\
&&+ \ C'[1-{\rm sgn}(\alpha_{1},0)]e^{C's}\int^{s}_{0}\int_{M}|\nabla^{2}u|^{4}dV_{t}dt,
\label{3.48}
\end{eqnarray}
for all $s\in[0,T)$, where $C'=C'(g_{0},u_{0},\alpha_{1},
\beta_{1},\beta_{2},C,C_{0},A_{1},\chi(M))$ is a uniform constant. Here $|\nabla_{g(t)}u(t)|^{2}_{g(t)}\leq A_{1}$ holds along the flow (by the regularity) for some uniform constant $A_{1}>0$ (which depends only on $g_{0},u_{0}$ and $\alpha_{0},\beta_{1},\beta_{2}$).
\end{theorem}

\begin{proof} Integrating {\color{blue}{(\ref{3.43})}} over $[0,s]$ we obtain
$$
e^{-C_{1}t}\int_{M}f\!\ dV_{t}+\frac{1}{2}e^{-C_{1}s}
\int^{s}_{0}\int_{M}f^{2}dV_{t}dt
$$
$$
\leq \ \ \int^{s}_{0}
\left[C_{3}A_{2}e^{-C_{1}t}+C_{7}e^{-C_{1}t}
+C_{5}e^{(C_{6}-C_{1})t}\right]dt+C_{2}\int^{s}_{0}e^{-C_{1}t}\int_{M}S^{2}dV_{t}dt
$$
$$
+ \ C_{4}\int^{s}_{0}
e^{-C_{1}t}[1-{\rm sgn}(\alpha_{1},0)]
\int_{M}|\nabla^{2}u|^{4}dV_{t}dt
+\int_{M}fdV_{t}\bigg|_{t=0}.
$$
Using {\color{blue}{(\ref{3.29})}}, the above inequality becomes
$$
\int_{M}f\!\ dV_{s}+\int^{s}_{0}\int_{M}f^{2}dV_{t}dt \ \ \leq \ \
2e^{C_{1}s}\left(\int_{M}f\!\ dV_{t}\bigg|_{t=0}\right)
$$
$$
+ \ 2C_{3}e^{(C_{1}+C_{6})s}
+\frac{2C_{7}}{C_{1}}e^{C_{1}s}+2C_{5}\cdot\left\{\begin{array}{cc}
s e^{C_{1}s}, & C_{1}=C_{6},\\
e^{(C_{1}+C_{6})s}/|C_{1}-C_{6}|, & C_{1}\neq C_{6},
\end{array}\right.
$$
$$
+ \ 2C_{2}e^{C_{1}s}
\int^{s}_{0}\int_{M}S^{2}dV_{t}dt
+2C_{5}[1-{\rm sgn}(\alpha_{1},0)]e^{C_{1}s}
\int^{s}_{0}\int_{M}e^{-C_{1}t}|\nabla^{2}u|^{4}dV_{t}dt.
$$
$$
\leq \ \ C_{8}(1+s)e^{(C_{1}+C_{6})s}
+2C_{2}e^{C_{1}s}\int^{s}_{0}\int_{M}S^{2}dV_{t}dt
$$
$$
+ \ 2C_{5}[1-{\rm sgn}(\alpha_{1},0)]e^{C_{1}s}\int^{s}_{0}
\int_{M}e^{-C_{1}t}|\nabla^{2}u|^{4}dV_{t}dt
$$
for some uniform constant $C_{8}$ which depends only on $g_{0}, u_{0}, \alpha_{1}, \beta_{1},\beta_{2}, C, C_{0}$, and $A_{1}, \chi(M)$.

The second and third estimates follows from (see (3.22) and below in \cite{LY2})
\begin{equation*}
|{\rm Sic}|\leq 2\frac{|{\rm Sic}|^{2}}{S+C}+\frac{C}{2}, \ \ \
|{\rm Sic}|^{2}\leq 8\frac{|{\rm Sic}|^{4}}{(S+C)^{2}}+\frac{C^{2}}{4},
\end{equation*}
and ${\rm Vol}_{t}\leq e^{Ct}{\rm Vol}_{0}$.

The last estimate follows from {\color{blue}{(\ref{3.41})}}, as the argument in \cite{LY2}.
\end{proof}

Corresponding to case $(c2)$ in {\color{red}{Corollary \ref{c3.3}}}, we obtain

\begin{corollary}\label{c3.6} Let $(g(t),u(t))_{t\in[0,T)}$ be a solution to the regular $(\alpha_{1},0,\beta_{1},\beta_{2})$-Ricci flow on a closed $4$-manifold $M$ with $T\leq\infty$ and the initial data $(g_{0},u_{0})$. Assume that $S_{g(t)}+C\geq C_{0}>0$ and $S^{2}_{g(t)}
\leq C_{1}<\infty$ along the flow for some uniform constants $C, C_{0},C_{1}>0$. Then
\begin{eqnarray}
\int^{s}_{0}\int_{M}|{\rm Sic}_{g(t)}|^{2}_{g(t)}
dV_{g(t)}dt&\leq&C'(1+s) e^{C's}\nonumber\\
&&+ \ C'[1-{\rm sgn}(\alpha_{1},0)]e^{C's}\int^{s}_{0}\int_{M}|\nabla^{2}u|^{4}dV_{t}dt,
\label{3.49}
\end{eqnarray}

\begin{eqnarray}
\int^{s}_{0}\int_{M}|{\rm Sm}_{g(t)}|^{2}_{g(t)}dV_{g(t)}dt&\leq&C'(1+s) e^{C's}\nonumber\\
&&+ \ C'[1-{\rm sgn}(\alpha_{1},0)]e^{C's}\int^{s}_{0}\int_{M}|\nabla^{2}u|^{4}dV_{t}dt,
\label{3.50}
\end{eqnarray}
for all $s\in[0,T)$, where $C'=C'(g_{0},u_{0},\alpha_{1},
\beta_{1},\beta_{2},C,C_{0},C_{1}, A_{1},\chi(M))$ is a uniform constant. Here $|\nabla_{g(t)}u(t)|^{2}_{g(t)}\leq A_{1}$ holds along the flow (by the regularity) for some uniform constant $A_{1}>0$ (which depends only on $g_{0},u_{0}$ and $\alpha_{0},\beta_{1},\beta_{2}$).
\end{corollary}

To get the $L^{1}_{[0,T)}L^{p}(M)$-estimate of ${\rm Sic}_{g(t)}$, introduce the basic assumption {\bf BA} for a solution $(g(t),u(t))_{t\in[0,T)}$ to the regular $(\alpha_{1},0,\beta_{1},\beta_{2})$-Ricci flow:

\begin{itemize}

\item[(a)] $M$ is a closed $4$-manifold;

\item[(b)] $T<\infty$;

\item[(c)] $-1\leq S_{g(t)}\leq 1$ along the flow;

\item[(d)] $|\nabla_{g(t)}u(t)|^{2}_{g(t)}\leq A_{1}$ along the flow.

\end{itemize}
The last condition is obtained from the regularity of the flow and the third condition implies $S_{g(t)}+C\geq C_{0}>0$, where $C=2$ and $C_{0}=1$.

Under {\bf BA}, given $\epsilon>0$, one has (see the proof of Theorem 3.4 in \cite{LY2})
$$
\int_{M}\left[-2f^{2}+4\frac{{\rm Sm}({\rm Sic},{\rm Sic})}{S+2}
-fS\right]dV_{t}
$$
$$
\leq \ \ \int_{M}\left[-\left(2-\frac{4}{\epsilon^{2}}\right)
f^{2}+(12\epsilon^{2}+1)f+\epsilon^{2}
\left(|{\rm Sm}|^{2}-4|{\rm Sic}|^{2}
+S^{2}\right)\right]dV_{t}
$$
where $f=|{\rm Sic}|^{2}/(S+2)$. In this case, $C_{0}=1$, so that {\color{blue}{(\ref{3.37})}} becomes
$$
\frac{d}{dt}\int_{M}f\!\ dV_{t} \ \ \leq \ \
\int_{M}\left[-2f^{2}+4\frac{{\rm Sm}({\rm Sic},{\rm Sic})}{S+2}
-fS\right]dV_{t}
$$
$$
+ \ \int_{M}4|\alpha_{1}|
\left[|\beta_{1}|
+A_{1}(|\beta_{2}|+|\beta_{1}||\nabla^{2}u|)
+4(1-{\rm sgn}(\alpha_{1},0))|\nabla^{2}u|^{2}\right]f\!\ dV_{t}
$$
$$
+ \ 4|\alpha_{1}|(1+A^{2}_{1}|\beta_{1}|)
+4|\alpha_{1}|A_{1}|\beta_{2}|e^{2t}{\rm Vol}_{0}.
$$
$$
\leq \ \ \int_{M}\left[-\left(2-\frac{12}{\epsilon^{2}}
\right)f^{2}+\left(12\epsilon^{2}
+1+4|\alpha_{1}\beta_{1}|+4|\alpha_{1}\beta_{2}|A_{1}\right)f\right]dV_{t}
$$
$$
+ \ \epsilon^{2}\int_{M}\left(|{\rm Sm}|^{2}-4|{\rm Sic}|^{2}
+S^{2}\right)dV_{t}+4|\alpha_{1}|(1+A^{2}_{1}|\beta_{1}|)A_{2}
+4|\alpha_{1}|A_{1}|\beta_{2}|e^{2t}{\rm Vol}_{0}
$$
$$
+ \ \epsilon^{2}
\int_{M}\left[(\alpha_{1}\beta_{1})^{2}A^{2}_{1}
|\nabla^{2}u|^{2}+256|\alpha_{1}|^{2}(1-{\rm sgn}(\alpha_{1},0))|\nabla^{2}u|^{4}\right]dV_{t}.
$$
On the other hand, the identity {\color{blue}{(\ref{3.41})}} implies (see (3.13) in \cite{LY2})
\begin{equation*}
\int_{M}\left[|{\rm Sm}|^{2}
-4|{\rm Sic}|^{2}+S^{2}\right]dV_{t}
\leq32\pi^{5}\chi(M)
+13\alpha^{2}_{1}\int_{M}|\nabla u|^{4}dV_{t}
+\frac{243}{26}\int_{M}f\!\ dV_{t}.
\end{equation*}
Therefore, taking $\epsilon^{2}=12$ and using {\color{blue}{(\ref{3.28})}},
\begin{eqnarray}
\frac{d}{dt}\int_{M}f\!\ dV_{t}&\leq&\int_{M}
\left(-f^{2}+\widetilde{C}_{1}f\right)dV_{t}+
\widetilde{C}_{2}A_{2}+\widetilde{C}_{3}+\widetilde{C}_{4} e^{\widetilde{C}_{5}t}\nonumber\\
&&+ \ \widetilde{C}_{6}[1-{\rm sgn}(\alpha_{1},0)]\int_{M}|\nabla^{2}u|^{4}dV_{t},\label{3.51}
\end{eqnarray}
where
\begin{eqnarray*}
\widetilde{C}_{1}&=&145+4|\alpha_{1}\beta_{1}|+4|\alpha_{1}\beta_{2}|A_{1}
+\frac{1458}{13} \ \ = \ \ \widetilde{C}_{1}(\alpha_{1},\beta_{1},\beta_{2},A_{1}),\\
\widetilde{C}_{2}&=&4|\alpha_{1}|(1+A^{2}_{1}|\beta_{1}|)
+12(\alpha_{1}\beta_{1})^{2}A^{2}_{1} \ \ = \ \ \widetilde{C}_{2}(\alpha_{1},\beta_{1},A_{1}),\\
\widetilde{C}_{3}&=&385\pi^{5}\chi(M),\\
\widetilde{C}_{4}&=&\left(156\alpha^{2}_{1}A_{1}
+4|\alpha_{1}|A_{1}|\beta_{2}|\right){\rm Vol}_{0} \ \ = \ \
\widetilde{C}_{4}(\alpha_{1},\beta_{2},A_{1},{\rm Vol}_{0}),\\
\widetilde{C}_{5}&=&2|\beta_{2}|+2|\alpha_{1}-2\beta^{2}_{1}|A_{1}+2 \ \ = \ \ \widetilde{C}_{5}(\alpha_{1},\beta_{1},\beta_{2},A_{1}),\\
\widetilde{C}_{6}&=&3072|\alpha_{1}|^{2} \ \ = \ \ \widetilde{C}_{6}(\alpha_{1}).
\end{eqnarray*}

\begin{theorem}\label{t3.7} Suppose that $(g(t),u(t))_{t\in[0,T)}$ satisfies {\bf BA}. Then
\begin{eqnarray}
\int_{M}|{\rm Sic}_{g(s)}|^{2}_{g(s)}dV_{g(s)}
&\leq&\widetilde{C}(1+s)e^{\widetilde{C}s}\nonumber\\
&&+ \ \widetilde{C}[1-{\rm sgn}(\alpha_{1},0)]e^{\widetilde{C}s}
\int^{s}_{0}\int_{M}|\nabla^{2}_{g(t)}u(t)|^{4}dV_{g(t)}dt,\label{3.52}\\
\int_{M}|{\rm Sm}_{g(s)}|^{2}_{g(s)}dV_{g(s)}
&\leq&\widetilde{C}(1+s)e^{\widetilde{C}s}\nonumber\\
&&+ \ \widetilde{C}[1-{\rm sgn}(\alpha_{1},0)]e^{\widetilde{C}s}
\int^{s}_{0}\int_{M}|\nabla^{2}_{g(t)}u(t)|^{4}dV_{g(t)}dt,\label{3.53}\\
\int^{s}_{0}\int_{M}|{\rm Sic}_{g(t)}|^{4}_{g(t)}
dV_{g(t)}dt&\leq&\widetilde{C}(1+s)e^{\widetilde{C}s}\nonumber\\
&&+ \ \widetilde{C}[1-{\rm sgn}(\alpha_{1},0)]e^{\widetilde{C}s}
\int^{s}_{0}\int_{M}|\nabla^{2}_{g(t)}u(t)|^{4}dV_{g(t)}dt,\label{3.54}\\
\int^{T}_{s}\int_{M}|{\rm Sic}_{g(t)}|^{p}_{g(t)}
dV_{g(t)}dt&\leq&\left[(T-s)e^{T}{\rm Vol}_{0}\right]^{\frac{4-p}{4}}
e^{\widetilde{C}T}\bigg[\widetilde{C}(1+T)\nonumber\\
&&+ \ \widetilde{C}[1-{\rm sgn}(\alpha_{1},0)]
\int^{T}_{0}\int_{M}|\nabla^{2}_{g(t)}u(t)|^{4}dV_{g(t)}dt\bigg]^{\frac{p}{4}}
\end{eqnarray}
for any $s\in[0,T)$ and $0<p<4$. Here $\widetilde{C}$ is a uniform constant which depends only on $g_{0},u_{0}, \alpha_{1},\beta_{1},\beta_{2}, A_{1}, \chi(M)$.
\end{theorem}

\begin{proof} From {\color{blue}{(\ref{3.51})}} and {\color{blue}{(\ref{3.29})}}, one has
$$
e^{-\widetilde{C}_{1}s}\int_{M}f\!\ dV_{s}
+e^{-\widetilde{C}_{1}s}\int^{s}_{0}\int_{M}
f^{2}dV_{t}dt \ \ \leq \ \ \int_{M}f\!\ dV_{t}\bigg|_{t=0}
$$
$$
+ \ \widetilde{C}_{2}e^{\widetilde{C}_{5}s}+\frac{\widetilde{C}_{3}}{\widetilde{C}_{1}}
\left(1-e^{-\widetilde{C}_{1}s}\right)
+\widetilde{C}_{4}\cdot\left\{\begin{array}{cc}
s, & \widetilde{C}_{1}=\widetilde{C}_{5},\\
e^{(\widetilde{C}_{1}+\widetilde{C}_{5})s}/|\widetilde{C}_{1}
+\widetilde{C}_{5}|, & \widetilde{C}_{1}\neq\widetilde{C}_{5},
\end{array}\right.
$$
$$
+ \ \widetilde{C}_{6}[1-{\rm sgn}(\alpha_{1},0)]
\int^{s}_{0}\int_{M}e^{-\widetilde{C}_{1}t}
|\nabla^{2}u|^{4}dV_{t}dt
$$
and hence
$$
\int_{M}f\!\ dV_{s}+\int^{s}_{0}\int_{M}
f^{2}dV_{t}dt \ \ \leq \ \
\widetilde{C}_{7}(1+s)e^{(\widetilde{C}_{1}+\widetilde{C}_{5})s}
$$
$$
+ \ \widetilde{C}_{6}[1-{\rm sgn}(\alpha_{1},0)]e^{\widetilde{C}_{1}s}
\int^{s}_{0}\int_{M}e^{-\widetilde{C}_{1}t}|\nabla^{2}u|^{4}dV_{t}dt.
$$
for some uniform constant $\widetilde{C}_{7}$ which depends only on $g_{0},u_{0}, \alpha_{1},\beta_{1},\beta_{2}, A_{1}, \chi(M)$.

Because $|S\leq 1$, we have $\frac{1}{3}|{\rm Sic}|^{2}\leq f\leq |{\rm Sic}|^{2}$. The above estimate immediately implies {\color{blue}{(\ref{3.52})}} and {\color{blue}{(\ref{3.54})}}. The second estimate follows from {\color{blue}{(\ref{3.41})}} as the argument in \cite{LY2} (see the proof of Theorem 3.5). The last estimate follows from ${\rm Vol}_{t}\leq e^{t}{\rm Vol}_{0}$ and the H\"older inequality.
\end{proof}

According to our definition of ${\rm Sic}, {\rm Sm}$ in {\color{blue}{(\ref{3.10})}}, we see that the boundedness of ${\rm Sic}, {\rm Sm}$ is equivalent to the boundedness of ${\rm Sic}, {\rm Rm}$
for any regular Ricci flow.

\begin{corollary}\label{c3.8} Suppose that $(g(t),u(t))_{t\in[0,T)}$ satisfies {\bf BA}. Then
\begin{eqnarray}
\int_{M}|{\rm Ric}_{g(s)}|^{2}_{g(s)}dV_{g(s)}
&\leq&\widetilde{C}(1+s)e^{\widetilde{C}s}\nonumber\\
&&+ \ \widetilde{C}[1-{\rm sgn}(\alpha_{1},0)]e^{\widetilde{C}s}
\int^{s}_{0}\int_{M}|\nabla^{2}_{g(t)}u(t)|^{4}dV_{g(t)}dt,\label{3.56}\\
\int_{M}|{\rm Rm}_{g(s)}|^{2}_{g(s)}dV_{g(s)}
&\leq&\widetilde{C}(1+s)e^{\widetilde{C}s}\nonumber\\
&&+ \ \widetilde{C}[1-{\rm sgn}(\alpha_{1},0)]e^{\widetilde{C}s}
\int^{s}_{0}\int_{M}|\nabla^{2}_{g(t)}u(t)|^{4}dV_{g(t)}dt,\label{3.57}
\end{eqnarray}
for any $s\in[0,T)$. Here $\widetilde{C}$ is a uniform constant which depends only on $g_{0},u_{0}, \alpha_{1},\beta_{1},\beta_{2}$, $A_{1}, \chi(M)$.
\end{corollary}

\section{Bounded $L^{2}$-curvature conjecture for the Einstein scalar field equations}
\label{section4}

From {\color{red}{Theorem \ref{t2.7}}}, we can get an upper bound for the $L^{2}$-norm of ${\rm Rm}_{g(t)}$. Motivated by this estimate, we in this section impose a conjecture for the Einstein scalar field equations, which is analogous to the corresponding conjecture for the Einstein vacuum equations proved by Klainerman, Rodnianski, and Szeftel \cite{KRS2015, S1, S2, S3, S4, S5}.

\subsection{Initial value problem}\label{subsection4.1}

In this section we recall some basic results for Einstein scalar field equations from \cite{R2009}. Consider Einstein's equation
\begin{equation}
\boldsymbol{R}_{\alpha\beta}-\frac{1}{2}\boldsymbol{R}{\bf g}_{\alpha\beta}=\boldsymbol{T}_{\alpha\beta}\label{4.1}
\end{equation}
where $\boldsymbol{R}_{\alpha\beta}$ and $\boldsymbol{R}$denote, respectively, the Ricci curvature tensor and scalar curvature of a four dimensional Lorentzian space-time $({\bf M},
{\bf g})$. If the energy-momentum tensor $\boldsymbol{T}_{\alpha\beta}$ is chosen as
\begin{equation}
\boldsymbol{T}_{\alpha\beta}
=2\partial_{\alpha}u\partial_{\beta}u
-\frac{1}{2}|{\bf D}u|^{2}{\bf g},\label{4.2}
\end{equation}
where ${\bf D}$ is the Levi-Civita connection of ${\bf g}$ and $u$ is a smooth function on ${\bf M}$. In this case, the Einstein equation {\color{blue}{(\ref{4.1})}} can be written as
\begin{equation*}
\boldsymbol{R}_{\alpha\beta}-2\partial_{\alpha}u
\partial_{\beta}u=0.
\end{equation*}
As discussed in \cite{R2009}, we should impose a matter equation
\begin{equation*}
\boldsymbol{\Delta}u=0
\end{equation*}
for $u$, where $\boldsymbol{\Delta}:={\bf D}^{\alpha}{\bf D}_{\alpha}$. Hence we should consider a system of PDEs
\begin{equation}
\boldsymbol{R}_{\alpha\beta}
-2\partial_{\alpha}u\partial_{\beta}u=0, \ \ \
\boldsymbol{\Delta}u=0,\label{4.3}
\end{equation}
which is called the Einstein scalar field equation or the Einstein-Klein-Gordon
equation.
\\

An {\it initial data set} $(\Sigma,g,k,u_{0},u_{1})$ for {\color{blue}{(\ref{4.3})}} consists of a three dimensional manifold $\Sigma$, a Riemannian metric $g$, a symmetric $2$-tensor $k$, together with two functions $u_{0}$ and $u_{1}$ on $\Sigma$, all assumed to be smooth, verifying the constraint equations,
\begin{eqnarray}
\nabla^{j}k_{ij}-\nabla_{i}{\rm tr}\!\ k&=&u_{1}\nabla_{i}u_{0},\label{4.4}\\
R-|k|^{2}+({\rm tr}\!\ k)^{2}&=&u^{2}_{1}+|\nabla u_{0}|^{2},\label{4.5}
\end{eqnarray}
where $\nabla$ is the Levi-Civita connection of $g$.
\\

Given an initial data set $(\Sigma,g,k,u_{0},u_{1})$, the {\it Cauchy problem}
consists in finding a four-dimensional Lorentzian manifold $({\bf M},
{\bf g})$ and a smooth function $u$ on ${\bf M}$ satisfying {\color{blue}{(\ref{4.3})}}, and also an embedding $\iota: \Sigma\to{\bf M}$ such that
\begin{equation}
\iota^{\ast}{\bf g}=g, \ \ \ \iota^{\ast}u=u_{0}, \ \ \
\iota^{\ast}K=k,\ \ \ \iota^{\ast}(\boldsymbol{N}u)=u_{1},\label{4.6}
\end{equation}
where $\boldsymbol{N}$ is the future-directed unit normal to $\iota(\Sigma)$ and $K$ is the second fundamental form of $\iota(\Sigma)$.

The local existence and uniqueness result for globally hyperbolic developments can be found in \cite{R2009}, Theorem 14.2. For stability and
instability for Einstein's scalar field equation, we refer to
\cite{D2003, DW2018, LM2016, LM2017, V2018, Wang2016, Wang, WJH2018}.

\subsection{Bounded $L^{2}$-curvature conjecture for Einstein's equations}
\label{subsection4.2}

For Einstein's equations (i.e., $u=0$ in {\color{blue}{(\ref{4.3})}}, and the initial data is denoted by $(\Sigma,g,k)$), Klainerman \cite{K1999} proposed the following conjecture:
\begin{quote}
{\it The Einstein vacuum equations admits local Cauchy developments for initial data sets $(\Sigma,g,k)$ with locally finite $L^{2}$-curvature and locally finite $L^{2}$-norm of the first covariant derivatives of $k$.}
\end{quote}

This conjecture was recently solved by Klainerman, Rodnianski and
Szeftel \cite{KRS2015}. To give a precise result, we assume that the space-time $({\bf M},{\bf g})$ to be foliated by the level surfaces $\Sigma_{t}$ $=\mathfrak{t}^{-1}(t)$ of a time function $\mathfrak{t}$. Let $\boldsymbol{T}$ denote the unit normal to $\Sigma_{t}$, and let $k$ the second fundamental form of $\Sigma_{t}$, i.e., $k_{ij}:=-{\bf g}({\bf D}_{i}\boldsymbol{T},e_{j})$, where $(e_{i})_{1\leq i\leq 3}$ denote an arbitrary frame on $\Sigma_{t}$. We also assume that the $\Sigma_{t}$-foliation is maximal, i.e., we have
\begin{equation*}
{\rm tr}_{g}k=0
\end{equation*}
where $g=g(t)$ is the induced metric on $\Sigma_{t}$.

\begin{theorem}\label{t4.1}{\bf (Klainerman-Rodnianski-Szeftel, 2015)} Let $({\bf M},{\bf g})$ an asymptotically flat solution to the Einstein vacuum equations together with a maximal foliation by space-like hyper-surfaces $\Sigma_{t}$ defined as level hyper-surfaces of a time function $\mathfrak{t}$. Assume that the initial slice $(\Sigma,g,k)$ is such that the Ricci curvature ${\rm Ric}\in L^{2}(\Sigma)$, $\nabla k\in L^{2}(\Sigma)$, and $\Sigma$ has a strictly positive volume radius on scales $\leq 1$, i.e.,
\begin{equation}
r_{{\rm vol}}(\Sigma,1):=\inf_{p\in\Sigma}
\inf_{r\in(0,1]}\frac{{\rm vol}_{g}(B_{g}(p,r))}{r^{3}}>0.\label{4.7}
\end{equation}
Then there exists a time
\begin{equation*}
T:=T\left(||{\rm Ric}||_{L^{2}(\Sigma)},||\nabla k||_{L^{2}(
\Sigma)}, r_{{\rm vol}}(\Sigma,1)\right)>0
\end{equation*}
and a constant
\begin{equation*}
C:=C\left(||{\rm Ric}||_{L^{2}(\Sigma)},||\nabla k||_{L^{2}(
\Sigma)},r_{{\rm vol}}(\Sigma,1)\right)>0
\end{equation*}
such that the following control
\begin{equation*}
||{\bf Rm}||_{L^{\infty}_{[0,T]}L^{2}(\Sigma_{t})}
\leq C, \ \ \ ||\nabla k||_{L^{\infty}_{[0,T]}L^{2}(\Sigma_{t})}
\leq C, \ \ \ \inf_{t\in[0,T]}r_{{\rm vol}}(\Sigma_{t},1)\geq\frac{1}{C},
\end{equation*}
holds on $t\in[0,T]$.
\end{theorem}

\subsection{Bounded $L^{2}$-curvature conjecture for the
Einstein scalar field equations}\label{subsection4.3}

Motivated by {\color{red}{Theorem \ref{t2.7}}} and {\color{red}{Theorem \ref{t4.1}}}, we propose the following

\begin{conjecture}\label{c4.2}The Einstein scalar field equations admit local Cauchy developments for initial data sets $(\Sigma, g, k, u_{0},u_{1})$ with locally finite $L^{2}$-curvature, locally finite $L^{2}$-norm of the first covariant derivatives of $k$, locally finite $L^{2}$-norm of the covariant derivatives (up to second order) of $u_{0}$, and locally finite $L^{2}$-norm of the covariant derivatives (up to first order) of $u_{1}$.
\end{conjecture}

An interesting question related to {\color{red}{Theorem \ref{t4.1}}} is

\begin{question}\label{q4.3} Can we extend {\color{red}{Theorem \ref{t4.1}}} to the Einstein scalar
field equations?
\end{question}

\section{${\rm Sm}$ and Wylie-Yeroshkin Riemann curvature}\label{section5}

In this section we compare our curvature ${\rm Sm}$ with $\alpha_{1}=2$ {\color{blue}{(\ref{3.10})}} with a notion of curvature introduced recently by Wylie and Yeroshkin \cite{WY2016}. For other notions of sectional curvatures,
see \cite{KW2017, KWY2017, Wylie2015}.

Let $(M,g)$ be a Riemannian manifold of dimension $n$ with a smooth function $u$. Wylie and Yeroshkin introduced the following weighted connection
\begin{equation}
\nabla^{u}_{X}Y:=\nabla_{X}Y-(Yu)X-(Xu)Y.\label{5.1}
\end{equation}
By Proposition 3.3 in \cite{WY2016}, we have
\begin{equation}
R^{u}_{ijk\ell}=R_{ijk\ell}
+\nabla_{j}\nabla_{k}u\!\ g_{i\ell}
-\nabla_{i}\nabla_{k}u\!\ g_{j\ell}
+\nabla_{j}u\nabla_{k}u\!\ g_{i\ell}
-\nabla_{i}u\nabla_{k}u\!\ g_{j\ell},\label{5.2}
\end{equation}
where $R^{u}_{ijk\ell}:=\langle{\rm Rm}^{\alpha}(\partial_{i},
\partial_{j})\partial_{k},\partial_{\ell}\rangle$ and ${\rm Rm}^{\alpha}$ is the induced Riemann curvature tensor associated to the connection $\nabla^{u}$. The Ricci curvature associated to $\nabla^{u}$ is defined by
\begin{equation}
R^{u}_{jk}:=g^{i\ell}R^{u}_{ijk\ell}
=R_{ik}+(n-1)\nabla_{j}\nabla_{k}u+(n-1)\nabla_{j}u\nabla_{k}u.\label{5.3}
\end{equation}
Here the last formula also follows from Proposition 3.3 in \cite{WY2016}.

Recall from {\color{blue}{(\ref{3.10})}} that (with $\alpha_{1}=2$)
\begin{equation}
S_{ijk\ell}=R_{ijk\ell}-\nabla_{i}u\nabla_{k}u\!\ g_{j\ell}-\nabla_{i}u\nabla_{j}u\!\ g_{k\ell}.\label{5.4}
\end{equation}
From now on, we are given a smooth function $u$ on $M$ and write
\begin{eqnarray}
R^{{\bf L}}_{ijk\ell}&:=&S_{ijk\ell}, \ \ \
R^{{\bf WY}}_{ijk\ell} \ \ := \ \ R^{u}_{ijk\ell},\nonumber\\
R^{{\bf L}}_{jk}&:=&g^{i\ell}R^{{\bf L}}_{ijk\ell}, \ \ \
R^{{\bf WY}}_{jk} \ \ := \ \ R^{u}_{jk} \ \ = \ \ g^{i\ell}R^{{\bf WY}}_{ijk\ell},\label{5.5}\\
R^{{\bf L}}&:=&g^{jk}R^{{\bf L}}_{jk}, \ \ \
R^{{\bf WY}} \ \ := \ \ g^{jk}R^{{\bf WY}}_{jk}.\nonumber
\end{eqnarray}
From {\color{blue}{(\ref{5.3})}} and {\color{blue}{(\ref{5.4})}}, we have
\begin{equation*}
{\rm Ric}^{{\bf L}}={\rm Ric}-2du\otimes du, \ \ \
{\rm Ric}^{{\bf WY}}={\rm Ric}+(n-1)du\otimes du
+(n-1)\nabla^{2}u.
\end{equation*}

\begin{remark}\label{r5.1} We note that ${\rm Ric}^{{\bf L}}$ and ${\rm Ric}^{{\bf WY}}$ are
actually the Ricci curvatures in the sense of Bakey-\'Emery \cite{BE1985}. We here use our notions to keep the paper smoothly.
\end{remark}

There is another type of Ricci curvature given by
\begin{equation}
\widehat{R}^{{\bf WY}}_{jk}:=g^{i\ell}R^{{\bf WY}}_{ji\ell k}
=R_{jk}+\left(\Delta u+|\nabla u|^{2}\right) g_{jk}
-\nabla_{j}\nabla_{k}u-\nabla_{j}u\nabla_{k}u.\label{5.6}
\end{equation}

\begin{lemma}\label{l5.2}{\bf (Basic identities for $R^{{\bf L}}_{ijk\ell}$ and $R^{{\bf WY}}_{ijk\ell}$)}  We have
\begin{equation}
R^{{\bf WY}}_{ijk\ell}-R^{{\bf L}}_{ijk\ell}
=\nabla_{i}u\nabla_{j}u\!\ g_{k\ell}
+\nabla_{k}u\nabla_{j}u\!\ g_{i\ell}
+\nabla_{j}\nabla_{k}u\!\ g_{i\ell}-\nabla_{i}\nabla_{k}u\!\ g_{j\ell},
\label{5.7}
\end{equation}
\begin{equation}
R^{{\bf WY}}_{ijk\ell}-R^{{\bf WY}}_{jik\ell}
=2\left(\nabla_{j}u\nabla_{k}u\!\ g_{i\ell}
-\nabla_{i}\nabla_{k}u\!\ g_{j\ell}+\nabla_{j}u\nabla_{k}u\!\ g_{i\ell}
-\nabla_{i}u\nabla_{k}u\!\ g_{j\ell}\right),\label{5.8}
\end{equation}
\begin{eqnarray}
R^{{\bf WY}}_{ijk\ell}-R^{{\bf WY}}_{ij\ell k}
&=&\nabla_{i}\nabla_{\ell}u\!\ g_{jk}+\nabla_{j}\nabla_{k}u\!\ g_{i\ell}
-\nabla_{i}\nabla_{k}u\!\ g_{j\ell}-\nabla_{j}\nabla_{\ell}u\!\ g_{ik}\nonumber\\
&&+ \ \nabla_{i}u\nabla_{\ell}u\!\ g_{jk}
+\nabla_{j}u\nabla_{k}u\!\ g_{i\ell}-\nabla_{i}u\nabla_{k}u\!\ g_{j\ell}
-\nabla_{j}u\nabla_{\ell}u\!\ g_{ik},\label{5.9}
\end{eqnarray}
\begin{equation}
R^{{\bf WY}}_{ijk\ell}-R^{{\bf WY}}_{k\ell ij}
=\nabla_{j}\nabla_{k}u\!\ g_{i\ell}-\nabla_{\ell}\nabla_{i}u\!\ g_{jk}
+\nabla_{j}u\nabla_{k}u\!\ g_{i\ell}-\nabla_{\ell}u\nabla_{i}u\!\ g_{jk},\label{5.10}
\end{equation}
\begin{equation}
R^{{\bf L}}_{ijk\ell}-R^{{\bf L}}_{jik\ell}
=\nabla_{j}u\nabla_{k}u\!\ g_{i\ell}
-\nabla_{i}u\nabla_{k}u\!\ g_{j\ell},\label{5.11}
\end{equation}
\begin{equation}
R^{{\bf L}}_{ijk\ell}-R^{{\bf L}}_{ij\ell k}
=\nabla_{i}u\nabla_{\ell}u\!\ g_{jk}-\nabla_{i}u\nabla_{k}u\!\ g_{j\ell},\label{5.12}
\end{equation}
\begin{equation}
R^{{\bf L}}_{ijk\ell}-R^{{\bf L}}_{k\ell ij}
=\nabla_{k}u\nabla_{\ell}u\!\ g_{ij}
-\nabla_{i}u\nabla_{j}u\!\ g_{k\ell},\label{5.13}
\end{equation}
\begin{equation}
\frac{1}{2}\left(R^{{\bf WY}}_{ijk\ell}
-R^{{\bf WY}}_{jik\ell}\right)
=\left(R^{{\bf L}}_{ijk\ell}-R^{{\bf L}}_{jik\ell}\right)
+\nabla_{j}\nabla_{k}u\!\ g_{i\ell}
-\nabla_{i}\nabla_{k}u\!\ g_{j\ell},\label{5.14}
\end{equation}
\begin{equation}
\widehat{{\rm Ric}}{}^{{\bf WY}}
-{\rm Ric}^{{\bf WY}}=\left(\Delta u+|\nabla u|^{2}\right)g
-n\left(\nabla^{2}u+du\otimes du\right).\label{5.15}
\end{equation}
\end{lemma}

\subsection{Integral inequalities for scalar and Ricci curvatures}
\label{subsection5.1}

We now have four different types of Ricci curvatures, ${\rm Ric}$, ${\rm Ric}^{{\bf L}}$, ${\rm Ric}^{{\bf WY}}$,
and $\widehat{{\rm Ric}}^{{\bf WY}}$, and three different
types of scalar curvatures, $R$, $R^{{\bf L}}$, and $R^{{\bf WY}}$. In order to compare those quantities, we introduce a notation $\mathcal{P}\leq_{{\bf I},\mu}\mathcal{Q}$, which is an integral inequality with respect to the measure $\mu$.

\begin{definition}\label{d5.3} Given two scalar quantities $\mathcal{P}, \mathcal{Q}$
on $(M,g)$, and a measure $\mu$, we write $\mathcal{P}\leq_{{\bf I},\mu}
\mathcal{Q}$ if the following inequality
\begin{equation}
\int_{M}\mathcal{P}\!\ d\mu\leq\int_{\mathcal{M}}\mathcal{Q}\!\ d\mu
\label{5.16}
\end{equation}
holds. When $d\mu$ is the volume form $dV$, we simply write {\color{blue}{(\ref{5.16})}} as $\mathcal{P}\leq_{{\bf I}}
\mathcal{Q}$. When $d\mu$ is the measured volume form $e^{f}dV$, we write {\color{blue}{(\ref{5.16})}} as $
\mathcal{P}_{\leq{\bf I},f}\mathcal{Q}$. Similarly, we can define
$\mathcal{P}_{{\bf I},\mu}\mathcal{Q}$.
\end{definition}

\begin{proposition}\label{p5.4} For any measure $\mu$ on $M$ and smooth function $u$ on $M$, we have
\begin{equation}
R^{{\bf L}}\leq_{{\bf I},\mu} R, \ \ \ R\leq_{{\bf I}}
R^{{\bf WY}}, \ \ \ R=_{{\bf I},u}R^{{\bf WY}}.
\label{5.17}
\end{equation}
\end{proposition}

\begin{proof} It follows from the definitions
\begin{equation*}
R^{{\bf L}}=R-2|\nabla u|^{2}, \ \ \ R^{{\bf WY}}=R+(n-1)(|\nabla u|^{2}
+\Delta u)
\end{equation*}
and the fact that integral of $(\Delta u+|\nabla u|^{2})$ with respect to $e^{u}dV$ is zero.
\end{proof}

This proposition shows that $R^{{\bf L}}\leq_{{\bf I}}
R\leq_{{\bf I}} R^{{\bf WY}}$ and $R^{{\bf L}}\leq_{{\bf I},u} R=_{{\bf I},u}R^{{\bf WY}}$. Thus, in the sense of integrals, $R_{{\bf L}}$ is
weaker and $R^{{\bf WY}}$ is stronger than $R$, respectively.
\\

Next we consider the similar question on Ricci curvatures.

\begin{definition}\label{d5.5} Let $(M,g)$ be a closed Riemannian manifold with a smooth function $u$, and $\mu$ be a given measure on $M$. Given two Ricci curvatures ${\rm Ric}^{\clubsuit}, {\rm Ric}^{\diamondsuit}\in\mathfrak{Ric}_{4}
:=\{{\rm Ric}, {\rm Ric}^{{\bf L}}, {\rm Ric}^{{\bf WY}},
\widehat{{\rm Ric}}^{{\bf WY}}\}$, we say
\begin{equation}
{\rm Ric}^{\clubsuit}\leq_{{\bf I},\mu}{\rm Ric}^{\diamondsuit}
\label{5.18}
\end{equation}
if ${\rm Ric}^{\clubsuit}(X,X)\leq_{{\bf I},\mu}{\rm Ric}^{\diamondsuit}
(X,X)$ in the sense of {\color{red}{Definition \ref{d5.3}}} for all vector fields $X
\in\mathfrak{X}(M)$. Similarly we can define ${\rm Ric}^{\clubsuit}\leq_{{\bf I}}{\rm Ric}^{\diamondsuit}$ and ${\rm Ric}^{\clubsuit}\leq_{{\bf I},f}{\rm Ric}^{\diamondsuit}$.

We say
\begin{equation}
{\rm Ric}^{\clubsuit}\leq_{{\bf IK},\mu}{\rm Ric}^{\diamondsuit}\label{5.19}
\end{equation}
if ${\rm Ric}^{\clubsuit}(X,X)\leq_{{\bf I},\mu}{\rm Ric}^{\diamondsuit}
(X,X)$ in the sense of {\color{red}{Definition \ref{d5.3}}} for all {\it Killing} vector fields $X
\in\mathfrak{X}_{{\bf K}}(M)$, where $\mathfrak{X}_{{\bf K}}(M)$ is the space of all Killing vector
fields on $M$. Similarly we can define ${\rm Ric}^{\clubsuit}\leq_{{\bf IK}}{\rm Ric}^{\diamondsuit}$ and ${\rm Ric}^{\clubsuit}\leq_{{\bf IK},f}{\rm Ric}^{\diamondsuit}$.

Consider the subset $\mathfrak{X}_{{\bf KC}}(M)$ of $\mathfrak{X}_{{\bf K}}
(M)$, which consists of Killing vector fields on $M$ with constant norm. we say
\begin{equation}
{\rm Ric}^{\clubsuit}\leq_{{\bf IKC},\mu}{\rm Ric}^{\diamondsuit}
\label{5.20}
\end{equation}
if ${\rm Ric}^{\clubsuit}(X,X)\leq_{{\bf I},\mu}{\rm Ric}^{\diamondsuit}
(X,X)$ in the sense of {\color{red}{Definition \ref{d5.3}}} for all $X
\in\mathfrak{X}_{{\bf KC}}(M)$. Similarly we can define ${\rm Ric}^{\clubsuit}\leq_{{\bf IKC}}{\rm Ric}^{\diamondsuit}$ and ${\rm Ric}^{\clubsuit}\leq_{{\bf IKC},f}{\rm Ric}^{\diamondsuit}$.
\end{definition}

\begin{theorem}\label{t5.6} Let $(M,g)$ be a closed Riemannian manifold with a smooth function $u$ and $\mu$ be a given measure on $M$. Then we have
\begin{itemize}

\item[(i)] ${\rm Ric}^{{\bf L}}\leq_{{\bf I}, \mu}{\rm Ric}$.

\item[(ii)] ${\rm Ric}\leq_{{\bf IKC}}{\rm Ric}^{{\bf
WY}}$.

\item[(iii)] ${\rm Ric}\leq_{{\bf IKC}}\widehat{{\rm Ric}}{}^{{\bf WY}}$.

\end{itemize}

\end{theorem}

\begin{proof} From {\color{blue}{(\ref{5.5})}}, we have ${\rm Ric}^{{\bf L}}(X,X)
={\rm Ric}(X,X)-2\langle X,\nabla u\rangle^{2}$ and then the first
part follows. To prove the last result, we compute
\begin{equation*}
\int_{M}\left[{\rm Ric}^{{\bf WY}}(X,X)
-{\rm Ric}(X,X)\right]dV
=(n-1)\int_{M}\left[\langle X,\nabla u\rangle^{2}
+X^{i}X^{j}\nabla_{i}\nabla_{j}u\right]dV.
\end{equation*}
Then we suffice to verify that the last integral is nonnegative for
any Killing vector field $X$. According to {\color{blue}{(\ref{5.23})}}, we can prove (ii) and (iii) immediately. Indeed, for (iii),
\begin{eqnarray*}
\int_{M}\left[\widehat{{\rm Ric}}{}^{{\bf WY}}
(X,X)-{\rm Ric}(X,X)\right]dV
&=&\frac{3}{2}\int_{M}u\Delta|X|^{2}dV\\
&&+ \ \int_{M}[|X|^{2}|\nabla u|^{2}-\langle X,\nabla u\rangle^{2}]dV\\
&\geq&\frac{3}{2}\int_{M}u\Delta|X|^{2}dV
\end{eqnarray*}
for any Killing vector field $X$.
\end{proof}

To prove {\color{red}{Lemma \ref{l5.7}}}, we first recall Yano's formula (see
\cite{Yano1, Yano2, YB1953} and, \cite{LL2015} for related topics)
\begin{equation}
\int_{M}|\mathscr{L}_{X}g|^{2}dV
=\int_{M}[|\nabla X|^{2}+|{\rm div}(X)|^{2}-{\rm Ric}(X,X)]dV, \ \ \ X
\in\mathfrak{X}(M),\label{5.21}
\end{equation}
gives a necessary and sufficient condition to $X$ being Killing
or not. Namely,
\begin{equation}
X\in\mathfrak{X}_{{\bf K}}(M)\Longleftrightarrow\Delta X+\nabla{\rm div}X
+{\rm Ric}(X)=0={\rm div} X.\label{5.22}
\end{equation}

\begin{lemma}\label{l5.7} For any vector field $X$ and any smooth function $u$, we have
\begin{eqnarray}
\int_{M}X^{i}X^{j}\nabla_{i}\nabla_{j}u\!\ dV&=&\int_{M}
u\langle X,\Delta X+\nabla{\rm div}X+{\rm Ric}(X)\rangle dV\nonumber\\
&&+ \ \int_{M}\frac{1}{2}u[|\mathscr{L}_{X}g|^{2}-\Delta|X|^{2}]dV
-\int_{M}\langle X,\nabla u\rangle{\rm div}X\!\ dV.
\label{5.23}
\end{eqnarray}
In particular, when $X\in\mathfrak{X}_{{\bf K}}(M)$, we get
\begin{equation}
\int_{M}X^{i}X^{j}\nabla_{i}\nabla_{j}u\!\ dV
=-\frac{1}{2}\int_{M}u\Delta|X|^{2}dV=
-\frac{1}{2}\int_{M}|X|^{2}\Delta u\!\ dV.\label{5.24}
\end{equation}

\end{lemma}

\begin{proof} We start from the computation:
\begin{eqnarray*}
\int_{M}X^{i}X^{j}\nabla_{i}\nabla_{j}u\!\ dV&=&
-\int_{M}\nabla_{j}u\nabla_{i}(X^{i}X^{j})\!\ dV\\
&=&-\int_{M}\nabla_{j}u\left(X^{j}{\rm div}X+X^{i}\nabla_{i}X^{j}
\right)dV\\
&=&-\int_{M}\langle X,\nabla u\rangle{\rm div}X\!\ dV
+\int_{M}u\nabla_{j}(X^{i}\nabla_{i}X^{j})dV.
\end{eqnarray*}
The last integral, denoted by $I_{u}$, as a function of $u$, is equal to
\begin{eqnarray*}
I_{u}&=&\int_{M}u(\nabla_{j}X^{i}\nabla_{i}X^{j})dV
+\int_{M}uX^{i}\nabla_{j}\nabla_{i}X^{j}dV\\
&=&\int_{M}u(\nabla_{j}X^{i}\nabla_{i}X^{j})dV
+\int_{M}uX^{i}\left(\nabla_{i}{\rm div}X+R_{ij}X^{j}\right)dV\\
&=&\int_{M}u\left[\nabla_{j}X^{i}\nabla_{i}X^{j}
+{\rm Ric}(X,X)\right]dV-\int_{M}{\rm div}X\left(\langle X,\nabla u\rangle
+u{\rm div}X\right)dV\\
&=&\int_{M}u\left[\nabla_{j}X^{i}
\nabla_{i}X^{j}
+{\rm Ric}(X,X)-|{\rm div}(X)|^{2}\right]dV-\int_{M}\langle X,\nabla u
\rangle{\rm div}X\!\ dV.
\end{eqnarray*}
Using the following identities:
\begin{eqnarray*}
\frac{1}{2}|\mathscr{L}_{X}g|^{2}&=&|\nabla X|^{2}
+\nabla_{j}X^{i}\nabla_{i}X^{j},\\
\int_{M}u|{\rm div}(X)|^{2}dV&=&\int_{M}(u{\rm div}X)\nabla_{i}X^{i}dV\\
&=&-\int_{M}X^{i}\left(\nabla_{i}u{\rm div}X+u\nabla_{i}{\rm div}X\right)\!\ dV\\
&=&-\int_{M}\langle X,\nabla u\rangle{\rm div}X\!\ dV
-\int_{M}u\langle X,\nabla{\rm div}X\rangle dV,\\
\int_{M}u|\nabla X|^{2}dV&=&\int_{M}u\nabla^{i}X_{j}\nabla_{i}X^{j}dV \ \ = \ \ -\int_{M}X^{j}\left(\nabla_{i}u\nabla^{i}X_{j}
+u\Delta X_{j}\right)dV\\
&=&-\int_{M}u\langle X,\Delta X\rangle dV-\frac{1}{2}
\int_{M}\nabla_{i}u\nabla^{i}|X|^{2}dV,
\end{eqnarray*}
we obtain
\begin{equation}
I_{u}=\int_{M}u\left[\langle X,\Delta X+\nabla{\rm div}X
+{\rm Ric}(X)\rangle dV+\frac{1}{2}|\mathscr{L}_{X}g|^{2}
-\frac{1}{2}\Delta|X|^{2}\right]dV.\label{5.25}
\end{equation}
Then {\color{blue}{(\ref{5.23})}} follows from {\color{blue}{(\ref{5.25})}}.
\end{proof}

A consequence of {\color{red}{Theorem \ref{t5.6}}} indicates
\begin{equation}
{\rm Ric}^{{\bf L}}\leq_{{\bf IKC}}{\rm Ric}\leq_{{\bf IKC}}{\rm Ric}^{{\bf WY}} \ \ \ \text{and} \ \ \
{\rm Ric}^{{\bf L}}\leq_{{\bf IKC}}{\rm Ric}\leq_{{\bf IKC}}\widehat{{\rm Ric}}{}^{{\bf WY}}.
\label{5.26}
\end{equation}

To prove an analogous result in {\color{red}{Theorem \ref{t5.6}}} between ${\rm Ric}$ and ${\rm Ric}^{{\bf WY}},
\widehat{{\rm Ric}}{}^{{\bf WY}}$ along constant Killing vector
fields, for some measure $\mu$ other than the volume form, we first do the following computation. For the measure $e^{f}dV$, where $f\equiv f(u)$
is a smooth function depends only on $u$, we have
\begin{equation}
\int_{M}\left[{\rm Ric}^{{\bf WY}}(X,X)
-{\rm Ric}(X,X)\right]e^{f}dV
=(n-1)\int_{M}\left[\langle X,\nabla u\rangle^{2}
+X^{i}X^{j}\nabla_{i}\nabla_{j}u\right]e^{f}dV.\label{5.27}
\end{equation}
As above, the last integral is equal to
\begin{eqnarray}
\int_{M}X^{i}X^{j}\nabla_{i}\nabla_{j}u e^{f}dV&=&
-\int_{M}\nabla_{j}u\nabla_{i}(e^{f}X^{i}X^{j})dV\nonumber\\
&=&-\int_{M}\nabla_{j}u\left(f'\nabla_{i}uX^{i}X^{j}
+X^{j}{\rm div}X+X^{i}\nabla_{i}X^{j}\right)e^{f}dV\nonumber\\
&=&-\int_{M}\langle X,\nabla u\rangle {\rm div}X\!\ e^{f}dV
-\int_{M}f'\langle\nabla u,X\rangle^{2}e^{f}dV\nonumber\\
&&+ \ \int_{M}u e^{f}\nabla_{j}(X^{i}\nabla_{i}X^{j})dV
+\int_{M}u f'e^{f}\nabla_{j}uX^{i}\nabla_{i}X^{j}\label{5.28}\\
&=&-\int_{M}\langle X,\nabla u\rangle{\rm div}X e^{f}dV
-\int_{M}f'\langle\nabla u,X\rangle^{2}e^{f}dV\nonumber\\
&&+ \ I_{ue^{f}}-I_{u^{2}f' e^{f}}
-\int_{M}u\nabla_{j}\left(u f'e^{f}\right)X^{i}\nabla_{i}X^{j}dV\nonumber
\end{eqnarray}
where we used the notion $I_{u}$ that is explicitly
given in {\color{blue}{(\ref{5.25})}}. Next we require
\begin{equation*}
u f'e^{f}=1\Longrightarrow\text{we can take} \ f(u)=\ln\ln u.
\end{equation*}
However, in the above function $f(u)$, we should assume that $u$ is strictly positive. It motivates us to consider the modified function associated with $u$.
\\

Given a smooth function $u$ on the closed Riemannian manifold $(M,g)$,
set
\begin{equation}
\widetilde{u}:=u-u_{\min}+c_{0}\geq c_{0}\geq\frac{1}{e}, \ \ \
\widetilde{f}:=\ln\ln\widetilde{u},\label{5.29}
\end{equation}
where $u_{\min}:=\min_{M}u$ and $c_{0}\geq e^{-1}$ is the unique constant such that $c_{0}
\ln c_{0}=1$. Replacing $(u,f)$ by $(\widetilde{u},
\widetilde{f})$ in {\color{blue}{(\ref{5.27})}} (since $\nabla\widetilde{u}
=\nabla u$), we immediately obtain
\begin{equation*}
\int_{M}
\left[{\rm Ric}^{{\bf WY}}(X,X)-{\rm Ric}(X,X)\right]e^{\widetilde{f}}dV
=(n-1)\int_{M}\left[\langle X,\nabla\widetilde{u}\rangle^{2}
+X^{i}X^{j}\nabla_{i}\nabla_{j}\widetilde{u}\right]e^{\widetilde{f}}
dV
\end{equation*}
where
\begin{eqnarray*}
\int_{M}X^{i}X^{j}\nabla_{i}\nabla_{j}\widetilde{u} e^{\widetilde{f}}dV
&=&-\int_{M}\langle X,\nabla\widetilde{u}\rangle{\rm div}X e^{f}dV
-\int_{M}\widetilde{f}'\langle\nabla\widetilde{u},X\rangle^{2}e^{\widetilde{f}}
dV\\
&&+ \ I_{\widetilde{u}e^{\widetilde{f}}}-I_{\widetilde{u}^{2}\widetilde{f}'
e^{\widetilde{f}}}-\int_{M}\widetilde{u}\nabla_{j}
\left(\widetilde{u}\widetilde{f}'e^{\widetilde{f}}\right)X^{i}\nabla_{i}
X^{j}dV\\
&=&-\int_{M}\langle X,\nabla\widetilde{u}\rangle{\rm div}X e^{f}dV
-\int_{M}\frac{1}{\widetilde{u}\ln\widetilde{u}}
\langle X,\nabla\widetilde{u}\rangle^{2}e^{\widetilde{f}}dV\\
&&+ \ I_{\widetilde{u}e^{\widetilde{f}}}
-I_{\widetilde{u}^{2}\widetilde{f}'e^{\widetilde{f}}}.
\end{eqnarray*}
When $X\in\mathfrak{X}_{{\bf KC}}(M)$, we can simplify the above integral into
\begin{equation*}
\int_{M}X^{i}X^{j}\nabla_{i}\nabla_{j}
\widetilde{u}e^{\widetilde{f}}dV=
-\int_{M}\frac{1}{\widetilde{u}\ln\widetilde{u}}
\langle X,\nabla\widetilde{u}\rangle^{2}e^{\widetilde{f}}dV
\end{equation*}
and hence
\begin{equation*}
\int_{M}\left[{\rm Ric}^{{\bf WY}}(X,X)
-{\rm Ric}(X,X)\right]e^{\widetilde{f}}dV
=(n-1)\int_{M}\left(1-\frac{1}{\widetilde{u}\ln\widetilde{u}}
\right)\langle X,\nabla\widetilde{u}\rangle^{2}e^{\widetilde{f}}dV.
\end{equation*}
Since the function $t\ln t$, $t\geq e^{-1}$, is increasing, it follows
that $\widetilde{u}\ln\widetilde{u}\geq c_{0}\ln c_{0}=1$. Therefore,
we obtain the first part of the following theorem.

\begin{theorem}\label{t5.8} Let $(M,g)$ be a closed Riemannian manifold with a smooth function $u$ and $\mu$ be a given measure on $M$. Then we have
\begin{itemize}

\item[(i)] ${\rm Ric}\leq_{{\bf IKC},\widetilde{f}}{\rm Ric}^{{\bf
WY}}$, and

\item[(ii)] ${\rm Ric}\leq_{{\bf IKC},\widetilde{f}}\widehat{{\rm Ric}}{}^{{\bf WY}}$.

\end{itemize}
where $\widetilde{f}:=u-u_{\min}+c_{0}$ and $c_{0}\geq1/e$.
\end{theorem}

\begin{proof} The proof of the second part is precisely the same as that
of (iii) in {\color{red}{Theorem \ref{t5.6}}}.
\end{proof}

\subsection{Killing vector fields with constant length}
\label{subsection5.2}

In {\color{red}{Theorem \ref{t5.6}}} and {\color{red}{Theorem \ref{t5.8}}}, we proved integral
inequalities along Killing vector fields with constant length. In this
subsection we review some existence results on such vector fields.

Eisenhart \cite{E1926} proved that a unit vector field $X$ on a (connected) complete Riemannian manifold $(M,g)$ is the unit Killing vector field if and only if the angles between $X$ and tangent vectors to each geodesic in $(M,g)$ are constant along this geodesic. As earlier, Bianchi \cite{B1918} proved that A Killing vector field $X$ on a complete Riemannian manifold $(M,g)$ has constant length if and only if every integral curve of $X$ is a geodesic in
$(M,g)$. For a proof, we refer to \cite{BN2006}.

There are two necessary conditions for the existence of Killing vector fields of constant length on a given Riemannian manifold $(M,g)$, one is $\chi(M)=0$ by Hopf's theorem while another one, by Bott's theorem \cite{Bott1967}, is that all the Pontrjagin numbers of the oriented cover of $M$ are zero

The existence of Killing vector fields with constant length on a complete Riemannian manifold $(M,g)$ is connected
with {\it Clifford-Wolf translations} or {\it Clifford
translations}, which is an isometry $s^{\bf CW}$ on $(M,g)$ such that $d(x,s^{\bf CW}(x))\equiv$ constant for all $x\in M$.

\begin{itemize}

\item If a one-parameter isometry group generated by a Killing vector field $X$ consists of Clifford-Wolf translations, then $X$ had constant length.

\item If $X$ is a Killing vector field of constant length on a compact Riemannian manifold $(M,g)$, then the isometries $\gamma(t)$, generated by $X$, are Clifford-Wolf translations for sufficiently small $|t|$.

\end{itemize}
The first fact is obvious by definition. The compactness condition in the second fact can be generalized to the condition that $(M,g)$ has the injectivity radius, bounded from below by some positive constant. A proof can be found in \cite{BN2006}.

\begin{proposition}\label{p5.9} On each of $28$ homotopical seven-dimensional spheres $M$, there exist a Riemannian metric $g$ and a nonzero vector field $X$, such that

\begin{itemize}

\item ${\rm Ric}^{{\bf L}}(X,X)\leq_{{\bf I}}{\rm Ric}(X,X)\leq_{\bf I}{\rm Ric}^{{\bf WY}}(X,X)$ and ${\rm Ric}^{{\bf L}}(X,X)\leq_{{\bf I}}{\rm Ric}(X,X)
    \leq_{{\bf I}}\widehat{{\rm Ric}}{}^{{\bf WY}}(X,X)$ hold.

\item for any smooth function $u$ on $M$, ${\rm Ric}^{{\bf L}}(X,X)\leq_{{\bf I},\widetilde{f}}{\rm Ric}(X,X)\leq_{{\bf I},\widetilde{f}}{\rm Ric}^{{\bf WY}}(X,X)$ and ${\rm Ric}^{{\bf L}}(X,X)\leq{\rm Ric}(X,X)
    \leq_{{\bf I},\widetilde{f}}\widehat{{\rm Ric}}{}^{{\bf WY}}(X,X)$ hold, where $\tilde{f}:=u-u_{\min}+c_{0}$ with $c_{0}\geq1/e$.

\end{itemize}

\end{proposition}

\begin{proof} According to \cite{BN2006} (see Corollary 11) and \cite{MY1973} we can take $X$ to be a Killing vector field of constant
length with respect to $g$. Then the results follow from {\color{red}{Theorem \ref{t5.6}}} and {\color{red}{Theorem \ref{t5.8}}}.
\end{proof}

We say that a Riemannian metric $g$ on $M$ is of {\it cohomogeneity 1} if some compact Lie group $G$ acts smoothly and isometrically on $M$ and the space of orbits $M/G$ with respect to this action is one-dimensional.

\begin{proposition}\label{p5.10} Let $n\geq2$ and $\epsilon>0$. On the sphere $\mathbb{S}^{2n-1}$, there are a (real-analytic) Riemannian
metric $g_{\epsilon}$, of cohomogeneity $1$, with the property that all section curvatures of $g_{\epsilon}$ differ from $1$ at most by $\epsilon$, and a (real-analytic) nonzero vector field $X_{\epsilon}$, such that

\begin{itemize}

\item ${\rm Ric}^{{\bf L}}_{g_{\epsilon}}(X_{\epsilon},X_{\epsilon})\leq_{{\bf I}}{\rm Ric}_{g_{\epsilon}}(X_{\epsilon},X_{\epsilon})\leq_{\bf I}{\rm Ric}^{{\bf WY}}_{g_{\epsilon}}(X_{\epsilon},X_{\epsilon})$ and ${\rm Ric}^{{\bf L}}_{g_{\epsilon}}(X_{\epsilon},X_{\epsilon})\leq_{{\bf I}}{\rm Ric}_{g_{\epsilon}}(X_{\epsilon},X_{\epsilon})
    \leq_{{\bf I}}\widehat{{\rm Ric}}{}^{{\bf WY}}_{g_{\epsilon}}
    (X_{\epsilon},X_{\epsilon})$ hold.

\item for any smooth function $u$ on $M$, ${\rm Ric}^{{\bf L}}_{g_{\epsilon}}(X_{\epsilon},X_{\epsilon})\leq_{{\bf I},\widetilde{f}}{\rm Ric}_{g_{\epsilon}}(X_{\epsilon},X_{\epsilon})\leq_{{\bf I},\widetilde{f}}{\rm Ric}^{{\bf WY}}_{g_{\epsilon}}$ $(X_{\epsilon},X_{\epsilon})$ and ${\rm Ric}^{{\bf L}}_{g_{\epsilon}}(X_{\epsilon},X_{\epsilon})\leq{\rm Ric}_{g_{\epsilon}}(X_{\epsilon},X_{\epsilon})
    \leq_{{\bf I},\widetilde{f}}\widehat{{\rm Ric}}{}^{{\bf WY}}_{g_{\epsilon}}(X_{\epsilon},X_{\epsilon})$ hold, where $\tilde{f}:=u-u_{\min}+c_{0}$ with $c_{0}\geq1/e$.

\end{itemize}

\end{proposition}

\begin{proof} According to \cite{BN2006} (see Theorem 21), we can take $X$ to be a Killing
vector field $X$ of unit length with respect to $g$. Then the results follow from {\color{red}{Theorem \ref{t5.6}}} and {\color{red}{Theorem \ref{t5.8}}}.
\end{proof}

\begin{remark}\label{r5.11} Fix $n\geq2$ and $D>0$. Let $\mathfrak{S}_{n,D}$ denote the class of simply-connected $n$-dimensional Riemannian manifolds $(M,g)$ with the sectional curvature $|{\rm Sec}_{g}|_{g}\leq 1$ and with ${\rm diam}(M,g)\leq D$. The class $\mathfrak{S}_{n,\pi/\sqrt{\delta}}$ contains a subsclass $\mathfrak{PS}_{n,\delta}$, which consists of all simply-connected $n$-dimensional Riemannian manifold $(M,g)$ with the sectional curvature $0<\delta<{\rm Sec}_{g}\leq 1$.

Tuschmann \cite{Tuschmann1997} proved that there is a positive number $v:=v(n,D)$ with the following property: if $(M,g)\in\mathfrak{S}_{n,D}$ satisfies ${\rm vol}(M,g)<v$, then

\begin{itemize}

\item[(i)] there is a smooth locally free action of the $\mathbb{S}^{1}$-action on $M$, and

\item[(ii)] for every $\epsilon>0$ there exists a $\mathbb{S}^{1}$-invariant metric $g_{\epsilon}$ on $M$ such that
    \begin{equation*}
    e^{-\epsilon}g\leq g_{\epsilon}<e^{\epsilon}g, \ \ \ |\nabla_{g}
    -\nabla_{g_{\epsilon}}|_{g}<\epsilon, \ \ \ |\nabla^{i}_{g_{\epsilon}}
    {\rm Rm}_{g_{\epsilon}}|_{g}<C(n,i,\epsilon).
    \end{equation*}

\end{itemize}

If the $\mathbb{S}^{1}$-action in (i) is free, Berestovskii and Nikonorov \cite{BN2006} observed that we can find a Killing vector field $X_{\epsilon}$ of unit length with respect to $g_{\epsilon}$. Consequently, in this case,

\begin{itemize}

\item ${\rm Ric}^{{\bf L}}_{g_{\epsilon}}(X_{\epsilon},X_{\epsilon})\leq_{{\bf I}}{\rm Ric}_{g_{\epsilon}}(X_{\epsilon},X_{\epsilon})\leq_{\bf I}{\rm Ric}^{{\bf WY}}_{g_{\epsilon}}(X_{\epsilon},X_{\epsilon})$ and ${\rm Ric}^{{\bf L}}_{g_{\epsilon}}(X_{\epsilon},X_{\epsilon})\leq_{{\bf I}}{\rm Ric}_{g_{\epsilon}}(X_{\epsilon},X_{\epsilon})
    \leq_{{\bf I}}\widehat{{\rm Ric}}{}^{{\bf WY}}_{g_{\epsilon}}
    (X_{\epsilon},X_{\epsilon})$ hold.

\item for any smooth function $u$ on $M$, ${\rm Ric}^{{\bf L}}_{g_{\epsilon}}(X_{\epsilon},X_{\epsilon})\leq_{{\bf I},\widetilde{f}}{\rm Ric}_{g_{\epsilon}}(X_{\epsilon},X_{\epsilon})\leq_{{\bf I},\widetilde{f}}{\rm Ric}^{{\bf WY}}_{g_{\epsilon}}$ $(X_{\epsilon},X_{\epsilon})$ and ${\rm Ric}^{{\bf L}}_{g_{\epsilon}}(X_{\epsilon},X_{\epsilon})\leq{\rm Ric}_{g_{\epsilon}}(X_{\epsilon},X_{\epsilon})
    \leq_{{\bf I},\widetilde{f}}\widehat{{\rm Ric}}{}^{{\bf WY}}_{g_{\epsilon}}(X_{\epsilon},X_{\epsilon})$ hold, where $\tilde{f}:=u-u_{\min}+c_{0}$ with $c_{0}\geq1/e$.

\end{itemize}

\end{remark}

\subsection{Remark on ${\rm Rm}^{{\bf L}}$ and ${\rm Rm}^{{\bf WY}}$}
\label{subsection5.3}

The nonnegativity of $R^{{\bf L}}_{ijk\ell}$ was used in \cite{Wu2018}
to prove the compactness for gradient shrinking Ricci harmonic solitons.
\\

There is no useful relation between ${\rm Rm}^{{\bf L}}$ and ${\rm Rm}^{{\bf WY}}$. More precisely, we can find a Riemannian manifold $(M,g)$ so that ${\rm Rm}^{{\bf L}}(X,Y,Y,X)< {\rm Rm}^{{\bf WY}}(X,Y,Y,X)$
for some triple $(X, Y, u)$ of smooth vector fields $X, Y$ and smooth function $u$, and ${\rm Rm}^{{\bf L}}(X,Y,Y,X)> {\rm Rm}^{
{\bf WY}}(X,Y,Y,X)$ for another such triple $(X', Y', u')$.

\begin{example}\label{e5.12} Consider the Euclidean space $(\mathbb{R}^{n},
g_{\mathbb{R}^{n}})$ with the flat Riemannian metric $g_{\mathbb{R}^{n}}$. Consider a smooth function $u(x)=\varphi(r)$ with $r=|x|^{2}$,
where $\varphi(r)$ is a smooth function of variable $r$. Let $T=x^{i}\partial/\partial x^{i}$ denote the position vector field
on $\mathbb{R}^{n}$. Then
\begin{equation*}
\nabla_{i}u=2T_{i}\varphi', \ \ \ \nabla_{i}\nabla_{j}u
=4T_{i}T_{j}\varphi''+2\delta_{ij}\varphi'.
\end{equation*}
According to (\ref{5.5}), we have
\begin{eqnarray*}
R^{{\bf L}}_{ijk\ell}&=&-4\varphi'{}^{2}T_{i}T_{k}
\delta_{j\ell}-4\varphi'{}^{2}T_{i}T_{j}\delta_{k\ell},\\
R^{{\bf WY}}_{ijk\ell}&=&\left(4\varphi''T_{j}T_{k}
+2\varphi'\delta_{jk}\right)\delta_{i\ell}
-\left(4\varphi''T_{i}T_{k}+2\varphi'\delta_{ik}\right)\delta_{j\ell}\\
&&+ \ 4\varphi'{}^{2}\delta_{i\ell}T_{j}T_{k}
-4\varphi'{}^{2}\delta_{j\ell}T_{i}T_{k},
\end{eqnarray*}
so that
\begin{eqnarray*}
{\rm Rm}^{{\bf L}}(X,Y,Y,X)&=&-8\varphi'{}^{2}\langle X,T\rangle
\langle Y,T\rangle\langle X,Y\rangle,\\
{\rm Rm}^{{\bf WY}}(X,Y,Y,X)&=&
4(\varphi''+\varphi'{}^{2})\left(|X|^{2}
\langle Y,T\rangle^{2}-\langle X,T\rangle\langle Y,T\rangle
\langle X,Y\rangle\right)\\
&&+ \ 2\varphi'(|X|^{2}|Y|^{2}-\langle X,Y\rangle^{2})
\end{eqnarray*}
For any vector fields $X, Y$. Choosing $X=T$ and $Y$ with the property
that $\langle Y,T\rangle=0$ yields
\begin{equation*}
{\rm Rm}^{{\bf L}}(T,Y,Y,T)=0, \ \ \ {\rm Rm}^{{\bf WY}}(T,Y,Y,T)=2\varphi'|T|^{2}
|Y|^{2}.
\end{equation*}
When $\varphi(r)=r$, we have ${\rm Rm}^{{\bf L}}(T,Y,Y,T)
\leq{\rm Rm}^{{\bf WY}}(T,Y,Y,T)$. On the other hand, for $\varphi(r)=
-r$, we have ${\rm Rm}^{{\bf L}}(T,Y,Y,T)
\geq {\rm Rm}^{{\bf WY}}(T,Y,Y,T)$.
\end{example}

There is also no useful pointwise relation between ${\rm Rm}$ and ${\rm Rm}^{{\bf L}}$. By definition,
\begin{equation*}
{\rm Rm}^{{\bf L}}(X,Y,Y,X)={\rm Rm}(X,Y,Y,X)
-2\langle X,\nabla u\rangle\langle Y, \nabla u\rangle\langle X,Y\rangle
\end{equation*}
for any vector fields $X, Y\in\mathfrak{X}(M)$.

\begin{remark}\label{r5.13} (1) For $n\geq2$, take $X, Y\in\mathfrak{X}(M)$ so that $\langle X, Y\rangle=0$ at a point. Then, at this point, we get ${\rm Rm}^{{\bf L}}(X,Y,Y,X)={\rm Rm}(X,Y,Y,X)$.

(2) When $X=\nabla u$, we have
\begin{equation*}
{\rm Rm}^{{\bf L}}(\nabla u,Y,Y,\nabla u)=
{\rm Rm}(\nabla u,Y,Y,\nabla u)-2|\nabla u|^{2}
\langle Y,\nabla u\rangle^{2}\leq{\rm Rm}(\nabla u,Y,Y,\nabla u).
\end{equation*}

(3) For $n\geq3$, take $X, Y\in\mathfrak{X}(M)$ so that $\langle X,\nabla u\rangle, \langle Y,\nabla u\rangle, \langle X,Y\rangle<0$ at a point. Then, at this point, ${\rm Rm}^{{\bf L}}(X,Y,Y,X)>{\rm Rm}(X,Y,Y,X)$.

(4) In general, for any vector fields $X, Y\in\mathfrak{X}(M)$, we have
\begin{equation*}
{\rm Rm}^{{\bf L}}(X,X,X,X)=-2\langle X,\nabla u\rangle^{2}
|X|^{2}\leq0={\rm Rm}(X,X,X,X)
\end{equation*}
and the equality holds at a point if and only if $\langle X,\nabla u\rangle=0$ or $X
=0$ at this point.
\end{remark}

To give a relation between ${\bf Rm}^{{\bf L}}(X,Y,Y,X)$ and ${\bf Rm}(X,Y,Y,X)$ in the sense of ${\bf IKC}$, we let
\begin{equation}
J(X,Y):=\int_{M}\langle X,\nabla u\rangle\langle Y,\nabla u\rangle
\langle X,Y\rangle dV, \ \ \ X, Y\in\mathfrak{X}(M).\label{5.30}
\end{equation}
It is clear that $J(X,Y)=J(Y,X)$. Compute
$$
\int_{M}\langle X,\nabla u\rangle\langle Y,\nabla u\rangle
\langle X,Y\rangle dV \ \ = \ \ \int_{X}\nabla_{i}u\left[X^{i}\langle Y,\nabla u\rangle
\langle X,Y\rangle\right]dV
$$
$$
= \ \
-\int_{M}u\left[{\rm div}(X)\langle Y,\nabla u\rangle\langle X,Y\rangle
+X^{i}\nabla_{i}(\langle Y,\nabla u\rangle\langle X,Y\rangle)\right]dV
$$
$$
= \ \ -\int_{M}u
\left[{\rm div}(X)\langle Y,\nabla u\rangle\langle X,Y\rangle
+X^{i}\nabla_{i}(Y^{j}\nabla_{j}uX^{k}Y_{k})\right]dV
$$
$$
= \ \
-\int_{M}u\bigg[{\rm div}(X)\langle Y,\nabla u\rangle\langle X,Y\rangle
+X^{i}\bigg(\nabla_{i}Y^{j}\nabla_{j}uX^{k}Y_{k}
$$
$$
+ \ Y^{j}\nabla_{i}\nabla_{j}uX^{k}Y_{k}
+Y^{j}\nabla_{j}u\nabla_{i}X^{k}Y_{k}
+Y^{j}\nabla_{j}uX^{k}\nabla_{i}Y_{k}\bigg)\bigg]dV
$$
$$
= \ \ -\int_{M}u{\rm div}(X)\langle Y,\nabla u\rangle\langle X,Y\rangle dV
-\int_{M}uX^{i}\nabla_{i}Y^{j}\nabla_{j}uX^{k}Y_{k}dV
$$
$$
- \ \int_{M}uX^{i}Y^{j}\nabla_{i}\nabla_{j}uX^{k}Y_{k}dV
-\int_{M}u X^{i}Y^{j}\nabla_{j}u\nabla_{i}X^{k}Y_{k}dV
-\int_{M}uX^{i}Y^{j}\nabla_{j}uX^{k}\nabla_{i}Y_{k}dV
$$
$$
=: \ \ -\int_{M}u{\rm div}(X)\langle Y,\nabla u\rangle\langle X,Y\rangle dV
-J_{1}-J_{2}-J_{3}-J_{4}.
$$
Using the definition of Lie derivative that $(\mathscr{L}_{Y}g)_{ik}
=\nabla_{i}Y_{k}+\nabla_{k}Y_{i}$, we have
$$
J_{4} \ \ = \ \ \int_{M}uX^{i}Y^{j}\nabla_{j}uX^{k}\nabla_{i}Y_{k}dV \ \
= \ \ -\int_{M}u\nabla_{j}(uX^{i}Y^{j}X^{k}\nabla_{i}Y_{k})dV
$$
$$
= \ \ -\int_{M}u\bigg[\nabla_{j}uX^{i}Y^{j}X^{k}\nabla_{i}Y_{k}
+u\nabla_{j}X^{i}Y^{j}X^{k}\nabla_{i}Y_{k}
$$
$$
+ \ u X^{i}{\rm div}(Y)X^{k}\nabla_{i}Y_{k}
+u X^{i}Y^{j}\nabla_{j}X^{k}\nabla_{i}Y_{k}
+u X^{i}Y^{j}X^{k}\nabla_{j}\nabla_{i}Y_{k}\bigg]dV
$$
$$
= \ \ -J_{4}-\int_{M}u^{2}\bigg[{\rm div}(Y)X^{i}X^{k}\nabla_{i}
Y_{k}+X^{i}Y^{j}X^{k}\nabla_{j}\nabla_{i}Y_{k}
$$
$$
+ \ X^{i}Y^{j}\nabla_{j}X^{k}\nabla_{i}Y_{k}
+X^{k}Y^{j}\nabla_{j}X^{i}\nabla_{i}Y_{k}\bigg]dV
$$
$$
= \ \ -J_{4}-\int_{M}u^{2}\left[{\rm div}(Y)
X^{i}X^{k}\nabla_{i}Y_{k}
+X^{i}Y^{j}X^{k}\nabla_{j}\nabla_{i}Y_{k}
+X^{i}Y^{j}\nabla_{j}X^{k}(\mathscr{L}_{Y}g)_{ik}\right]dV;
$$
hence
\begin{equation}
J_{4}=-\frac{1}{2}\int_{M}u^{2}\left[{\rm div}(Y)
X^{i}X^{k}\nabla_{i}Y_{k}
+X^{i}Y^{j}X^{k}\nabla_{j}\nabla_{i}Y_{k}
+X^{i}Y^{j}\nabla_{j}X^{k}(\mathscr{L}_{Y}g)_{ik}\right]dV.\label{5.31}
\end{equation}
In particular
\begin{equation}
J_{4}=-\frac{1}{2}\int_{M}u^{2}(X^{i}Y^{j}X^{k}\nabla_{j}
\nabla_{i}Y_{k})dV=-\frac{1}{2}\int_{M}u^{2}
(X^{i}X^{j}Y^{k}\nabla_{k}\nabla_{i}Y_{j})dV, \ \ \ Y
\in\mathfrak{X}_{{\bf K}}(M).\label{5.32}
\end{equation}
Similarly,
$$
J_{3} \ \ = \ \ \int_{M}uX^{i}Y^{j}\nabla_{j}uY^{k}\nabla_{i}X_{k}dV \ \
= \ \ -\int_{M}u\nabla_{j}(u X^{i}Y^{j}Y^{k}\nabla_{i}X_{k})dV
$$
$$
= \ \ -\int_{M}u\bigg[\nabla_{j}uX^{i}Y^{j}Y^{k}\nabla_{i}X_{k}
+u\nabla_{j}X^{i}Y^{j}Y^{k}\nabla_{i}X_{k}
$$
$$
+ \ uX^{i}{\rm div}(Y)Y^{k}\nabla_{i}X_{k}
+u X^{i}Y^{j}\nabla_{j}Y^{k}\nabla_{i}X_{k}
+uX^{i}Y^{j}Y^{k}\nabla_{j}\nabla_{i}X_{k}\bigg]dV
$$
$$
= \ \ -J_{3}-\int_{M}u^{2}\bigg[{\rm div}(Y)X^{i}Y^{k}\nabla_{i}X_{k}
+X^{i}Y^{j}Y^{k}\nabla_{j}\nabla_{i}X_{k}
$$
$$
+ \ \nabla_{j}X^{i}\nabla_{i}X_{k}Y^{j}Y^{k}
+\nabla_{j}Y^{k}\nabla_{i}X_{k}X^{i}Y^{j}\bigg]dV.
$$
Hence
\begin{eqnarray}
J_{3}&=&-\frac{1}{2}\int_{M}u^{2}\bigg[{\rm div}(Y)X^{i}Y^{k}\nabla_{i}
X_{k}+X^{i}Y^{j}Y^{k}\nabla_{j}\nabla_{i}X_{k}\nonumber\\
&&+ \ \nabla_{j}X^{i}\nabla_{i}X_{k}Y^{j}Y^{k}
+\langle\nabla_{X}X,\nabla_{Y}Y\rangle\bigg]dV.\label{5.33}
\end{eqnarray}
From {\color{blue}{(\ref{5.32})}} and {\color{blue}{(\ref{5.33})}}, we obtain
\begin{eqnarray}
J(X,Y)&=&-\int_{M}u\nabla_{j}u\nabla_{X}Y^{j}\langle X,Y\rangle dV
-\int_{M}uX^{i}Y^{j}\nabla_{i}\nabla_{j}u\langle X,Y\rangle dV\nonumber\\
&&+ \ \frac{1}{2}\int_{M}u^{2}
\left[X^{i}Y^{j}Y^{k}\nabla_{j}\nabla_{i}X_{k}
+\nabla_{Y}X^{i}\nabla_{i}X_{k}Y^{k}+\langle\nabla_{X}X,
\nabla_{Y}Y\rangle\right]dV\nonumber\\
&&+ \ \frac{1}{2}\int_{M}u^{2}\left(X^{i}X^{j}Y^{k}
\nabla_{k}\nabla_{i}Y_{j}\right)dV\nonumber\\
&=&-\int_{M}u\langle X,Y\rangle\left[\langle\nabla u,\nabla_{X}Y\rangle
+X^{i}Y^{j}\nabla_{i}\nabla_{j}u\right]dV\label{5.34}\\
&&+ \ \frac{1}{2}\int_{M}u^{2}\bigg[\langle\nabla_{X}X,\nabla_{Y}Y\rangle
+\langle\nabla_{\nabla_{Y}X}X,Y\rangle\nonumber\\
&&+ \ X^{i}Y^{k}(Y^{j}\nabla_{j}\nabla_{i}X_{k}
+X^{j}\nabla_{k}\nabla_{i}Y_{j})\bigg]dV.\nonumber
\end{eqnarray}
According to the following identities
\begin{eqnarray*}
X^{i}Y^{k}Y^{j}\nabla_{j}\nabla_{i}X_{k}&=&Y^{k}Y^{j}\nabla_{j}(X^{i}
\nabla_{i}X_{k})-Y^{k}Y^{j}\nabla_{j}X^{i}\nabla_{i}X_{k}\\
&=&Y^{k}Y^{j}\nabla_{j}\nabla_{X}X_{k}-\nabla_{Y}X^{i}\nabla_{i}X_{k}Y^{k}\\
&=&\langle Y,\nabla_{Y}\nabla_{X}X\rangle-\langle Y,\nabla_{\nabla_{Y}X}X\rangle,\\
X^{i}Y^{k}X^{j}\nabla_{k}\nabla_{i}Y_{j}&=&
Y^{k}X^{j}\nabla_{k}(X^{i}\nabla_{i}Y_{j})-Y^{k}X^{j}\nabla_{k}X^{i}\nabla_{i}Y_{j}\\
&=&Y^{k}X^{j}\nabla_{k}\nabla_{X}Y_{j}-\nabla_{Y}X^{i}\nabla_{i}Y_{j}
X^{j}\\
&=&\langle X,\nabla_{Y}\nabla_{X}Y\rangle
-\langle X,\nabla_{\nabla_{Y}X}Y\rangle,
\end{eqnarray*}
we find that
$$
\langle\nabla_{X}X,\nabla_{Y}Y\rangle
+\langle\nabla_{\nabla_{Y}X}X,Y\rangle+X^{i}Y^{k}(Y^{j}\nabla_{j}\nabla_{i}X_{k}
+X^{j}\nabla_{k}\nabla_{i}Y_{j})
$$
$$
= \ \ \langle\nabla_{X}X,\nabla_{Y}Y\rangle
+\langle Y,\nabla_{\nabla_{Y}X}X\rangle
+\langle X,\nabla_{Y}\nabla_{X}Y\rangle
$$
$$
+ \ \langle Y,\nabla_{Y}\nabla_{X}X\rangle
-\langle X,\nabla_{\nabla_{Y}X}X\rangle
-\langle Y\nabla_{\nabla_{Y}X}X\rangle
$$
$$
= \ \ \langle X,\nabla_{Y}\nabla_{X}Y
-\nabla_{\nabla_{Y}X}X\rangle
+\langle\nabla_{X}X,\nabla_{Y}Y\rangle
+\langle Y,\nabla_{Y}\nabla_{X}X\rangle
$$
$$
= \ \ {\rm Rm}(Y,X,Y,X)+\langle\nabla_{X}\nabla_{Y}Y
-\nabla_{\nabla_{X}Y}Y,X\rangle+Y\langle Y,\nabla_{X}X\rangle.
$$
Plugging it into {\color{blue}{(\ref{5.34})}} and using $\nabla_{X}X=0$ for any $X\in\mathfrak{X}_{{\bf KC}}(M)$ yields
\begin{eqnarray}
J(X,Y)&=&-\int_{M}u\langle X,Y\rangle[\langle\nabla u,
\nabla_{X}Y\rangle+X^{i}Y^{j}\nabla_{i}\nabla_{j}u]dV\nonumber\\
&&+ \ \frac{1}{2}\int_{M}u^{2}\left[-{\rm Rm}(X,Y,Y,X)
-\langle X,\nabla_{\nabla_{X}Y}Y\rangle\right]dV, \ \ \ X, Y\in\mathfrak{X}_{{\bf KC}}(M).\label{5.35}
\end{eqnarray}
Since $J$ is symmetric, we can rewrite {\color{blue}{(\ref{5.35})}} in a symmetric form.
Changing $X$ and $Y$ in {\color{blue}{(\ref{5.35})}} we have
\begin{eqnarray*}
J(X,Y)&=&-\int_{M}u\langle X,Y\rangle[\langle\nabla u,
\nabla_{Y}X\rangle+X^{i}Y^{j}\nabla_{i}\nabla_{j}u]dV\nonumber\\
&&+ \ \frac{1}{2}\int_{M}u^{2}\left[-{\rm Rm}(X,Y,Y,X)
-\langle Y,\nabla_{\nabla_{Y}X}YX\rangle\right]dV, \ \ \ X, Y\in\mathfrak{X}_{{\bf KC}}(M).
\end{eqnarray*}
Combining it with {\color{blue}{(\ref{5.35})}} we arrive at
\begin{eqnarray}
J(X,Y)&=&-\int_{M}u\langle X,Y\rangle\left[\frac{1}{2}\langle\nabla u,
\nabla_{X}Y+\nabla_{Y}X\rangle
+X^{i}Y^{j}\nabla_{i}\nabla_{j}u\right]dV\nonumber\\
&&+ \ \frac{1}{2}\int_{M}u^{2}\Lambda dV, \ \ \ X, Y\in\mathfrak{X}_{{\bf KC}}
(M)\label{5.36}
\end{eqnarray}
where
\begin{equation}
\Lambda:=-{\rm Rm}(X,Y,Y,X)-\frac{1}{2}
\langle X,\nabla_{\nabla_{X}Y}Y\rangle
-\frac{1}{2}\langle Y,\nabla_{\nabla_{Y}X}X\rangle.\label{5.37}
\end{equation}
The following obvious identities
$$
-{\rm Rm}(X,Y,Y,X)-\langle X,\nabla_{\nabla_{X}Y}Y\rangle \ = \
-{\rm Rm}(X,Y,Y,X)+\langle Y,\nabla_{\nabla_{X}Y}X\rangle
-\nabla_{X}Y\langle X,Y\rangle,
$$
$$
-{\rm Rm}(X,Y,Y,X)-\langle Y,\nabla_{\nabla_{Y}X}X\rangle \ = \
-{\rm Rm}(X,Y,Y,X)+\langle X,\nabla_{\nabla_{Y}X}Y\rangle
-\nabla_{Y}X\langle Y,X\rangle,
$$
imply
\begin{eqnarray}
\Lambda&=&-{\rm Rm}(X,Y,Y,X)-\frac{1}{4}\langle X,Y\rangle
(\nabla_{X}Y+\nabla_{Y}X)\nonumber\\
&&+ \ \frac{1}{4}\left[\langle Y,\nabla_{[X,Y]}X\rangle
+\langle X,\nabla_{[Y,X]}Y\rangle\right].\label{5.38}
\end{eqnarray}
In summary, we obtain
\begin{eqnarray}
\int_{M}{\bf Rm}^{{\bf L}}(X,Y,Y,X)dV&=&\int_{M}u\langle X,Y\rangle
\left[\langle \nabla u,\nabla_{X}Y+\nabla_{Y}X\rangle
+2X^{i}Y^{j}\nabla_{i}\nabla_{j}u\right]dV\nonumber\\
&&+ \ \int_{M}\frac{u^{2}}{2}\left[\langle X,\nabla_{\nabla_{X}Y}Y
\rangle+\langle Y,\nabla_{\nabla_{Y}X}X\rangle\right]dV\label{5.39}\\
&&+ \ \int_{M}(1+u^{2})
{\rm Rm}(X,Y,Y,X)dV, \ \ \ X, Y\in\mathfrak{X}_{{\bf KC}}(M).\nonumber
\end{eqnarray}
or
\begin{eqnarray}
\int_{M}{\bf Rm}^{{\bf L}}(X,Y,Y,X)dV&=&\int_{M}u\langle X,Y\rangle
\left[\langle \nabla u,\nabla_{X}Y+\nabla_{Y}X\rangle
+2X^{i}Y^{j}\nabla_{i}\nabla_{j}u\right]dV\nonumber\\
&&+ \ \int_{M}\frac{u^{2}}{4}\langle X,Y\rangle
(\nabla_{X}Y+\nabla_{Y}X)dV\nonumber\\
&&- \ \int_{M}\frac{u^{2}}{4}\left[\langle Y,\nabla_{[X,Y]}X\rangle
+\langle X,\nabla_{[Y,X]}Y\rangle\right]dV\label{5.40}\\
&&+ \ \int_{M}(1+u^{2})
{\rm Rm}(X,Y,Y,X)dV, \ \ \ X, Y\in\mathfrak{X}_{{\bf KC}}(M).\nonumber
\end{eqnarray}

\section{Uniqueness for the Ricci-harmonic flow}\label{section6}

In the section, we prove the forward and backward uniqueness of solutions for the Ricci-harmonic flow. Suppose that $(M,g_{0})$ is a complete Riemannian manifold of dimension $n$ and $u_{0}$ is a smooth function on $M$. Consider the Ricci-harmonic flow
\begin{equation}
\partial_{t}g(t)=-2\!\ {\rm Ric}_{g(t)}+4\!\ \nabla_{g(t)}u(t)
\otimes \nabla_{g(t)} u(t), \ \ \ \partial_{t}u(t)=\Delta_{g(t)}u(t),\label{6.1}
\end{equation}
on $M\times[0,T]$.

\subsection{Forward uniqueness}\label{subsection6.1}

We now use the idea in \cite{K2014} to prove the following

\begin{theorem}\label{t6.1}{\bf (Forward uniqueness of the Ricci-harmonic flow)} Suppose that $(g(t),$ $u(t))$ and $(\tilde{g}(t),\tilde{u}(t))$ are two smooth
complete solutions of {\color{blue}{(\ref{6.1})}} with
\begin{equation}
\sup_{M\times[0,T]}\left(|{\rm Rm}_{g(t)}|_{g(t)}+|{\rm Rm}_{\tilde{g}(t)}|_{\tilde{g}(t)}\right)\leq K,\label{6.2}
\end{equation}
for some uniform constant $K$. If $(g(0), u(0))=(\tilde{g}(0),
\tilde{u}(0))=(g_{0}, u_{0})$, then $g(t)\equiv\tilde{g}(t)$ for each $t\in[0,T]$.
\end{theorem}

We now write
\begin{equation*}
g:=g(t), \ {\rm Rm}:={\rm Rm}_{g(t)}, \ \Gamma:=\Gamma_{g(t)}, \
u:=u(t), \ \nabla:=\nabla_{g(t)}, \ dV:=dV_{g(t)},
\end{equation*}
and
\begin{equation*}
\tilde{g}:=\tilde{g}(t), \ \widetilde{{\rm Rm}}:=
{\rm Rm}_{\tilde{g}(t)}, \ \widetilde{\Gamma}:=\Gamma_{\tilde{g}(t)}, \ \tilde{u}:=\tilde{u}(t), \
\widetilde{\nabla}:=\nabla_{\tilde{g}(t)}, \
d\widetilde{V}:=dV_{\tilde{g}(t)}.
\end{equation*}
We further fix a norm $|\cdot|:=|\cdot|_{g(t)}$. The $(p,q)$-tensor fields $T=(T^{i_{1}\cdots i_{p}}{}_{j_{1}
\cdots j_{q}})$ are smooth sections of the $(p,q)$-tensor bundle $\mathcal{T}^{p}_{q}(M)$ over $M$. For $(p,q)=(0,0)$, we write $\mathcal{T}^{0}_{0}(M)$ as $C^{\infty}(M)$. Introduce \begin{equation*}
(h,A,T,v,w,y)\in\mathcal{T}^{0}_{2}(M)\oplus\mathcal{T}^{1}_{2}(M)
\oplus\mathcal{T}^{1}_{3}(M)\oplus\mathcal{T}^{0}_{0}(M)\oplus\mathcal{T}^{0}_{1}
(M)\oplus\mathcal{T}^{0}_{2}(M)
\end{equation*}
according to the following definitions:
\begin{eqnarray}
h&\equiv& h(t) \ \ := \ \ g-\tilde{g}, \ \ \ h_{ij} \ \ := \ \ g_{ij}-\tilde{g}_{ij},\label{6.3}\\
A&\equiv& A(t) \ \ := \ \ \nabla-\widetilde{\nabla}, \ \ \ A^{k}_{ij} \ \ = \ \ \Gamma^{k}_{ij}-\widetilde{\Gamma}^{k}_{ij},\label{6.4}\\
T&\equiv& T(t) \ \ = \ \ {\rm Rm}-\widetilde{{\rm Rm}}, \ \ \ T^{\ell}_{ijk} \ \ = \ \
R^{\ell}_{ijk}-\widetilde{R}^{\ell}_{ijk},\label{6.5}\\
v&\equiv&v(t) \ \ := \ \  u-\tilde{u},\label{6.6}\\
w&\equiv& w(t) \ \ := \ \ \nabla u-\widetilde{\nabla}\tilde{u}, \ \ \
w_{i} \ \ = \ \ \nabla_{i}v,\label{6.7}\\
y&\equiv&y(t) \ \ := \ \ \nabla^{2}u-\widetilde{\nabla}{}^{2}\tilde{u}, \ \ \ y_{ij}
=\nabla_{i}\nabla_{j}u-\widetilde{\nabla}_{i}\widetilde{\nabla}_{j}
\tilde{u}.\label{6.8}
\end{eqnarray}
From the definitions, we have
\begin{eqnarray*}
y_{ij}&=&\partial^{2}_{ij}u-\partial^{2}_{ij}\tilde{u}
-\left(\Gamma^{k}_{ij}-\widetilde{\Gamma}^{k}_{ij}\right)\partial_{k}
\tilde{u}+\Gamma^{k}_{ij}\left(\partial_{k}\tilde{u}-
\partial_{k}u\right)\\
&=&\left[\left(\partial^{2}_{ij}u-\Gamma^{k}_{ij}\partial_{k}u\right)
-\left(\partial^{2}_{ij}\tilde{u}-\Gamma^{k}_{ij}\partial_{k}
\tilde{u}\right)\right]-A^{k}_{ij}\partial_{k}\tilde{u}\\
&=&\nabla_{i}w_{j}-A^{k}_{ij}\partial_{k}\tilde{u} \ \ = \ \
\frac{1}{2}\left(\nabla_{i}w_{j}+\nabla_{j}w_{i}\right)
-A^{k}_{ij}\partial_{k}\tilde{u}
\end{eqnarray*}
so that
\begin{equation}
y=\nabla w+A\ast\widetilde{\nabla}\tilde{u}.\label{6.9}
\end{equation}
Here $V\ast W$ means the linear combination of contractions of tensors $V$ and $W$ with respect to $g(t)$. From \cite{K2014}, we have
\begin{eqnarray}
\tilde{g}^{ij}-g^{ij}&=&\tilde{g}^{ia}g^{jb}
h_{ab}, \ \ \ g^{-1}-\tilde{g}^{-1} \ \ = \ \ \tilde{g}^{-1}\ast g\ast h,\label{6.10}\\
\nabla_{k}\tilde{g}^{ij}&=&
\tilde{g}^{\ell j}A^{i}_{k\ell}+\tilde{g}^{i\ell}A^{j}_{ka}, \ \ \nabla\tilde{g}^{-1} \ \ = \ \ \tilde{g}^{-1}\ast A.\label{6.11}
\end{eqnarray}
Further,
\begin{equation*}
\partial_{t}h_{ij}=-2T^{\ell}_{\ell ij}+4\left(\partial_{i}u
\partial_{j}u-\partial_{i}\tilde{u}\partial_{j}\tilde{u}\right)
\end{equation*}
so that
\begin{equation}
\partial_{t}h_{ij}=-2T^{\ell}_{\ell ij}
+w_{i}w_{j}+w_{i}\partial_{j}\tilde{u}+w_{j}\partial_{i}\tilde{u}
=-2T^{\ell}_{\ell ij}+w\ast w+w\ast\widetilde{\nabla}
\tilde{u}.\label{6.12}
\end{equation}
The evolution of $\partial_{t}A^{k}_{ij}$ can be derived similarly as that in \cite{K2014}. Indeed,
\begin{eqnarray*}
\partial_{t}A^{k}_{ij}
&=&\tilde{g}^{mk}\left(\widetilde{\nabla}_{i}\widetilde{R}_{jm}
+\widetilde{\nabla}_{j}\widetilde{R}_{im}-\widetilde{\nabla}_{m}
\widetilde{R}_{ij}\right)
-g^{mk}\left(\nabla_{i}R_{jm}+\nabla_{j}R_{im}-\nabla_{m}R_{ij}\right)\\
&&+ \ 4\left(g^{mk}\nabla_{m}u\nabla_{i}\nabla_{j}u
-\tilde{g}^{mk}\partial_{m}\tilde{u}\widetilde{\nabla}_{i}
\widetilde{\nabla}_{j}\tilde{u}\right)\\
&=&\tilde{g}^{-1}\ast h\ast\widetilde{\nabla}\widetilde{{\rm Rm}}
+A\ast\widetilde{{\rm Rm}}+1\ast\nabla T\\
&&+ \ 4\left[\left(g^{mk}-\tilde{g}^{mk}\right)
\widetilde{\nabla}_{m}\tilde{u}\widetilde{\nabla}_{i}
\widetilde{\nabla}_{j}\tilde{u}+g^{mk}
\left(\nabla_{m}u\nabla_{i}\nabla_{j}u
-\widetilde{\nabla}_{m}\tilde{u}\widetilde{\nabla}_{i}
\widetilde{\nabla}_{j}\tilde{u}\right)\right]
\end{eqnarray*}
where the third line comes from the computation in \cite{K2014}. In the last line, writing
\begin{equation*}
\nabla_{m}u\nabla_{i}\nabla_{j}u
-\widetilde{\nabla}_{m}\tilde{u}\widetilde{\nabla}_{i}
\widetilde{\nabla}_{j}\tilde{u}
=\left(\nabla_{m}u-\widetilde{\nabla}_{m}\tilde{u}\right)
\widetilde{\nabla}_{i}\widetilde{\nabla}_{j}\tilde{u}
+\nabla_{m}u\left(\nabla_{i}\nabla_{j}u
-\widetilde{\nabla}_{i}\widetilde{\nabla}_{j}\tilde{u}\right)
\end{equation*}
we can conclude
\begin{eqnarray}
\partial_{t}A&=&\tilde{g}^{-1}\ast h\ast\widetilde{\nabla}
\widetilde{{\rm Rm}}+A\ast\widetilde{{\rm Rm}}
+1\ast\nabla T\nonumber\\
&&+ \ \tilde{g}^{-1}\ast h\ast\widetilde{\nabla}u
\ast\widetilde{\nabla}^{2}\tilde{u}
+w\ast\widetilde{\nabla}^{2}\tilde{u}
+\nabla u\ast y.
\end{eqnarray}
The same argument gives
\begin{eqnarray*}
\partial_{t}T^{\ell}_{ijk}&=&\nabla_{a}
\left(g^{ab}\nabla_{b}R^{\ell}_{ijk}-\tilde{g}^{ab}
\widetilde{\nabla}_{b}\widetilde{R}^{\ell}_{ijk}\right)
+\tilde{g}^{-1}\ast A\ast\widetilde{\nabla}\widetilde{{\rm Rm}}+\tilde{g}^{-1}\ast h\ast\widetilde{{\rm Rm}}
\ast\widetilde{{\rm Rm}}\\
&&+ \ T\ast{\rm Rm}
+T\ast\widetilde{{\rm Rm}}+g^{\ell m}
\left(\nabla_{i}\nabla_{m}u\nabla_{k}\nabla_{j}u
-\nabla_{i}\nabla_{k}u\nabla_{j}\nabla_{m}u\right)\\
&&- \ \tilde{g}^{\ell m}\left(\widetilde{\nabla}_{i}\widetilde{\nabla}_{m}
\tilde{u}\widetilde{\nabla}_{k}\widetilde{\nabla}_{j}\tilde{u}
-\widetilde{\nabla}_{i}\widetilde{\nabla}_{k}\tilde{u}
\widetilde{\nabla}_{j}\widetilde{\nabla}_{m}\tilde{u}\right).
\end{eqnarray*}
Since
$$
g^{\ell m}\nabla_{i}\nabla_{m}u\nabla_{k}\nabla_{j}u
-\tilde{g}^{\ell m}\widetilde{\nabla}_{i}\widetilde{\nabla}_{m}
\tilde{u}\widetilde{\nabla}_{k}\widetilde{\nabla}_{j}\tilde{u}
$$
$$
= \ \ \left(g^{\ell m}-\tilde{g}^{\ell m}\right)
\widetilde{\nabla}_{i}\widetilde{\nabla}_{m}\tilde{u}
\widetilde{\nabla}_{k}\widetilde{\nabla}_{j}\tilde{u}
+g^{\ell m}\left(\nabla_{i}\nabla_{m}u\nabla_{k}\nabla_{j}u
-\widetilde{\nabla}_{i}\widetilde{\nabla}_{m}\tilde{u}
\widetilde{\nabla}_{k}\widetilde{\nabla}_{j}\tilde{u}\right)
$$
$$
= \ \ \tilde{g}^{-1}\ast h\ast\widetilde{\nabla}^{2}\tilde{u}
\ast\widetilde{\nabla}^{2}\tilde{u}
+g^{\ell m}
\bigg[\left(\nabla_{i}\nabla_{m}u-\widetilde{\nabla}_{i}
\widetilde{\nabla}_{m}\tilde{u}\right)\widetilde{\nabla}_{k}
\widetilde{\nabla}_{j}\tilde{u}
$$
$$
+ \ \nabla_{i}\nabla_{m}u
\left(\nabla_{k}\nabla_{j}u-\widetilde{\nabla}_{k}\widetilde{\nabla}_{j}
\tilde{u}\right)\bigg] \ \ = \ \ \tilde{g}^{-1}\ast h
\ast\widetilde{\nabla}^{2}\tilde{u}\ast\widetilde{\nabla}^{2}\tilde{u}
+y\ast\widetilde{\nabla}^{2}\tilde{u}
+y\ast\nabla^{2}u,
$$
it follows that
\begin{eqnarray}
\partial_{t}T^{\ell}_{ijk}&=&\nabla_{a}
\left(g^{ab}\nabla_{b}R^{\ell}_{ijk}-\tilde{g}^{ab}
\widetilde{\nabla}_{b}\widetilde{R}^{\ell}_{ijk}\right)
+\tilde{g}^{-1}\ast A\ast\widetilde{\nabla}\widetilde{{\rm Rm}}+\tilde{g}^{-1}\ast h\ast\widetilde{{\rm Rm}}
\ast\widetilde{{\rm Rm}}\nonumber\\
&&+ \ T\ast{\rm Rm}
+T\ast\widetilde{{\rm Rm}}+\tilde{g}^{-1}\ast h\ast\widetilde{\nabla}^{2}\tilde{u}\ast\widetilde{\nabla}^{2}\tilde{u}
+y\ast\nabla^{2}u+y\ast\widetilde{\nabla}^{2}\tilde{u}.\label{6.14}
\end{eqnarray}
Finally, we compute the evolution equation of $v$. Because
\begin{eqnarray*}
\partial_{t}v&=&\Delta u-\widetilde{\Delta}\tilde{u} \ \ = \ \
\Delta\left(u-\tilde{u}\right)+\left(g^{ij}\nabla_{i}\nabla_{j}
-\tilde{g}^{ij}\widetilde{\nabla}_{i}\widetilde{\nabla}_{j}\right)\tilde{u}\\
&=&\Delta v+\left(g^{ij}-\tilde{g}^{ij}\right)
\widetilde{\nabla}_{i}\widetilde{\nabla}_{j}\tilde{u}
+g^{ij}\left(\nabla_{i}\nabla_{j}-\widetilde{\nabla}_{i}\widetilde{\nabla}_{j}
\right)\tilde{u}\\
&=&\Delta v+\left(g^{ij}-\tilde{g}^{ij}\right)
\widetilde{\nabla}_{i}\widetilde{\nabla}_{j}\tilde{u}
+g^{ij}\left(\widetilde{\Gamma}^{k}_{ij}-\Gamma^{k}_{ij}\right)
\partial_{k}\tilde{u}
\end{eqnarray*}
we get
\begin{equation}
\partial_{t}v=\Delta v+\tilde{g}^{-1}\ast h
\ast\widetilde{\nabla}^{2}\tilde{u}+A\ast\widetilde{\nabla}
\tilde{u}.\label{6.15}
\end{equation}
From {\color{blue}{(\ref{6.12})}} -- {\color{blue}{(\ref{6.15})}}, we arrive at
\begin{eqnarray}
|\partial_{t}h|&\lesssim&|T|+|w|^{2}+|w||\widetilde{\nabla}
\tilde{u}|,\label{6.16}\\
|\partial_{t}A|&\lesssim&|\tilde{g}^{-1}|
|h||\widetilde{\nabla}\widetilde{{\rm Rm}}|+|A||\widetilde{{\rm Rm}}|+|\nabla T|\nonumber\\
&&+ \ |\tilde{g}^{-1}||h||\widetilde{\nabla}^{2}
\tilde{u}|+|w||\widetilde{\nabla}^{2}\tilde{u}|+|\nabla u||y|,\label{6.17}\\
|\partial_{t}T-\Delta T-{\rm div}\!\ S|&\lesssim&
|\tilde{g}^{-1}||A||\widetilde{\nabla}\widetilde{{\rm Rm}}|+|\tilde{g}^{-1}||h||\widetilde{{\rm Rm}}|^{2}
+|T|(|{\rm Rm}|+|\widetilde{{\rm Rm}}|)\nonumber\\
&&+ \ |\tilde{g}^{-1}||h||\widetilde{\nabla}^{2}\tilde{u}|^{2}+|y|(|\nabla^{2}u|+|\widetilde{\nabla}^{2}
\tilde{u}|),\label{6.18}\\
|\partial_{t}v-\Delta v|&\lesssim&|\tilde{g}^{-1}||h||\widetilde{\nabla}^{2}\tilde{u}|
+|A||\widetilde{\nabla}\tilde{u}|.\label{6.19}
\end{eqnarray}
Here the tensor $S=(S^{a\ell}_{ijk})$ is defined as
\begin{equation*}
S^{a\ell}_{ijk}:=g^{ab}\nabla_{b}\widetilde{R}^{\ell}_{ijk}
-\tilde{g}^{ab}\widetilde{\nabla}_{b}\widetilde{R}^{\ell}_{ijk}
=\tilde{g}^{-1}\ast h\ast\widetilde{\nabla}\widetilde{{\rm Rm}}
+A\ast\widetilde{{\rm Rm}}
\end{equation*}
and satisfies
\begin{equation}
|S|\lesssim|\tilde{g}^{-1}||h||\widetilde{\nabla}\widetilde{{\rm Rm}}|+|A||\widetilde{{\rm Rm}}|.\label{6.20}
\end{equation}

To further study, we need a version of Bernstein-Bando-Shi (BBS) estimate on higher derivatives of ${\rm Rm}$ and $u$. Under the
curvature condition {\color{blue}{(\ref{6.2})}}, according to
{\color{red}{Theorem \ref{tB.2}}}, we have
\begin{equation}
|\nabla u|\lesssim1, \ \ \ |\widetilde{\nabla}
\tilde{u}|\lesssim1,\label{6.21}
\end{equation}
and
\begin{equation}
|\nabla{\rm Rm}|
+|\nabla^{2}u|+|\widetilde{\nabla}\widetilde{{\rm Rm}}|
+|\widetilde{\nabla}{}^{2}\tilde{u}|\lesssim1.\label{6.22}
\end{equation}

\begin{proposition}\label{p6.2} Assume the curvature condition {\color{blue}{(\ref{6.2})}}. We first prove
\begin{equation}
|h(t)|\leq K_{0}\!\ t, \ \ \ |A(t)|\leq K_{0}t^{1/2}, \ \ \ |v(t)|\leq K_{0}t,
\label{6.23}
\end{equation}
on $M\times[0,T]$ for some uniform constant $K_{0}=K_{0}(n,K,T,g_{0},u_{0})$.
\end{proposition}

\begin{proof} Since $g(0)=\tilde{g}(0)=g_{0}$, the inequality {\color{blue}{(\ref{6.23})}} is trivial for $t=0$. In the following we may without loss of generality assume that $t\in(0,T]$. Given a space-time point $(p,t)\in M\times(0,T]$. The first can be proved by using {\color{blue}{(\ref{6.21})}}
\begin{eqnarray*}
|h(p,t)|&\lesssim&|h(p,t)|_{g_{0}} \ \ \lesssim \ \
\int^{t}_{0}|\partial_{s}h(p,s)|_{g_{0}}ds \ \ \lesssim \ \
\int^{t}_{0}|\partial_{s}h(p,s)|\!\ ds\\
&\lesssim&\int^{t}_{0}
\left[|T|+|w|^{2}+|w||\widetilde{\nabla}\tilde{u}|\right](p,s)ds \ \
\lesssim \ \ \int^{t}_{0}ds \ \ \lesssim \ \ t.
\end{eqnarray*}
For second one, using
, one has, by {\color{blue}{(\ref{6.22})}},
\begin{eqnarray*}
|A(p,t)|&\lesssim&|A(p,t)|_{g_{0}} \ \ \lesssim \ \
\int^{t}_{0}|\partial_{s}A(p,s)|_{g_{0}}ds \ \ \lesssim \ \ \int^{t}_{0}|\partial_{s}A(p,s)|ds\\
&\lesssim&\int^{t}_{0}
\left[|\widetilde{\nabla}\widetilde{{\rm Rm}}|
+|\nabla{\rm Rm}|+|\nabla u||\nabla^{2}u|
+|\widetilde{\nabla}\tilde{u}||\widetilde{\nabla}^{2}\tilde{u}|
\right]ds \\
&\lesssim&\int^{t}_{0}ds \ \ \lesssim \ \ t \ \ \lesssim \ \ t^{1/2}.
\end{eqnarray*}
The last one follows from $|\partial_{t}v|\lesssim|\nabla^{2}u|+|\widetilde{\nabla}{}^{2}\tilde{u}|$.
\end{proof}

Notice that the above proposition gives an explicit bound for $\nabla^{2}u$ and hence for $\Delta u$, provided the condition {\color{blue}{(\ref{6.2})}} holds. However, {\color{red}{Theorem \ref{t1.1}}} or {\color{red}{Theorem \ref{t2.6}}} gives an explicit bound for $\Delta u$ under a weaker condition {\color{blue}{(\ref{1.2})}}.
\\

To prove the uniqueness, we need the following

\begin{lemma}\label{l6.3} Consider a smooth family $(g(t))_{t\in[0,T]}$ of complete metrics on $M$ with $g(t)\geq\gamma^{-1}g$ for some uniform constant $\gamma$, where $g:=g(0)$. Choose any given point $x_{0}\in M$ and set $r(x):=d_{g}(x,x_{0})$. Then the following statement is true: For any given constants $L_{1}, L_{2}>0$, there exist a constant $T':=T'(n,\gamma,L_{1},L_{2},T)>0$ and a function $\eta: M\times[0,T']\to\mathbb{R}$ smooth in $t$, Lipschitz on $M\times\{t\}$, and satisfying
\begin{equation}
\partial_{t}\eta\geq L_{1}|\eta|^{2}_{g(t)}, \ \ \
e^{-\eta}\leq e^{-L_{2}r^{2}}\label{6.24}
\end{equation}
on $M\times[0,\tau]$ for all $\tau\in(0,T']$.
\end{lemma}

\begin{proof} See \cite{K2014}. Actually, we can pick $T'=\min\{T,1/4(\gamma L_{1}L_{2})^{1/2}\}$.
\end{proof}

Under the condition {\color{blue}{(\ref{6.2})}}, {\color{red}{Lemma \ref{l6.3}}} implies that for any given (sufficiently small and independent of $T$) constant $B>0$ there exist a constant $T'=T'(n,K, B,$ $T)>0$ and such a function $\eta: M\times[0,T']\to\mathbb{R}$ satisfy
\begin{equation}
e^{-\eta}\leq e^{-B r^{2}_{0}/T}, \ \ \ \partial_{t}\eta
\geq B|\nabla\eta|^{2}\label{6.25}
\end{equation}
on $M\times[0,\tau]$ for all $\tau\in(0,T']$; hence $\partial_{t}\eta, |\nabla\eta|^{2}$ are $e^{-\eta}dV$-integrable for all $t\in(0,T']$. For detail, see \cite{K2014}.
\\

There are three claims about the integrability of $\mathcal{E}(t)$ defined below. The proof is similar to that in \cite{K2014}, except for some extra terms.

\begin{itemize}

\item[(a)] $t^{-1}|h|^{2}, t^{-\beta}|A|^{2}, |T|^{2}$, $|v|^{2}$, and $|w|^{2}$ are uniformly bounded. It has been proved in {\color{red}{Proposition \ref{p6.2}}}.

\item[(b)] $\partial_{t}|h|^{2}, \partial_{t}|A|^{2}, \partial_{t}|T|^{2}$, $\partial_{t}|v|^{2}$, and $\partial_{t}|w|^{2}$ are uniformly bounded on $M\times[0,T]$ and consequently are $e^{-\eta}dV$-integrable for all $t\in(0,T']$. Compute
    \begin{eqnarray*}
    \left|\partial_{t}|h|^{2}\right|&=&\left|\partial_{t}(g^{-1}\ast g^{-1}\ast h\ast h)\right|\\
    &=&
    \left|g^{-1}\ast g^{-1}\ast g^{-1}\ast\partial_{t}g\ast h\ast h+g^{-1}\ast g^{-1}\ast h\ast\partial_{t}h\right|\\
    &\lesssim&|{\rm Ric}||h|^{2}
    +|h||\partial_{t}h| \ \ \lesssim \ \ |h|^{2}+|h|\left(|T|+|w|^{2}
    +|w|\right) \ \ \lesssim \ \ 1,
    \end{eqnarray*}
    by {\color{blue}{(\ref{6.16})}} and {\color{blue}{(\ref{6.22})}}. For $\partial_{t}|A|^{2}$ one has
    \begin{eqnarray*}
    \left|\partial_{t}|A|^{2}\right|&=&\partial_{t}(g^{-1}\ast g^{-1}\ast g\ast A\ast A)\\
    &=&g^{-1}\ast g^{-1}\ast g^{-1}\ast g\ast\partial_{t}g
    \ast A\ast A+g^{-1}\ast g^{-1}\ast g\ast A\ast\partial_{t}A\\
    &\lesssim&|{\rm Ric}||A|^{2}+|A||\partial_{t}A|\\
    &\lesssim&|A|^{2}
    +|A|\left(|\widetilde{\nabla}\widetilde{{\rm Rm}}|
    +|\nabla{\rm Rm}|+|\nabla u||\nabla^{2}u|+|\widetilde{\nabla}\tilde{u}||\widetilde{\nabla}{}^{2}
    \tilde{u}|\right) \ \ \lesssim \ \ 1,
    \end{eqnarray*}
    by {\color{blue}{(\ref{6.21})}}, {\color{blue}{(\ref{6.22})}}, {\color{blue}{(\ref{6.23})}}, and the proof of {\color{red}{Proposition \ref{p6.2}}}. Similarly, we can deal with $\partial_{t}|S|^{2}$, $\partial_{t}|v|^{2}$, and $\partial_{t}|w|^{2}$. For example,
    $$
    \left|\partial_{t}|v|^{2}\right|\lesssim|v|
    \left[|\nabla^{2}u|+|\widetilde{\nabla}{}^{2}\tilde{u}|\right]
    \lesssim|v|\lesssim 1
    $$
    using again {\color{red}{Proposition \ref{p6.2}}}.

\end{itemize}

Fixed $\beta\in(0,1)$, introduce the quantity
\begin{equation}
\mathcal{E}(t):=\int_{M}\left(t^{-1}|h|^{2}+t^{-\beta}|A|^{2}+|T|^{2}
+|v|^{2}+|w|^{2}\right)e^{-\eta}dV.\label{6.26}
\end{equation}
Here the function $\eta$ is determined by {\color{blue}{(\ref{6.25})}}.

\begin{itemize}

\item[(c)] $\mathcal{E}(t)$ is differentiable on $[0,T']$ and $\lim_{t\to0}\mathcal{E}(t)=0$. This follows from {\color{red}{Proposition \ref{p6.2}}} and Lebesgue's dominated convergence theorem.

\end{itemize}
The above claims (a) -- (c) allow us frequently to take time derivatives inside the integrals.

\begin{proposition}\label{p6.4} Assume that the curvature condition {\color{blue}{(\ref{6.2})}} is satisfied. There exist uniform constants $N=N(n,K, T, g_{0},u_{0})>0$ and $T_{0}=T_{0}(n,K,T, g_{0},u_{0},\beta)\in(0,T]$ such that
\begin{equation}
\mathcal{E}'(t)\leq N\!\ \mathcal{E}(t)\label{6.27}
\end{equation}
for $t\in[0,T_{0}]$. Hence $\mathcal{E}(t)\equiv0$ for all $t\in[0,T_{0}]$.
\end{proposition}

\begin{proof} As in \cite{K2014}, for any $t\in(0,T]$ and $\alpha\in(0,1)$, define
\begin{equation}
\mathcal{G}(t) \ \ := \ \ \int_{M}|T|^{2}e^{-\eta}dV, \ \ \
\mathcal{H}(t) \ \ := \ \ \int_{M}t^{-1}|h|^{2}e^{-\eta}dV,\label{6.28}
\end{equation}
\begin{equation}
\mathcal{I}(t) \ \ := \ \ \int_{M}t^{-\beta}|A|^{2}e^{-\eta}dV, \ \ \
\mathcal{J}(t) \ \ := \ \ \int_{M}|\nabla T|^{2}e^{-\eta}dV,\label{6.29}
\end{equation}
\begin{equation}
\mathcal{V}(t) \ := \ \int_{M}|v|^{2}e^{-\eta}dV, \
\mathcal{D}(t) \ := \ \int_{M}|y|^{2}e^{-\eta}dV, \
\mathcal{B}(t) \ := \ \int_{M}|w|^{2}e^{-\eta}dV.\label{6.30}
\end{equation}
Then
\begin{equation}
\mathcal{E}(t)=\mathcal{G}(t)+\mathcal{H}(t)+\mathcal{I}(t)+\mathcal{V}(t)
+\mathcal{B}(t).\label{6.31}
\end{equation}
We denote by $C$ any constant depending only on $n$, and by $N$ any constant depending on $n, K, T, g_{0}, u_{0},\beta$.

Since $g(t)$ are all uniformly equivalent to $g=g(0)$, we can replace the norm $|\cdot|$ in {\color{blue}{(\ref{6.26})}} by $|\cdot|_{g}$. Hence we may regard the norm $|\cdot|$ in {\color{blue}{(\ref{6.26})}} is independent of time.
\\

{\bf (1) Estimate for $\mathcal{G}'$.} Start with
\begin{eqnarray*}
\mathcal{G}'&=&\int_{M}\left[\partial_{t}|T|^{2}-
|T|^{2}\partial_{t}\eta+\left(-R+2|\nabla u|^{2}\right)|T|^{2}\right]e^{-\eta}dV\\
&\leq&N\mathcal{G}+\int_{M}\left[2\langle\partial_{t}T,T\rangle
-|T|^{2}\partial_{t}\eta\right]e^{-\eta}dV.
\end{eqnarray*}
Using {\color{blue}{(\ref{6.18})}}, {\color{blue}{(\ref{6.20})}}, {\color{blue}{(\ref{6.2})}}, {\color{blue}{(\ref{6.29})}},
and {\color{blue}{(\ref{6.31})}}, we have
\begin{eqnarray*}
\mathcal{G}'&\leq& N\mathcal{G}
+\int_{M}\bigg[2\langle\Delta T+{\rm div}\!\ S,T\rangle
+C|\tilde{g}^{-1}||\widetilde{\nabla}\widetilde{{\rm Rm}}||A||T|\\
&&+ \ C|\tilde{g}^{-1}||\widetilde{{\rm Rm}}|^{2}|h||T|+C\left(|{\rm Rm}|+|\widetilde{{\rm Rm}}|\right)|T|^{2}
-|T|^{2}\partial_{t}\eta\\
&&+ \ C|\tilde{g}^{-1}||\widetilde{\nabla}{}^{2}\tilde{u}|^{2}
|h||T|+\left(|\nabla^{2}u|+|\widetilde{\nabla}{}^{2}\tilde{u}|\right)
|y||T|\bigg]e^{-\eta}dV\\
&\leq& N\mathcal{G}+\int_{M}\bigg[\left(2\langle\Delta T+{\rm div}\!\ S,T\rangle-|T|^{2}\partial_{t}\eta\right)\\
&&+ \ N t^{-1/2}|A||T|+N|h||T|+N|T|^{2}+N|y||T|\bigg]e^{-\eta}dV
\end{eqnarray*}
The same argument in \cite{K2014} implies
\begin{equation*}
\int_{M}\left[2\langle\Delta T+{\rm div}\!\ S,T\rangle-|T|^{2}\partial_{t}\eta\right]e^{-\eta}dV\leq-\mathcal{J}
+N\mathcal{H}+N t^{\beta}\mathcal{I}
\end{equation*}
by choosing an appropriate constant $B$ in {\color{blue}{(\ref{6.25})}}. Therefore
\begin{equation}
\mathcal{G}'\leq N\mathcal{G}+N\mathcal{H}
+\left(N+t^{\beta-1}\right)\mathcal{I}-\mathcal{J}+
\frac{1}{4}\mathcal{D}.\label{6.32}
\end{equation}

{\bf (2) Estimate for $\mathcal{H}'$.} Since $\partial_{t}\eta\geq0$, we get
\begin{eqnarray*}
\mathcal{H}'&=&-t^{-1}\mathcal{H}+t^{-1}
\int_{M}\left[2\langle\partial_{t}h,h\rangle-|h|^{2}\partial_{t}\eta+|h|^{2}\left(-R+2|\nabla u|^{2}\right)\right]e^{-\eta}dV\\
&\leq&(N-t^{-1})\mathcal{H}+Ct^{-1}\int_{M}
|h|\left(|T|^{2}+|w|^{2}+|w||\widetilde{\nabla}\tilde{u}|\right)e^{-\eta}dV\\
&\leq&(N-t^{-1})\mathcal{H}+Ct^{-1}\int_{M}|h||T|e^{-\eta}dV
+Nt^{-1}\int_{M}|h|w|e^{-\eta}dV\\
&\leq&\left(N-\frac{1}{2}t^{-1}\right)\mathcal{H}
+C\mathcal{G}+N\int_{M}|w|^{2}e^{-\eta}dV.
\end{eqnarray*}
Thus
\begin{equation}
\mathcal{H}'\leq\left(N-\frac{1}{2}t^{-1}\right)\mathcal{H}+C\mathcal{G}
+N\mathcal{B}.\label{6.33}
\end{equation}

{\bf (3) Estimate for $\mathcal{I}'$.} As in the estimation of $\mathcal{H}'$, we have
\begin{eqnarray*}
\mathcal{I}'&\leq&(N-\beta t^{-1})\mathcal{I}
+Ct^{-\beta}\int_{M}|A|\bigg(|\tilde{g}^{-1}||\widetilde{\nabla}\widetilde{{\rm Rm}}||h|+|A||\widetilde{{\rm Rm}}|\\
&&+ \ |\nabla T|
+|\tilde{g}^{-1}||\widetilde{\nabla}{}^{2}\tilde{u}||h|+|w|
\widetilde{\nabla}{}^{2}\tilde{u}|+|\nabla u||y|\bigg)e^{-\eta}dV\\
&\leq&(N-\beta t^{-1})\mathcal{I}
+\int_{M}\bigg(C t^{-\beta}|\nabla T||A|+N t^{-\frac{1}{2}-\beta}|h||A|\\
&&+ \ N t^{-\beta}|h||A|+N t^{-\beta}|w||A|+N t^{-\beta}|A||y|\bigg)
e^{-\eta}dV\\
&\leq&(N-\beta t^{-1})\mathcal{I}+\left(C t^{-\beta}\mathcal{I}+\mathcal{J}\right)+\left(N\mathcal{H}
+C t^{-\beta}\mathcal{I}\right)\\
&&+ \ \left(N\mathcal{H}+t\mathcal{I}\right)
+\left(C t^{-\beta}\mathcal{I}+N\mathcal{B}\right)
+\left(N t^{-\beta}\mathcal{I}+\frac{1}{4}\mathcal{D}\right).
\end{eqnarray*}
Consequently
\begin{equation}
\mathcal{I}'\leq N\mathcal{H}+\left(N-\beta t^{-1}+N t^{-\beta}\right)
\mathcal{I}+\mathcal{J}+N\mathcal{B}+\frac{1}{4}\mathcal{D}.\label{6.34}
\end{equation}

{\bf (4) Estimate for $\mathcal{V}'$.} This estimate is similar to (1),
\begin{eqnarray*}
\mathcal{V}'&\leq&N\mathcal{V}+\int_{M}\left[2\langle\partial_{t}v,v\rangle
-|v|^{2}\partial_{t}\eta\right]e^{-\eta}dV\\
&\leq& N\mathcal{V}+\int_{M}
\left[2\langle\Delta v,v\rangle-|v|^{2}\partial_{t}\eta+C|v|\left(|\tilde{g}^{-1}||h|\widetilde{\nabla}{}^{2}
\tilde{u}|+|A||\widetilde{\nabla}\tilde{u}|\right)\right]e^{-\eta}dV\\
&\leq&N\mathcal{V}+\int_{M}\left[2\langle\Delta v,v\rangle
-|v|^{2}\partial_{t}\eta\right]e^{-\eta}dV+N\int_{M}
\left(|v||h|+|v||A|\right)e^{-\eta}dV\\
&\leq&N\mathcal{V}+t\mathcal{H}+t^{\beta}\mathcal{I}+\int_{M}\left[-2|\nabla v|^{2}+2|v||\nabla\eta||\nabla v|
-|v|^{2}\partial_{t}\eta\right]e^{-\eta}dV.
\end{eqnarray*}
Thus, by choosing an appropriate constant $B$ in {\color{blue}{(\ref{6.25})}},
\begin{equation}
\mathcal{V}'\leq N\mathcal{V}+t\mathcal{H}+t^{\beta}\mathcal{I}-\mathcal{B}\label{6.35}
\end{equation}

{\bf (5) Estimate for $\mathcal{B}'$.} Because the evolution equation
{\color{blue}{(\ref{A.4})}} is linear in $\nabla_{i}u$, we conclude that
\begin{equation}
\partial_{t}w=\Delta w+{\rm Rm}\ast w\label{6.36}
\end{equation}
and hence
\begin{eqnarray}
\mathcal{B}'&\leq&N\mathcal{B}+\int_{M}\left[2\langle\partial_{t}
w,w\rangle-|w|^{2}\partial_{t}\eta\right]e^{-\eta}dV\nonumber\\
&\leq&N\mathcal{B}+\int_{M}\left[2\langle\Delta w,w\rangle-|w|^{2}
\partial_{t}\eta\right]e^{-\eta}dV\label{6.37}\\
&\leq&N\mathcal{B}-\int_{M}|\nabla w|^{2}e^{-\eta}dV\nonumber
\end{eqnarray}
by choosing an appropriate constant $B$ in {\color{blue}{(\ref{6.25})}}. From {\color{blue}{(\ref{6.32})}} -- {\color{blue}{(\ref{6.37})}}, we arrive at
\begin{equation}
\mathcal{E}'\leq N\mathcal{E}-t^{-1}\left(\beta-t^{\beta}-N t^{1-\beta}\right)\mathcal{I}+\frac{1}{2}\mathcal{D}-\int_{M}|\nabla w|^{2}e^{-\eta}dV.\label{6.38}
\end{equation}
On the other hand, the equation {\color{blue}{(\ref{6.9})}} yields
\begin{eqnarray*}
\frac{1}{2}\mathcal{D}&=&\frac{1}{2}\int_{M}|y|^{2}e^{-\eta}dV \ \ \leq \ \ \int_{M}\left[|\nabla w|^{2}+N|A|^{2}|\widetilde{\nabla}\tilde{u}|^{2}\right]e^{-\eta}dV\\
&\leq&\int_{M}|\nabla w|^{2}e^{-\eta}dV+N t^{\beta}\mathcal{I}.
\end{eqnarray*}
Therefore, the inequality {\color{blue}{(\ref{6.38})}} can be rewritten as
\begin{equation}
\mathcal{E}'\leq N\mathcal{E}-t^{-1}\left(\beta-t^{\beta}-N t^{1-\beta}\right)\mathcal{I}.\label{6.39}
\end{equation}
Now we can choose appropriate constants $T_{0}$ and $N$ such that the term $\beta-t^{\beta}-N t^{1-\beta}$ is nonnegative for any $t\in[0,T_{0}]$. In this case, the inequality {\color{blue}{(\ref{6.43})}} gives us the desired estimate $\mathcal{E}'(t)\leq N\mathcal{E}(t)$ on $[0,T_{0}]$.
\end{proof}

{\it The proof of {\color{red}{Theorem \ref{t6.1}}}.} Now the proof immediately follows from the above {\color{red}{Proposition \ref{p6.4}}}.

\subsection{Backward uniqueness}\label{subsection6.2}

In this subsection we use the main idea in \cite{K2010} to prove the backward uniqueness of the Ricci-harmonic flow. Recall

\begin{theorem}\label{t6.5}{\bf (Kotschwar \cite{K2010})} Consider a
smooth family of complete Riemannian metrics $(g(t))_{t\in[0,T]}$ on a smooth
$n$-dimensional manifold $M$, satisfying the evolution equation
\begin{equation}
\partial_{t}g(t)=b(t),\label{6.40}
\end{equation}
and a symmetric, positive-definite family of $(2,0)$-tensor fields
$\Lambda(t)$, $t\in[0,T]$, with
\begin{equation}
\Box:=\partial_{t}-\Delta_{\Lambda(t),g(t)}, \ \ \
\Delta_{\Lambda,g(t)}:={\rm tr}_{\Lambda(t)}\nabla^{2}_{g(t)}.
\label{6.41}
\end{equation}
Let $\mathscr{X}, \mathscr{Y}$ be finite direct sums of the $(p,q)$-bundle $\mathcal{T}^{p}_{q}(M)$ over $M$, and ${\bf X}\in C^{\infty}(\mathscr{X}\times
[0,T])$, ${\bf Y}\in C^{\infty}(\mathscr{Y}\times [0,T])$. Suppose the following assumptions hold:

\begin{itemize}

\item[(1)] there exist positive constants $P, Q, \alpha_{1}, \alpha_{2}$ such that
\begin{eqnarray*}
|b(t)|^{2}_{g(t)}+|\nabla_{g(t)}b(t)|^{2}_{g(t)}&\leq& P,\\
|\partial_{t}\Lambda(t)|^{2}_{g(t)}
+|\nabla_{g(t)}\Lambda(t)|^{2}_{g(t)}&\leq& Q,\\
\alpha_{1}g^{-1}(t)&\leq&\Lambda(t) \ \ \leq \ \ \alpha_{2}
g^{-1}(t),
\end{eqnarray*}

\item[(2)] there exists a nonnegative constant $K$ such that ${\rm Ric}_{g(t)}\geq-K g(t)$,

\item[(3)] there exist positive constants $a, A$ and some point $x_{0}\in M$ such that
    \begin{equation}
    |{\bf X}(x,t)|^{2}_{g(t)}
    +|\nabla_{g(t)}{\bf X}(x,t)|^{2}_{g(t)}
    +|{\bf Y}(x,t)|^{2}_{g(t)}
    \leq A e^{a d_{g(t)}(x_{0},x)},\label{6.42}
    \end{equation}

\item[(4)] ${\bf X}$ and ${\bf Y}$ satisfy the inequality
\begin{eqnarray}
|\Box{\bf X}|^{2}_{g(t)}&\leq&C\left(|{\bf X}|^{2}_{g(t)}
+|\nabla_{g(t)}{\bf X}|^{2}_{g(t)}+|{\bf Y}|^{2}_{g(t)}
\right),\label{6.43}\\
|\partial_{t}{\bf Y}|^{2}_{g(t)}&\leq& C
\left(|{\bf X}|^{2}_{g(t)}
+|\nabla_{g(t)}{\bf Y}|^{2}_{g(t)}
+|{\bf Y}|^{2}_{g(t)}\right)\label{6.44}
\end{eqnarray}
for some positive constant $C$.
\end{itemize}
If ${\bf X}(T)={\bf Y}(T)=0$, then ${\bf X}
={\bf Y}\equiv0$ on $M\times[0,T]$.
\end{theorem}

\begin{theorem}\label{t6.6}{\bf (Backward uniqueness of the Ricci-harmonic flow)} Suppose that $(g(t), u(t))$ and $(\tilde{g}(t), \tilde{u}(t))$ are two smooth complete solutions of {\color{blue}{(\ref{6.1})}} satisfying {\color{blue}{(\ref{6.2})}} for some
uniform constant $K$. If $(g(T), u(T))=(\tilde{g}(T),
\tilde{u}(T))$, then $(g(t), u(t))\equiv(\tilde{g}(t), \tilde{u}(t))$
for each $t\in[0,T]$.
\end{theorem}

Recall notions in {\color{blue}{(\ref{6.3})}} -- {\color{blue}{(\ref{6.8})}},
\begin{eqnarray*}
h&:=&g-\tilde{g}, \ \ \ h_{ij} \ \ = \ \ g_{ij}
-\tilde{g}_{ij}, \ \ \ A \ \ := \ \ \nabla-\widetilde{\nabla}, \ \ \ A^{k}_{ij} \ \ = \ \
\Gamma^{k}_{ij}-\widetilde{\Gamma}^{k}_{ij},\\
T&:=&{\rm Rm}-\widetilde{{\rm Rm}}, \ \ \ T^{\ell}_{ijk} \ \ = \ \
R^{\ell}_{ijk}-\widetilde{R}^{\ell}_{ijk}, \ \ \ v \ \ := \ \ u-\tilde{u}, \ \ \ w \ \ := \ \ \nabla u-\widetilde{\nabla}\tilde{u}, \\
w_{i}&:=&\nabla_{i}v, \ \ \ y \ \ := \ \ \nabla^{2}u-\widetilde{\nabla}^{2}\tilde{u}, \ \ \
y_{ij} \ \ = \ \ \nabla_{i}\nabla_{j}u
-\widetilde{\nabla}_{i}\widetilde{\nabla}_{j}\tilde{u}.
\end{eqnarray*}
Define new tensor fields
\begin{eqnarray}
B&:=&\nabla A,\label{6.45}\\
U&:=&\nabla{\rm Rm}-\widetilde{\nabla}\widetilde{{\rm Rm}},
\label{6.46}\\
x&:=&\nabla w,\label{6.47}\\
z&:=&\nabla^{3}u-\widetilde{\nabla}^{3}\tilde{u}.\label{6.48}
\end{eqnarray}
Consider direct sums of tensor fields
\begin{equation*}
{\bf X}:=(T\oplus U)\oplus(y\oplus z), \ \ \
{\bf Y}:=(h\oplus A\oplus B)\oplus(v\oplus w\oplus x).
\end{equation*}
Using {\color{red}{Lemma \ref{lA.2}}}, we obtain
\begin{eqnarray*}
\partial_{t}\Gamma&=&g^{-1}\ast\nabla{\rm Rm}
+g^{-1}\ast\nabla u\ast\nabla^{2}u,\\
\partial_{t}{\rm Rm}&=&
\Delta{\rm Rm}+g^{-1}\ast{\rm Rm}\ast{\rm Rm}
+g^{-1}\ast{\rm Rm}\ast\nabla u\ast\nabla u
+g^{-1}\ast\nabla^{2}u\ast\nabla^{2}u,\\
\partial_{t}\nabla{\rm Rm}&=&
\Delta\nabla{\rm Rm}
+g^{-1}\ast{\rm Rm}\ast\nabla{\rm Rm}+g^{-1}\ast\nabla{\rm Rm}
\ast\nabla u\ast\nabla u\\
&&+ \ g^{-1}\ast{\rm Rm}\ast\nabla u
\ast\nabla^{2}u+g^{-1}\ast\nabla^{2}u\ast\nabla^{3}u.
\end{eqnarray*}
We also recall
\begin{equation*}
\tilde{g}^{ij}-g^{ij}=\tilde{g}^{ia}
g^{jb}h_{ab} \ \ \ \text{or} \ \ \
\tilde{g}^{-1}-g^{-1}=\tilde{g}^{-1}\ast g^{-1}\ast h
\end{equation*}
and
\begin{equation}
\nabla_{c}h_{ab}=A^{p}_{ca}\tilde{g}_{pb}
+A^{p}_{cb}\tilde{g}_{ap} \ \ \ \text{or} \ \ \
\nabla h=A\ast\tilde{g}.\label{6.49}
\end{equation}
The last identity follows from
\begin{eqnarray*}
\nabla_{v}h_{ab}&=&\partial_{v}h_{ab}-\Gamma^{p}_{ca}h_{pb}
-\Gamma^{p}_{cb}h_{ap}\\
&=&\partial_{c}g_{ab}
-\partial_{c}\tilde{g}_{ab}
-\Gamma^{p}_{ca}g_{pb}+\Gamma^{p}_{ca}
\tilde{g}_{pb}-\Gamma^{p}_{cb}g_{ap}
+\Gamma^{p}_{cb}\tilde{g}_{ap}\\
&=&\nabla_{c}g_{ab}
-\partial_{c}\tilde{g}_{ab}
+\left(A^{p}_{ca}+\widetilde{\Gamma}^{p}_{ca}
\right)\tilde{g}_{pb}
+\left(A^{p}_{cb}+\widetilde{\Gamma}^{p}_{cb}\right)
\tilde{g}_{ap}\\
&=&\nabla_{c}h_{ab}-\widetilde{\nabla}_{c}\tilde{g}_{ab}
+A^{p}_{ca}\tilde{g}_{pb}+A^{p}_{cb}\tilde{g}_{ap}.
\end{eqnarray*}

\begin{lemma}\label{l6.7} One has
\begin{eqnarray}
\partial_{t}h_{ij}&=&-2T^{\ell}_{\ell ij}\nonumber\\
&&+ \ 4\left(w_{i}w_{j}
+w_{i}\widetilde{\nabla}_{j}\tilde{u}
+w_{j}\widetilde{\nabla}_{i}\tilde{u}\right),\label{6.50}\\
\partial_{t}A^{k}_{ij}&=&\bigg[-g^{mk}\left(U^{p}_{ipjm}
+U^{p}_{jpim}-U^{p}_{mpij}\right)\nonumber\\
&&+ \ g^{kb}\tilde{g}^{ma}
h_{ab}\left(\widetilde{\nabla}_{i}\widetilde{R}_{jm}
+\widetilde{\nabla}_{j}\widetilde{R}_{im}
-\widetilde{\nabla}_{m}\widetilde{R}_{ij}\right)\bigg]+
4g^{mk}\nabla_{m}u y_{ij}\label{6.51}\\
&&+ \ 4g^{mk}w_{m}\widetilde{\nabla}_{i}
\widetilde{\nabla}_{j}\tilde{u}
-4\tilde{g}^{ma}g^{kb}h_{ab}
\widetilde{\nabla}_{m}\tilde{u}\widetilde{\nabla}_{i}\widetilde{\nabla}_{j}
\tilde{u},\nonumber\\
\partial_{t}B&=&\bigg[\nabla U+h\ast A\ast\tilde{g}^{-1}\ast\widetilde{\nabla}
\widetilde{{\rm Rm}}+A\ast\tilde{g}^{-1}
\ast\widetilde{\nabla}\widetilde{{\rm Rm}}\nonumber\\
&&+ \ h\ast\tilde{g}^{-1}\ast\widetilde{\nabla}{}^{2}
\widetilde{{\rm Rm}}+A\ast U+A\ast\widetilde{\nabla}\widetilde{{\rm Rm}}
\bigg]+\tilde{g}^{-1}\ast h\ast\widetilde{\nabla}\tilde{u}
\ast\widetilde{\nabla}{}^{2}\tilde{u}\nonumber\\
&&+ \ \tilde{g}^{-1}\ast A\ast\widetilde{\nabla}\tilde{u}
\ast\widetilde{\nabla}{}^{2}\tilde{u}+\tilde{g}^{-1}\ast h\ast\widetilde{\nabla}{}^{2}
\tilde{u}\ast\widetilde{\nabla}{}^{2}\tilde{u}
+\nabla u\ast\nabla^{2}u\ast A\label{6.52}\\
&&+ \ \tilde{g}^{-1}\ast h\ast A\ast\widetilde{\nabla}
\tilde{u}\ast\widetilde{\nabla}{}^{2}
\tilde{u}+\tilde{g}^{-1}\ast h\ast\widetilde{\nabla}\tilde{u}
\ast\widetilde{\nabla}{}^{3}\tilde{u}\nonumber\\
&&+ \ \nabla^{2}u\ast y+\nabla u\ast\nabla y+x\ast
\widetilde{\nabla}{}^{2}\tilde{u}
+w\ast\widetilde{\nabla}{}^{3}\tilde{u}+A\ast w\ast\widetilde{\nabla}{}^{2}
\tilde{u},\nonumber\\
\Box T&=&\bigg[h\ast\tilde{g}^{-1}\ast\widetilde{\nabla}{}^{2}
\widetilde{{\rm Rm}}+A\ast\widetilde{\nabla}\widetilde{{\rm Rm}}
+T\ast\widetilde{{\rm Rm}}
+T\ast T+A\ast A\ast\widetilde{{\rm Rm}}\nonumber\\
&&+ \ B\ast\widetilde{{\rm Rm}}+ h\ast\tilde{g}^{-1}\ast\widetilde{{\rm Rm}}
\ast\widetilde{{\rm Rm}}\bigg]
+h\ast\tilde{g}^{-1}
\ast\widetilde{{\rm Rm}}\ast\widetilde{\nabla}\tilde{u}\ast\widetilde{\nabla}
\tilde{u}\nonumber\\
&&+ \ \widetilde{{\rm Rm}}
\ast w\ast w+\widetilde{{\rm Rm}}\ast w\ast
\widetilde{\nabla}\tilde{u}
+T\ast\widetilde{\nabla}\tilde{u}
\ast\widetilde{\nabla}\tilde{u}+T\ast w\ast\widetilde{\nabla}\tilde{u}
\label{6.53}\\
&&+ \ T\ast w\ast w+y\ast y+y\ast\widetilde{\nabla}{}^{2}
\tilde{u}+h\ast g^{-1}\ast\widetilde{\nabla}{}^{2}
\tilde{u}\ast\widetilde{\nabla}{}^{2}
\tilde{u},\nonumber\\
\Box U&=&\bigg[h\ast\tilde{g}^{-1}\ast\widetilde{\nabla}{}^{3}
\widetilde{{\rm Rm}}+A\ast\widetilde{\nabla}{}^{2}\widetilde{{\rm Rm}}
+A\ast A\ast\widetilde{\nabla}\widetilde{{\rm Rm}}
+B\ast\widetilde{\nabla}\widetilde{{\rm Rm}}\nonumber\\
&&+ \ h\ast\tilde{g}^{-1}\ast\widetilde{{\rm Rm}}\ast\widetilde{\nabla}
\widetilde{{\rm Rm}}+U\ast\widetilde{{\rm Rm}}
+T\ast\widetilde{\nabla}\widetilde{{\rm Rm}}
+T\ast U\bigg]\nonumber\\
&&+ \ h\ast\tilde{g}^{-1}\ast\widetilde{\nabla}\widetilde{{\rm Rm}}
\ast\widetilde{\nabla}\tilde{u}
\ast\widetilde{\nabla}\tilde{u}+U\ast w\ast w
+w\ast w\ast\widetilde{\nabla}\widetilde{{\rm Rm}}\nonumber\\
&&+ \ U\ast w\ast\widetilde{\nabla}\tilde{u}
+\widetilde{\nabla}\widetilde{{\rm Rm}}\ast
w\ast\widetilde{\nabla}\tilde{u}+U\ast\widetilde{\nabla}
\tilde{u}\ast\widetilde{\nabla}\tilde{u}\label{6.54}\\
&&+ \ h\ast\tilde{g}^{-1}\ast\widetilde{{\rm Rm}}\ast
\widetilde{\nabla}\tilde{u}
\ast\widetilde{\nabla}{}^{2}\tilde{u}\nonumber+T\ast\widetilde{\nabla}\tilde{u}\ast y
+\widetilde{{\rm Rm}}\ast\widetilde{\nabla}\tilde{u}
\ast y\\
&&+ \ T\ast w\ast\widetilde{\nabla}{}^{2}
\tilde{u}+\widetilde{{\rm Rm}}\ast w\ast\widetilde{\nabla}{}^{2}
\tilde{u}+T\ast\widetilde{\nabla}\tilde{u}
\ast\widetilde{\nabla}{}^{2}\tilde{u}+T\ast w\ast y
\nonumber\\
&&+ \ \widetilde{{\rm Rm}}\ast w\ast y
+h\ast\tilde{g}^{-1}\ast\widetilde{\nabla}{}^{2}
\tilde{u}\ast\widetilde{\nabla}{}^{3}\tilde{u}
+y\ast z+y\ast\widetilde{\nabla}{}^{3}\tilde{u}
+z\ast\widetilde{\nabla}{}^{2}\tilde{u}.\nonumber
\end{eqnarray}
\end{lemma}

\begin{proof} For $h_{ij}$, one has
\begin{eqnarray*}
\partial_{t}h_{ij}&=&\partial_{t}g_{ij}-\partial_{t}\tilde{g}_{ij} \ \
= \ \ -2R_{ij}+4\nabla_{i}u\nabla_{j}u
+2\widetilde{R}_{ij}-4\widetilde{\nabla}_{i}\tilde{u}\widetilde{\nabla}_{j}
\tilde{u}\\
&=&-2T^{\ell}_{\ell ij}
+4\left(w_{i}+\widetilde{\nabla}_{i}\tilde{u}\right)
\left(w_{j}+\widetilde{\nabla}_{j}\tilde{u}\right)
-4\widetilde{\nabla}_{i}\tilde{u}\widetilde{\nabla}_{j}\tilde{u}
\end{eqnarray*}
which implies {\color{blue}{(\ref{6.50})}}.

In {\color{blue}{(\ref{6.51})}}, the terms enclosed in the bracket were derived in
\cite{K2010}. The remaining terms are
\begin{equation*}
4g^{mk}\nabla_{m}u\nabla_{i}\nabla_{j}u
-4\tilde{g}^{mk}\widetilde{\nabla}_{m}\tilde{u}
\widetilde{\nabla}_{i}\widetilde{\nabla}_{j}\tilde{u}
\end{equation*}
which, using $y_{ij}=\nabla_{i}\nabla_{j}u
-\widetilde{\nabla}_{i}\widetilde{\nabla}_{j}\tilde{u}$, $w_{m}
=\nabla_{m}y-\widetilde{\nabla}_{m}
\tilde{u}$, and $\tilde{g}^{mk}
-g^{mk}=\tilde{g}^{ma}g^{kb}h_{ab}$, is equal to
\begin{eqnarray*}
&&g^{mk}\nabla_{m}u\left(y_{ij}+\widetilde{\nabla}_{i}
\widetilde{\nabla}_{j}\tilde{u}\right)
-\tilde{g}^{mk}\widetilde{\nabla}_{m}\tilde{u}
\widetilde{\nabla}_{i}\widetilde{\nabla}_{j}\tilde{u}\\
&=&g^{mk}\nabla_{m}u y_{ij}
+\left(g^{mk}\nabla_{m}u-\tilde{g}^{mk}
\widetilde{\nabla}_{m}\tilde{u}\right)\widetilde{\nabla}_{i}
\widetilde{\nabla}_{j}\tilde{u}\\
&=&g^{mk}\nabla_{m}u y_{ij}
+\left[g^{mk}\left(\nabla_{m}u-\widetilde{\nabla}_{m}
\tilde{u}\right)-\left(\tilde{g}^{mk}
-g^{mk}\right)\widetilde{\nabla}_{m}\tilde{u}
\right]\widetilde{\nabla}_{i}\widetilde{\nabla}_{j}
\tilde{u}\\
&=&g^{mk}\nabla_{m}u y_{ij}
+g^{mk}w_{m}\widetilde{\nabla}_{i}\widetilde{\nabla}_{j}
\tilde{u}-\tilde{g}^{ma}g^{kb}h_{ab}
\widetilde{\nabla}_{m}\tilde{u}
\widetilde{\nabla}_{i}\widetilde{\nabla}_{j}
\tilde{u}.
\end{eqnarray*}
In particular, {\color{blue}{(\ref{6.51})}} implies
\begin{eqnarray*}
\partial_{t}A&=&\left(g^{-1}\ast U+g^{-1}\ast\tilde{g}^{-1}
\ast h\ast\widetilde{\nabla}\widetilde{{\rm Rm}}\right).
\end{eqnarray*}
According to the relation $\partial_{t}B=\partial_{t}\nabla A
=\nabla\partial_{t}A+\partial_{t}\Gamma\ast A$, we obtain {\color{blue}{(\ref{6.52})}} where the bracket follows from \cite{K2010} and the remaining terms are
$$
g^{-1}\ast\nabla u\ast\nabla^{2}u\ast A
+g^{-1}\ast\nabla^{2}u\ast y
+g^{-1}\ast\nabla u\ast\nabla y+g^{-1}
\ast\nabla w\ast\widetilde{\nabla}{}^{2}\tilde{u}
$$
$$
+ \ \tilde{g}^{-1}\ast w\ast\nabla\widetilde{\nabla}{}^{2}
\tilde{u}+g^{-1}\ast\nabla\tilde{g}^{-1}
\ast h\ast\widetilde{\nabla}\tilde{u}
\ast\widetilde{\nabla}{}^{2}\tilde{u}
+g^{-1}\ast\tilde{g}^{-1}\ast\nabla h
\ast\widetilde{\nabla}\tilde{u}
\ast\widetilde{\nabla}{}^{2}\tilde{u}
$$
$$
+ \ g^{-1}\ast\tilde{g}^{-1}\ast h
\ast\nabla\widetilde{\nabla}\tilde{u}
\ast\widetilde{\nabla}{}^{2}
\tilde{u}+g^{-1}\ast\tilde{g}^{-1}
\ast h\ast\widetilde{\nabla}\tilde{u}
\ast\nabla\widetilde{\nabla}{}^{2}
\tilde{u}.
$$
Applying the formula
\begin{equation}
\nabla W=\widetilde{\nabla}W+A\ast W
\end{equation}
for any tensor field $W$, $\nabla h=\tilde{g}\ast A$, and $\nabla\tilde{g}^{-1}
=\tilde{g}^{-1}\ast A$, which follows from
\begin{equation*}
\nabla_{k}\tilde{g}^{ij}
=\tilde{g}^{ia}\tilde{g}^{jb}
\nabla_{k}h_{ab}
=\tilde{g}^{ia}
\tilde{g}^{jb}
\left(A^{p}_{ka}\tilde{g}^{pb}
+A^{p}_{kb}\tilde{g}_{ap}\right)
=\tilde{g}^{ia}
A^{j}_{ka}
+\tilde{g}^{jb}A^{i}_{kb},
\end{equation*}
we complete the proof of {\color{blue}{(\ref{6.52})}}.

To prove the last two identities, recall from \cite{K2010} that
\begin{equation}
\widetilde{\nabla}{}^{2}W
=\nabla^{2}W+A\ast\widetilde{\nabla}W
+B\ast W+A\ast A\ast W
\end{equation}
for any tensor field $W$. In particular,
\begin{eqnarray*}
\widetilde{\Delta}\widetilde{{\rm Rm}}
&=&\tilde{g}^{-1}\ast h\ast\widetilde{\nabla}{}^{2}
\widetilde{{\rm Rm}}
+\Delta\widetilde{{\rm Rm}}
+A\ast\widetilde{\nabla}\widetilde{{\rm Rm}}
+B\ast\widetilde{{\rm Rm}}+A\ast A\ast\widetilde{{\rm Rm}},\\
\widetilde{\Delta}\widetilde{\nabla}
\widetilde{{\rm Rm}}&=&
\tilde{g}^{-1}\ast h\ast
\widetilde{\nabla}{}^{3}\widetilde{{\rm Rm}}
+\Delta\widetilde{\nabla}\widetilde{{\rm Rm}}
+A\ast\widetilde{\nabla}{}^{2}
\widetilde{{\rm Rm}}
+B\ast\widetilde{\nabla}\widetilde{{\rm Rm}}
+A\ast A\ast\widetilde{\nabla}\widetilde{{\rm Rm}}.
\end{eqnarray*}
For {\color{blue}{(\ref{6.53})}}, we have
$$
\Box T \ \ = \ \ \left(\partial_{t}-\Delta
\right)\left({\rm Rm}-\widetilde{{\rm Rm}}\right)
$$
$$
= \ \ g^{-1}\ast{\rm Rm}
\ast{\rm Rm}+g^{-1}\ast{\rm Rm}\ast\nabla u
\ast\nabla u+g^{-1}\ast\nabla^{2}u
\ast\nabla^{2}u-\partial_{t}\widetilde{{\rm Rm}}
+\Delta\widetilde{{\rm Rm}}
$$
$$
= \ \ g^{-1}
\ast{\rm Rm}\ast{\rm Rm}
+g^{-1}\ast{\rm Rm}\ast\nabla u\ast\nabla u
+g^{-1}\ast\nabla^{2}u\ast\nabla^{2}u
$$
$$
- \
\left[\widetilde{\Delta}\widetilde{{\rm Rm}}
+\tilde{g}^{-1}\ast\widetilde{{\rm Rm}}
\ast\widetilde{{\rm Rm}}+\tilde{g}^{-1}
\ast\widetilde{{\rm Rm}}\ast\widetilde{\nabla}
\tilde{u}\ast\widetilde{\nabla}\tilde{u}
+\tilde{g}^{-1}\ast\widetilde{\nabla}{}^{2}
\tilde{u}\ast\widetilde{\nabla}{}^{2}
\tilde{u}\right]
$$
$$
+ \ \left[\widetilde{\Delta}
\widetilde{{\rm Rm}}
+h\ast\tilde{g}^{-1}\ast\widetilde{\nabla}{}^{2}
\widetilde{{\rm Rm}}
+A\ast\widetilde{\nabla}\widetilde{{\rm Rm}}
+B\ast\widetilde{{\rm Rm}}
+A\ast A\ast\widetilde{{\rm Rm}}
\right]
$$
$$
= \ \
h\ast\tilde{g}^{-1}\ast\widetilde{\nabla}{}^{2}\widetilde{{\rm Rm}}
+A\ast\widetilde{\nabla}\widetilde{{\rm Rm}}
+B\ast\widetilde{{\rm Rm}}+A\ast A\ast\widetilde{{\rm Rm}}
$$
$$
+ \ g^{-1}\ast\left(T+\widetilde{{\rm Rm}}
\right)\ast\left(T+\widetilde{{\rm Rm}}\right)
-\tilde{g}^{-1}\ast\widetilde{{\rm Rm}}
\ast\widetilde{{\rm Rm}}
$$
$$
+ \ g^{-1}\ast\left(T+\widetilde{{\rm Rm}}
\right)\ast\left(w+\widetilde{\nabla}\tilde{u}
\right)\ast\left(w+\widetilde{\nabla}\tilde{u}
\right)
-\tilde{g}^{-1}\ast\widetilde{{\rm Rm}}
\ast\widetilde{\nabla}\tilde{u}\ast\widetilde{\nabla}
\tilde{u}
$$
$$
+ \ \tilde{g}^{-1}\ast\left(y+\widetilde{\nabla}{}^{2}
\tilde{u}\right)
\ast\left(y+\widetilde{\nabla}{}^{2}
\tilde{u}\right)
-\tilde{g}^{-1}\ast\widetilde{\nabla}{}^{2}
\tilde{u}\ast\widetilde{\nabla}{}^{2}
\tilde{u}.
$$
Simplifying terms gives {\color{blue}{(\ref{6.53})}}. Similarly,
$$
\Box U \ \ = \ \ h\ast\tilde{g}^{-1}\ast\widetilde{\nabla}{}^{3}
\widetilde{{\rm Rm}}
+A\ast\widetilde{\nabla}{}^{2}\widetilde{{\rm Rm}}
+B\ast\widetilde{\nabla}\widetilde{{\rm Rm}}
+A\ast A\ast\widetilde{\nabla}\widetilde{{\rm Rm}}
$$
$$
+\left[g^{-1}\ast\left(T+\widetilde{{\rm Rm}}
\right)\ast\left(U+\widetilde{\nabla}\widetilde{{\rm Rm}}
\right)-\tilde{g}^{-1}\ast\widetilde{{\rm Rm}}
\ast\widetilde{\nabla}\widetilde{{\rm Rm}}
\right]
$$
$$
+ \ \left[g^{-1}\ast
\left(U+\widetilde{\nabla}\widetilde{{\rm Rm}}\right)
\ast\left(w+\widetilde{\nabla}\tilde{u}
\right)\ast\left(w+\widetilde{\nabla}\tilde{u}
\right)-\tilde{g}^{-1}\ast\widetilde{\nabla}\widetilde{{\rm Rm}}
\ast\widetilde{\nabla}\tilde{u}\ast\widetilde{\nabla}
\tilde{u}\right]
$$
$$
+ \ \left[g^{-1}\ast\left(T
+\widetilde{{\rm Rm}}\right)
\ast\left(w+\widetilde{\nabla}\tilde{u}
\right)\ast\left(y+\widetilde{\nabla}{}^{2}
\tilde{u}\right)
-\tilde{g}^{-1}\ast\widetilde{{\rm Rm}}
\ast\widetilde{\nabla}\tilde{u}
\ast\widetilde{\nabla}{}^{2}\tilde{u}
\right]
$$
$$
+ \ \left[g^{-1}
\ast\left(y+\widetilde{\nabla}{}^{2}
\tilde{u}\right)\ast
\left(z+\widetilde{\nabla}{}^{3}
\tilde{u}\right)
-\tilde{g}^{-1}\ast\widetilde{\nabla}{}^{2}
\tilde{u}
\ast\widetilde{\nabla}{}^{3}
\tilde{u}\right].
$$
Simplifying terms gives {\color{blue}{(\ref{6.54})}}.
\end{proof}

According to {\color{red}{Theorem \ref{tB.2}}}, the condition $|{\rm Rm}|_{g(t)}
+|\widetilde{{\rm Rm}}|_{\tilde{g}(t)}\leq K$ implies that, for all
$m\geq0$, there exist constants $C_{m}=C_{m}(\delta, K, n, T)>0$ such that
\begin{equation}
\left|\nabla^{m}{\rm Rm}\right|
+\left|\nabla^{m+1}u\right|
+\left|\widetilde{\nabla}{}^{m}\widetilde{{\rm Rm}}
\right|_{\tilde{g}(t)}
+\left|\widetilde{\nabla}{}^{m+1}\tilde{u}
\right|_{\tilde{g}(t)}\leq C_{m}\label{6.57}
\end{equation}
on $M\times[0,T]$. We also have
\begin{equation*}
\frac{1}{\gamma}g(t)\leq\tilde{g}(t)\leq\gamma\!\ g(t)
\end{equation*}
on $M\times[0,T]$, for some positive constant $\gamma=\gamma(K, T)$. Hence $\nabla^{m}{\rm Rm}, \nabla^{m+1}u$, $\widetilde{\nabla}^{m}\widetilde{{\rm Rm}}, \widetilde{\nabla}^{m+1}\tilde{u}$, $m\geq0$, and $\tilde{g}^{-1}$
are uniformly bounded with respect to $g(t)$ on $[0,T]$ so that we can
replace the norm $|\cdot|_{\tilde{g}(t)}$ by $|\cdot|:=|\cdot|_{g(t)}$.

\begin{lemma}\label{l6.8} $h, A, B, T, U$ are uniformly bounded with respect to $g(t)$
on $[0, T]$. Moreover, $v, w, x, y, z$ are also uniformly bounded with respect to $g(t)$ on $[0,T]$.
\end{lemma}

\begin{proof} The first part was proved in \cite{K2010} in the exact
manner. The second part follows immediately from {\color{red}{Lemma \ref{lA.1}}}.
\end{proof}

\begin{lemma}\label{l6.9} Using the above lemma, one has
\begin{eqnarray*}
|\partial_{t}h|^{2}&\lesssim&
|T|^{2}+|w|^{2},\\
|\partial_{t}A|^{2}&\lesssim&
|U|^{2}+|h|^{2}
+|y|^{2}+|w|^{2},\\
|\partial_{t}B|^{2}&\lesssim&
|\nabla U|^{2}+|h|^{2}
+|A|^{2}+|U|^{2}+|y|^{2}+|\nabla y|^{2}
+|x|^{2}+|w|^{2},\\
|\partial_{t}v|^{2}&\lesssim&
|y|^{2}+|h|^{2},\\
|\partial_{t}w|^{2}&\lesssim&|z|^{2}
+|A|^{2}
+|h|^{2}+|v|^{2},\\
|\partial_{t}x|^{2}&\lesssim&
|\nabla z|^{2}
+|B|^{2}+|A|^{2}
+|h|^{2}+|v|^{2}+|w|^{2},\\
|\Box y|^{2}&\lesssim&|h|^{2}
+|A|^{2}+|B|^{2}
+|T|^{2}+|w|^{2}
+|y|^{2},\\
|\Box z|^{2}&\lesssim&|h|^{2}
+|A|^{2}+|B|^{2}+|T|^{2}+|z|^{2}
+|U|^{2}+|y|^{2}+|w|^{2}.
\end{eqnarray*}
\end{lemma}

\begin{proof} The first four inequality follows from {\color{red}{Lemma \ref{l6.9}}}. For the next three inequalities, we verify only the inequality for $|\partial_{t}v|^{2}$. By definition,
\begin{eqnarray*}
\partial_{t}v&=&\Delta u-\widetilde{\Delta}\tilde{u} \ \
= \ \ g^{ij}\nabla_{i}\nabla_{j}u-\tilde{g}^{ij}\widetilde{\nabla}_{i}
\widetilde{\nabla}_{j}\tilde{u}\\
&=&g^{ij}\left(\nabla_{i}\nabla_{j}u
-\widetilde{\nabla}_{i}\widetilde{\nabla}_{j}
\tilde{u}\right)+(g^{ij}-\tilde{g}^{ij})
\widetilde{\nabla}_{i}\widetilde{\nabla}_{j}
\tilde{u} \ \ = \ \ g^{ij}y_{ij}-\tilde{g}^{ia}g^{jb}h_{ab}
\widetilde{\nabla}_{i}\widetilde{\nabla}_{j}
\tilde{u}
\end{eqnarray*}
so that $\partial_{t}v=g^{-1}\ast y+\tilde{g}^{-1}\ast g^{-1}
\ast h\ast\widetilde{\nabla}^{2}\tilde{u}$. Similarly
\begin{eqnarray*}
\partial_{t}w&=&g^{-1}\ast\nabla y
+\tilde{g}^{-1}\ast A\ast h\ast\widetilde{\nabla}{}^{2}
\tilde{u}
+\tilde{g}^{-1}\ast\tilde{g}\ast A\ast\widetilde{\nabla}{}^{2}
\tilde{u}\\
&&+ \ \tilde{g}^{-1}\ast h\ast
\left(\widetilde{\nabla}{}^{3}\tilde{u}
+A\ast\widetilde{\nabla}{}^{2}
\tilde{u}\right)
\end{eqnarray*}
with $\nabla y=\nabla(\nabla^{2}u-\widetilde{\nabla}{}^{2}
\tilde{u})=z+A\ast\widetilde{\nabla}{}^{2}
\tilde{u}$. For the final two inequalities we only verify the inequality
for $|\Box y|^{2}$. From the identity
\begin{eqnarray*}
\widetilde{\Delta}\widetilde{\nabla}{}^{2}\tilde{u}&=&
\tilde{g}^{-1}\widetilde{\nabla}{}^{2}
\widetilde{\nabla}{}^{2}\tilde{u}\\
&=&\Delta\widetilde{\nabla}{}^{2}\tilde{u}
+h\ast\tilde{g}^{-1}\ast\widetilde{\nabla}{}^{4}\tilde{u}
+A\ast\widetilde{\nabla}{}^{3}\tilde{u}
+B\ast\widetilde{\nabla}{}^{2}\tilde{u}
+A\ast A\ast\widetilde{\nabla}{}^{2}\tilde{u},
\end{eqnarray*}
and
\begin{equation*}
\partial_{t}\nabla^{2}u
=\Delta\nabla^{2}u
+{\rm Rm}\ast\nabla^{2}u
+g^{-1}\ast\nabla u\ast\nabla u\ast\nabla^{2}u,
\end{equation*}
we obtain
\begin{eqnarray*}
\Box y&=&\left[{\rm Rm}\ast\nabla^{2}u
+g^{-1}\ast\nabla u\ast\nabla u\ast\nabla^{2}u
\right]-\partial_{t}\widetilde{\nabla}{}^{2}
\tilde{u}+\Delta\widetilde{\nabla}{}^{2}\tilde{u}\\
&=&h\ast\tilde{g}^{-1}\ast\widetilde{\nabla}{}^{4}
\tilde{u}+A\ast\widetilde{\nabla}{}^{3}
\tilde{u}+B\ast\widetilde{\nabla}{}^{2}
\tilde{u}+A\ast A\ast\widetilde{\nabla}{}^{2}\tilde{u}\\
&&+ \ \left(T+\widetilde{{\rm Rm}}
\right)\ast\nabla^{2}u-\widetilde{{\rm Rm}}\ast\widetilde{\nabla}{}^{2}
\tilde{u}\\
&&+ \ g^{-1}\ast\left(w+\widetilde{\nabla}
\tilde{u}\right)\ast\left(w+\widetilde{\nabla}\tilde{u}
\right)\ast\left(y+\widetilde{\nabla}{}^{2}
\tilde{u}\right)
-\tilde{g}^{-1}\ast\widetilde{\nabla}
\tilde{u}\ast\widetilde{\nabla}\tilde{u}
\ast\widetilde{\nabla}{}^{2}\tilde{u}\\
&=&h\ast\tilde{g}^{-1}\ast\widetilde{\nabla}{}^{4}
\tilde{u}
+A\ast\widetilde{\nabla}{}^{3}\tilde{u}
+B\ast\widetilde{\nabla}{}^{2}\tilde{u}
+A\ast A\ast A\ast\widetilde{\nabla}{}^{2}\tilde{u}
+T\ast\nabla^{2}u\\
&&+ \ w\ast w\ast y
+\widetilde{\nabla}\tilde{u}\ast w\ast y
+\widetilde{\nabla}\tilde{u}\ast\widetilde{\nabla}
\tilde{u}\ast y+w\ast w\ast\widetilde{\nabla}{}^{2}
\tilde{u}+w\ast\widetilde{\nabla}\tilde{u}
\ast\widetilde{\nabla}{}^{2}\tilde{u}\\
&&+ \ h\ast\tilde{g}^{-1}\ast\widetilde{\nabla}
\tilde{u}\ast\widetilde{\nabla}\tilde{u}
\ast\widetilde{\nabla}{}^{2}\tilde{u}.
\end{eqnarray*}
For $|\Box z|^{2}$, we need the evolution equation
\begin{eqnarray*}
\partial_{t}\nabla^{3}u&=&\Delta\nabla^{3}u
+g^{-1}\ast{\rm Rm}\ast\nabla^{3}u
+g^{-1}\ast\nabla{\rm Rm}\ast\nabla^{2}u\\
&&+ \ g^{-1}\ast\nabla u\ast\nabla^{2}u
\ast\nabla^{2}u+g^{-1}\ast\nabla u
\ast\nabla u\ast\nabla^{3}u
\end{eqnarray*}
and the identity
\begin{equation*}
\widetilde{\Delta}\widetilde{\nabla}{}^{3}\tilde{u}
=\Delta\widetilde{\nabla}{}^{3}u
+h\ast\tilde{g}^{-2}\ast\widetilde{\nabla}{}^{5}
u+A\ast\widetilde{\nabla}{}^{4}u
+B\ast\widetilde{\nabla}{}^{3}u+A\ast A\ast
\widetilde{\nabla}{}^{3}u.
\end{equation*}
Thus we prove the results.
\end{proof}

{\it The proof of {\color{red}{Theorem \ref{t6.6}}}.} The above two lemmas, {\color{red}{Lemma \ref{l6.8}}} and
{\color{red}{Lemma \ref{l6.9}}}, imply
\begin{eqnarray*}
|\Box T|^{2}&\lesssim&|h|^{2}
+|A|^{2}+|T|^{2}
+|B|^{2}
+|w|^{2}+|y|^{2},\\
|\Box U|^{2}&\lesssim&|h|^{2}
+|A|^{2}+|B|^{2}
+|U|^{2}+|T|^{2}
+|w|^{2}+|y|^{2}+|z|^{2}.
\end{eqnarray*}
Together {\color{red}{Lemma \ref{l6.9}}}, we have
\begin{eqnarray*}
|\Box{\bf X}|^{2}_{g(t)}&\lesssim&|{\bf X}|^{2}_{g(t)}
+|{\bf Y}|^{2}_{g(t)},\\
|\partial_{t}{\bf Y}|^{2}_{g(t)}
&\lesssim&|{\bf X}|^{2}_{g(t)}
+|{\bf Y}|^{2}_{g(t)}
+|\nabla {\bf X}|^{2}_{g(t)}.
\end{eqnarray*}
To apply {\color{red}{Theorem \ref{t6.5}}}, we need to check the boundedness of ${\bf X}, \nabla{\bf X}$ and ${\bf Y}$ on $[0, T]$, which are however,
followed from {\color{red}{Lemma \ref{lA.1}}}. Therefore, ${\bf X}={\bf Y}\equiv0$ on $
[0,T]$. Thus, $(g(t), u(t))=(\tilde{g}(t), \tilde{u}(t))\equiv0$ on $[0,T]$.

\appendix

\section{Evolution equations of the Ricci-harmonic flow}\label{sectionA}

We review some basic evolution equations of the Ricci-harmonic flow. Consider a Ricci-harmonic flow
\begin{equation}
\partial_{t}g(t)=-2\!\ {\rm Ric}_{g(t)}+
4\!\ du(t)\otimes du(t), \ \ \ \partial_{t}u(t)=\Delta_{g(t)}
u(t)\label{A.1}
\end{equation}
on a smooth manifold $M$. As before, we follow the convention
in Section \ref{section1}.

\begin{lemma}\label{lA.1} Under the flow {\color{blue}{(\ref{A.1})}}, we have
\begin{eqnarray}
\Box R_{ij}&=&-2R_{ik}R^{k}{}_{j}+2R_{pijq}R^{pq}
-4R_{pijq}\nabla^{p}u\nabla^{q}u\nonumber\\
&&+ \ 4\Delta u\!\ \nabla_{i}\nabla_{j}u
-4\nabla_{i}\nabla_{k}u\nabla^{k}\nabla_{j}u,\label{A.2}\\
\partial_{t}\Gamma^{k}_{ij}&=&g^{k\ell}
\left(-\nabla_{i}R_{j\ell}-\nabla_{j}R_{i\ell}
+\nabla_{\ell}R_{ij}+4\nabla_{i}\nabla_{j}u
\nabla_{\ell}u\right),\label{A.3}\\
\Box\partial_{i}u&=&-R_{ij}\nabla^{j}u,\label{A.4}\\
\Box\nabla_{i}\nabla_{j}u
&=&2R_{pijq}\nabla^{p}\nabla^{q}u
-R_{ip}\nabla_{j}\nabla^{p}u
-R_{jp}\nabla_{i}\nabla^{p}u
-4|\nabla u|^{2}\nabla_{i}\nabla_{j}u,\label{A.5}\\
\Box R_{ijk\ell}&=&2\left(B_{ijk\ell}-B_{ij\ell k}
-B_{i\ell jk}+B_{ikj\ell}\right)\nonumber\\
&&- \ \left(R_{i}{}^{p}R_{pjk\ell}+R_{j}{}^{p}R_{ijp\ell}
+R_{\ell}{}^{p}R_{ijkp}\right)\label{A.6}\\
&&+ \ 4\left(\nabla_{i}\nabla_{\ell}u\nabla_{j}\nabla_{k}u
-\nabla_{i}\nabla_{k}u\nabla_{j}\nabla_{\ell}u\right),\nonumber
\end{eqnarray}
where $B_{ijk\ell}:=-g^{pr}g^{qs}R_{ipjq}R_{kr\ell s}$, and
\begin{eqnarray}
\Box R&=&2|{\rm Ric}|^{2}
+4|\Delta u|^{2}-4|\nabla^{2}u|^{2}
-8\!\ {\rm Ric}(\nabla u,\nabla u),\label{A.7}\\
\Box|\nabla u|^{2}&=&-2|\nabla^{2}u|^{2}
-4|\nabla u|^{4},\label{A.8}\\
\Box\left(R-2|\nabla u|^{2}\right)&=&2\left|{\rm Ric}
-2\nabla u\otimes\nabla u\right|^{2}+4|\Delta u|^{2},\label{A.9}\\
\partial_{t}dV&=&-\left(R-2|\nabla u|^{2}\right)dV.\label{A.10}
\end{eqnarray}
\end{lemma}

\begin{proof} See for example \cite{LY1, List2005, List2008, Muller2009, Muller2012}. Note that in our notation for $R^{\ell}_{ijk}$ defined by
\begin{equation*}
R^{\ell}_{ijk}=\partial_{i}\Gamma^{\ell}_{jk}
-\partial_{j}\Gamma^{\ell}_{ik}+\Gamma^{p}_{jk}\Gamma^{\ell}_{ip}
-\Gamma^{p}_{ik}\Gamma^{\ell}_{jp},
\end{equation*}
the tensors $R^{\ell}_{ijk}$ is the minus of those used in \cite{List2005}.
\end{proof}

\begin{lemma}\label{lA.2} As $(1,3)$-tensor, we have
\begin{equation}
\partial_{t}{\rm Rm}
=g^{-1}\nabla^{2}{\rm Ric}
+g^{-1}{\rm Ric}\ast{\rm Rm}
+g^{-1}\nabla^{2}u\ast\nabla^{2}u+g^{-1}{\rm Rm}\ast\nabla u
\ast\nabla u,\label{A.11}
\end{equation}
and
\begin{eqnarray}
\partial_{t}R^{\ell}_{ijk}
&=&\left[\Delta R^{\ell}_{ijk}
+g^{pq}\left(R^{r}_{ijp}R^{\ell}_{rqk}
-2R^{r}_{pik}R^{\ell}_{jqr}
+2R^{\ell}_{pir}R^{r}_{jqk}\right)
-g^{pq}\left(R_{ip}R^{\ell}_{qjk}\right.\right.\nonumber\\
&&+ \ \left.\left.R_{jp}R^{\ell}_{iqk}
+R_{kp}R^{\ell}_{ijq}\right)+g^{p\ell}R_{pq}
R^{q}_{ijk}\right]-4 g^{p\ell}R^{q}_{ijk}\nabla_{q}u\nabla_{p}u\label{A.12}\\
&&+ \ 4g^{p\ell}\left(\nabla_{i}\nabla_{p}u
\nabla_{k}\nabla_{j}u-\nabla_{i}\nabla_{k}u\nabla_{j}
\nabla_{p}u\right),\nonumber\\
\partial_{t}\nabla_{a}R^{\ell}_{ijk}
&=&\left[\Delta\nabla_{a}R^{\ell}_{ijk}
+g^{pq}\nabla_{a}
\left(R^{r}_{ijp}R^{\ell}_{rqk}
-2R^{r}_{pik}R^{\ell}_{jqr}+2R^{\ell}_{pir}R^{r}_{jqk}\right)\right.
\nonumber\\
&&- \ \left.g^{pq}\left(R_{ip}\nabla_{a}R^{\ell}_{qjk}
-R_{jp}\nabla_{a}R^{\ell}_{iqk}
-R_{kp}\nabla_{a}R^{\ell}_{ijq}\right)
+g^{p\ell}R_{pq}\nabla_{a}R^{q}_{ijk}\right]\nonumber\\
&&- \ 4 g^{p\ell}\nabla_{a}R^{q}_{ijk}\nabla_{q}u\nabla_{p}u
-4g^{p\ell}R^{p}_{ijk}\nabla_{a}\nabla_{q}u\nabla_{p}u
-4g^{p\ell}R^{q}_{ijk}\nabla_{q}u\nabla_{a}\nabla_{p}u\nonumber\\
&&+ \ 4g^{p\ell}\nabla_{a}\nabla_{i}\nabla_{p}u
\nabla_{k}\nabla_{j}u
+4g^{p\ell}\nabla_{i}\nabla_{p}u\nabla_{a}\nabla_{k}\nabla_{j}u\label{A.13}\\
&&- \ 4g^{p\ell}\nabla_{a}\nabla_{i}\nabla_{k}u\nabla_{j}\nabla_{p}u
-4g^{p\ell}\nabla_{i}\nabla_{k}u\nabla_{a}\nabla_{j}\nabla_{p}u\nonumber\\
&&+ \ 4g^{\ell k}R^{b}_{ijk}\nabla_{a}\nabla_{b}u\nabla_{k}u
-4g^{nk}\nabla_{a}\nabla_{i}u\nabla_{k}uR^{\ell}_{bjk}\nonumber\\
&&- \ 4g^{bk}\nabla_{a}\nabla_{i}u\nabla_{k}u
R^{\ell}_{ibk}
-4g^{b\ell}\nabla_{a}\nabla_{k}u\nabla_{\ell}uR^{\ell}_{ijb}.\nonumber
\end{eqnarray}
\end{lemma}

We also need the existence result for the Ricci-harmonic flow.

\begin{theorem}\label{tA.3}{\bf (List, 2005)} Let $(M, g_{0})$ be a smooth complete $n$-dimensional Riemannian manifold with bounded curvature $|{\rm Rm}_{g_{0}}|_{g_{0}}
\leq K_{0}$. Consider a smooth function $u_{0}$ on $M$ satisfying
\begin{equation}
|u_{0}|^{2}_{g_{0}}+|\nabla_{g_{0}}u_{0}|^{2}_{g_{0}}
\leq C_{0}, \ \ \ |\nabla^{2}_{g_{0}}u_{0}|^{2}_{g_{0}}
\leq C_{1}.\label{A.14}
\end{equation}
Here $K_{0}, C_{0}, C_{1}$ are some positive constants. Then there exists a positive constant $T:=T(n, K_{-}, C_{0})$ such that the
initial value problem {\color{blue}{(\ref{A.1})}} with $(g(0), u(0))=(g_{0}, u_{0})$ has
a smooth solution $(g(t), u(t))$ on $M\times[0,T]$. Moreover, the solution satisfies
\begin{equation}
\frac{1}{C}g_{0}\leq g(t)\leq C\!\ g_{0}, \ \ \ t\in[0,T],\label{A.15}
\end{equation}
for some constant $C:=C(n, K_{0}, C_{0}, C_{1})$, and on $M\times[0,T]$ there is a bound
\begin{equation}
|{\rm Rm}_{g(t)}|^{2}_{g(t)}
+|u(t)|^{2}+|du(t)|^{2}_{g(t)}
+|\nabla^{2}_{g(t)}u(t)|^{2}_{g(t)}
\leq C'\label{A.16}
\end{equation}
for another constant $C':=C'(n, K_{0}, C_{0}, C_{1})$.
\end{theorem}

\begin{proof} See \cite{LY1, List2005, List2008, Muller2009, Muller2012}.
\end{proof}

\section{Some estimates of the Ricci-harmonic flow}\label{sectionB}

In this section we assume that $(g(t), u(t))$ is a solution to {\color{blue}{(\ref{A.1})}}
on $[0,T]$ where $M$ is a complete $n$-dimensional smooth manifold. Consider a geodesic ball $B_{g(T)}(x_{0},r)$ centered at a fixed point $x_{0}\in M$ with radius $R>0$.

\begin{theorem}\label{tB.1}{\bf (Interior estimates)} Under the above hypotheses, we have

\begin{itemize}

\item[(i)] If the following estimate
\begin{equation}
\sup_{B_{g(T)}(x_{0},R)}|{\rm Ric}_{g(t)}|_{g(t)}
\leq\frac{C}{R^{2}}\label{B.1}
\end{equation}
holds for some positive constant $C$, then, for all $t\in(0,T]$, there exists a constant $C_{n}$, depending only on $n$, such that
\begin{equation}
\sup_{B_{g(t)}(x_{0},R/2)}
|\nabla_{g(t)}u(t)|^{2}_{g(t)}
\leq CC_{n}
\left(\frac{1}{R^{2}}+\frac{1}{t}\right).\label{B.2}
\end{equation}

\item[(ii)] If the following estimate
\begin{equation}
\sup_{B_{g(T)}(x_{0},R)}|{\rm Rm}_{g(t)}|_{g(t)}
\leq\frac{C^{2}}{R^{4}}\label{B.3}
\end{equation}
holds for some positive constant $C$, then, for all $t\in(0,T]$ and all $m\geq0$, there exists a constant $C_{n,m}$, depending only on $n$ and $m$, such that
\begin{equation}
\sup_{B_{g(t)}(x_{0},R/2)}
\left[\left|\nabla^{m}_{g(t)}
{\rm Rm}_{g(t)}\right|^{2}_{g(t)}
+\left|\nabla^{m+2}_{g(t)}u(t)\right|^{2}_{g(t)}\right]
\leq C^{m+2}C_{n,m}
\left(\frac{1}{R^{2}}+\frac{1}{t}\right)^{m+2}.\label{B.4}
\end{equation}

\item[(iii)] If in addition $(g(t), u(t))$ is constructed in {\color{red}{Theorem \ref{tA.3}}}, then
    \begin{equation}
    \inf_{M}u_{0}\leq u(t)\leq\sup_{M}u_{0}
    \end{equation}
    for all $t\in(0,T]$ as long as the constructed solution exists, where $(g_{0},u_{0})$ is the initial data.

\end{itemize}

\end{theorem}

\begin{proof} See \cite{List2005, List2008}.
\end{proof}

\begin{theorem}\label{tB.2} Suppose that $(g(t), u(t))$ is a solution to {\color{blue}{(\ref{A.1})}} on $M\times[0,T]$, where $M$ is a complete $n$-dimensional
smooth manifold and $T\in(0,\infty)$.

\begin{itemize}

\item[(i)] If the following estimate
\begin{equation}
\sup_{M\times[0,T]}|{\rm Ric}_{g(t)}|_{g(t)}
\leq K\label{B.6}
\end{equation}
holds for some positive constant $K$, then
\begin{equation}
|\nabla_{g(t)}u(t)|^{2}_{g(t)}\leq 2K C_{n}\label{B.7}
\end{equation}
on $M\times[0,T]$, for some positive constant $C_{n}$
depending only on $n$.

\item[(ii)] If the following estimate
\begin{equation}
\sup_{M\times[0,T]}|{\rm Rm}_{g(t)}|_{g(t)}
\leq K\label{B.8}
\end{equation}
holds for some positive constant $K$, then
\begin{equation}
\left|\nabla^{m}_{g(t)}{\rm Rm}_{g(t)}
\right|^{2}_{g(t)}+\left|\nabla^{m+2}_{g(t)}u(t)\right|^{2}_{g(t)}\leq
C_{n,m}(4K)^{1+\frac{m}{2}}, \ \ \ m\geq0,\label{B.9}
\end{equation}
on $M\times[0,T]$, for some positive constant $C_{n,m}$ depending only on $n$ and $m$.

\item[(iii)] If $(g(t), u(t))$
is constructed in {\color{red}{Theorem \ref{tA.3}}} with the initial data
$(g_{0}, u_{0})$ satisfying the condition {\color{blue}{(\ref{A.14})}}, then
\begin{equation}
\left|\nabla^{m}_{g(t)}{\rm Rm}_{g(t)}
\right|^{2}_{g(t)}+\left|\nabla^{m+2}_{g(t)}u(t)\right|^{2}_{g(t)}\leq
C'_{n,m}, \ \ \ m\geq0,\label{B.10}
\end{equation}
on $M\times[0,T]$, for some positive constant $C'_{n,m}$ depending only on $n, m, K_{0}, C_{0}, C_{1}$.

\end{itemize}

\end{theorem}

\begin{proof} Given a space-time point $(x_{0}, t)\in M
\times(0,T]$ and consider the geodesic ball $B_{g(t)}(x_{0}, \sqrt{t}/2)$. Since
\begin{equation*}
\sup_{B_{g(T)}(x_{0},\sqrt{t})}
(\sqrt{t})^{2}|{\rm Ric}_{g(t)}|_{g(t)}
\leq (\sqrt{t})^{2} K
\end{equation*}
by {\color{blue}{(\ref{B.1})}} and {\color{blue}{(\ref{B.6})}}, it follows that
\begin{equation*}
\sup_{B_{g(t)}(x_{0},\sqrt{t}/2)}
|\nabla_{g(t)}u(t)|^{2}_{g(t)}\leq (\sqrt{t})^{2}K C_{n}\left(\frac{1}{(\sqrt{t})^{2}}
+\frac{1}{t}\right)=2K C_{n}.
\end{equation*}
In particular, $|\nabla_{g(t)}u(t)|^{2}_{g(t)}(x_{0})\leq 2K C_{n}$.

For (ii), consider the same geodesic ball $B_{g(t)}(x_{0},\sqrt{t}/2)$
and apply {\color{blue}{(\ref{B.4})}}. The part follows from {\color{red}{Theorem \ref{tA.3}}} and
the second one.
\end{proof}

\section{Evolution equations of the Ricci-harmonic flow}\label{sectionC}

Consider the $(\alpha_{1},0,\beta_{1},\beta_{2})$-Ricci flow:
\begin{eqnarray}
\partial_{t}g(t)&=&-2\!\ {\rm Ric}_{g(t)}
+2\alpha_{1}\nabla_{g(t)}u(t)\otimes\nabla_{g(t)}
u(t),\label{C.1}\\
\partial_{t}u(t)&=&\Delta_{g(t)}u(t)+\beta_{1}|\nabla_{g(t)}
u(t)|^{2}_{g(t)}
+\beta_{2}u(t)\label{C.2}
\end{eqnarray}
on a smooth manifold $M$, where $\alpha_{1},\beta_{1},\beta_{2}$ are given constants.

\begin{lemma}\label{lC.1} Under the flow {\color{blue}{(\ref{C.1})}} -- {\color{blue}{(\ref{C.2})}}, we have
\begin{eqnarray}
\Box R_{ij}&=&-2R_{ik}R^{k}{}_{j}
+2R_{pijq}R^{pq}-2\alpha_{1}R_{pijq}\nabla^{p}u\nabla^{q}u\nonumber\\
&&+ \ 2\alpha_{1}\Delta u\nabla_{i}\nabla_{j}u
-2\alpha_{1}\nabla_{i}\nabla_{k}u\nabla^{k}\nabla_{j}u,\label{C.3}\\
\Box R&=&2|{\rm Ric}|^{2}+2\alpha_{1}|\Delta u|^{2}-2\alpha_{1}|\nabla^{2}
u|^{2}- 4\alpha_{1}\langle{\rm Ric},\nabla u\otimes\nabla u\rangle,\label{C.4}\\
\Box R_{ijk\ell}&=&2(B_{ijk\ell}-B_{ij\ell k}+B_{ik j\ell}
-B_{i\ell jk})-(R_{i}{}^{p}R_{pijk\ell}+R_{j}{}^{p}R_{ipk\ell}\nonumber\\
&&+ \ R_{k}{}^{p}
R_{ijp\ell}+R_{\ell}{}^{p}R_{ijkp})+2\alpha_{1}(\nabla_{i}\nabla_{\ell}u
\nabla_{j}\nabla_{k}u-\nabla_{i}\nabla_{k}u\nabla_{j}\nabla_{\ell}u),
\label{C.5}\\
\Box|\nabla u|^{2}&=&2\beta_{2}|\nabla u|^{2}
-2|\nabla^{2}u|^{2}-2\alpha_{1}|\nabla u|^{4}+4\beta_{1}
\langle\nabla u\otimes\nabla u,\nabla^{2}u\rangle,\label{C.6}\\
\Box\nabla_{i}\nabla_{j}u&=&2R_{pijq}\nabla^{p}\nabla^{q}u
+\beta_{2}\nabla_{i}\nabla_{j}u-R_{ip}\nabla^{p}\nabla_{j}u
-R_{jp}\nabla^{p}\nabla_{i}u\nonumber\\
&&- \ 2\alpha_{1}|\nabla u|^{2}\nabla_{i}\nabla_{j}u
+2\beta_{1}\nabla^{k}u\nabla_{k}\nabla_{i}\nabla_{j}u\label{C.7}\\
&&+ \ 2\beta_{1}\nabla_{i}\nabla^{k}u\nabla_{j}\nabla_{k}u
+2\beta_{1}R_{pijq}\nabla^{o}u\nabla^{q}u,\nonumber\\
\Box(\nabla_{i}u\nabla_{j}u)&=&-\nabla^{k}u(R_{ik}\nabla_{j}u
+R_{jk}\nabla_{i}u)-2\nabla_{i}\nabla^{k}u\nabla_{j}\nabla_{k}u
+2\beta_{2}\nabla_{i}u\nabla_{j}u\nonumber\\
&&+ \ 2\beta_{1}\nabla^{k}u(\nabla_{i}u\nabla_{j}\nabla_{k}u
+\nabla_{j}u\nabla_{i}\nabla_{k}u),\label{C.8}
\end{eqnarray}
where $B_{ijk\ell}:=-g^{pr}g^{qs}R_{ipjq}R_{kr\ell s}$.
\end{lemma}

\begin{proof} See \cite{LY1}.
\end{proof}




\begin{thebibliography}{99}

\bibitem{BE1985} Bakry, D.; \'Emery, Michel. \textit{Diffusion
hypercontractives}, S\'eminaire de probabilit\'es, XIX, 1983/84, 177 -- 206, Lecture Notes in Math., {\bf 1123}, Springer, Berlin, 1985. MR0889476 (88j: 60131)


\bibitem{BN2006} Berestovskii, V. N.; Nikonorov, Yu. G. \textit{Killing vector fields of constant length on Riemannian manifolds}, Sibirsk. Mat. Zh., {\bf 49}(2008), no. 3, 497 -- 514; translation in Sib. Math. J., {\bf 49}(2008), no. 3, 395 -- 407. MR2442533 (2009f: 53046)

\bibitem{BN2015} Nikonorov, Yu. G. \textit{Killing vector fields of constant length on compact homogeneous Riemannian manifolds}, Ann. Global Anal. Geom., {\bf 48}(2015), no. 4, 305 -- 330. MR3422911

\bibitem{Bamler-Zhang2015} Bamler, Richard H.; Zhang, Qi S. \textit{Heat kernel and curvature bounds in Ricci flows with bounded scalar curvature}, Adv. Math., {\bf 319}(2017), 496 -- 450. MR3695879

\bibitem{B1918} Bianchi, L. \textit{Lezioni sulla teoria dei gruppi continui finiti di trasfomaioni}, Spoerri, Pisa, 1918.

\bibitem{Bott1967} Bott, Raoul. \textit{Vector fields and characteristic numbers}, Michigan Math. J., {\bf 14}(1967), 231--244. MR0211416 (35 $\#$2297)

\bibitem{Cao2011} Cao, Xiaodong. \textit{Curvature pinching estimate and singularities of the Ricci flow}, Comm. Anal. Geom., {\bf 19}(2011), no. 5, 975--990. MR2886714

\bibitem{Cao-Guo-Tran2015} Cao, Xiaoding; Guo, hongxin; Tran, Hung. \textit{Harnack estimates for conjugate heat kernel on evolving manifolds}, Math. Z., {\bf 281}(2015), no. 1-2, 201--214. MR3384867.

\bibitem{Chen2009} Chen, Bing-Long. \textit{Strong uniqueness of the Ricci flow}, J. Differential Geom., {\bf 82}(2009), no. 2, 363--382. MR25207960 (2009h: 53095)

\bibitem{CCGGIIKLLN3} Chow, Bennett; Chu, Sun-Chin; Glickenstein, David; Guenther, Christine; Isenberg, Jim; Ivey, Tom; Knopf, Dan; Lu, Peng; Luo, Feng; Ni, Lei. \textit{The Ricci flow: Techniques and Application: Part III: Geometric-Analytic Aspects}, Mathematical Surveys and Monographs, {\bf 163}, American Mathematical Society, Providence, RI, 2010. xx+517 pp. ISBN: 978-0-8218-4661-2 MR2604955 (2011g: 53142)

\bibitem{CLN2006} Chow, Bennett; Lu,Peng; Ni, Lei. \textit{Hamilton's
Ricci flow}, Gradient Studies
in Mathematics, {\bf 77}, American Mathematical Society, Providence, RI; Science Press, New York, 2006. xxxvi+608 pp. ISBN: 978-0-8281-4231-7; 0-8218-4231-5 (MT2274812) (2008a: 53068)


\bibitem{CZ2006} Cheng, Bing-Long; Zhu, Xi-Ping. \textit{Uniqueness of the Ricci flow on complete noncompact manifolds}, J. Differential Geom., {\bf 74}(2006), no. 1, 119 -- 154. MR2260930 (2007 i: 53071)

\bibitem{CZ2013} Cheng, Liang; Zhu, Anqiang. \textit{On the extension of the harmonic Ricci flow}, Geom. Dedicata, {\bf 164}(2013), 179--185. MR3054623

\bibitem{D2003} Dafermos, Mihalis. \textit{Stability and
instability of the Cauchy horizon for the spherically symmetric
Einstein-Maxwell-scalar field equations}, Ann. of Math. (2){\bf 158}(2003), no. 3, 875--928. MR2031855 (2005f: 83009)

\bibitem{DW2018} Dunn, Jake; Warnick, Claude. \textit{Stability
of the toroisal AdS Schwarzschild solution in the Einstein-Klein-Gordon system},
arXiv: 1807.04986v1

\bibitem{E1926} Eisenhart, Luther Pfahler. \textit{Riemannian geometry}, Princeton Landmarks in Mathematics, Princeton Paperbacks, Princeton University
Press, Princeton, NJ, 1997, x+306 pp. ISBN: 0-691-02353-0 MR1487892 (98h: 53001)

\bibitem{Enders-Muller-Topping2011} Enders, Joerg; Muller, Reto; Topping, Peter M. \textit{On type-I
singularities in Ricci flow}, Comm. Anal. Geom., {\bf 19}(2011), no. 5, 905--922. MR2886712.

\bibitem{Guo-Huang-Phong2015} Guo, Bin; Huang, Zhijie; Phong, Duong H. \textit{Pseudo-locality for a coupled Ricci flow}, Comm. Anal. Geom., {\bf 26}(2018), no. 3, 585 -- 626.


\bibitem{Guo-Philipowski-Thalmaier2013} Guo, Hongxin; Philipowski, Robert; Thalmairer; Anton. \textit{Entropy and lowest eigenvalue on evolving manifolds}, Pacific J. Math., {\bf 264}(2013), no. 1, 61--81. MR3079761.

\bibitem{Guo-Philipowski-Thalmaier2014} Guo, Hongxin; Philipowski, Robert; Thalmaier; Anton. \textit{A stochastic approach to the harmonic map heat flow on manifolds with time-dependent Riemannian metric}, Stochastic Process. Appl., {\bf 124}(2014), no. 11, 3535--3552. MR3249346.


\bibitem{Guo-Philipowski-Thalmaier2015} Guo, Hongxin; Philipowski, Robert; Thalmaier; Anton. \textit{An entropy formula for the heat equation on manifolds with time-dependent metric, application to ancient solutions}, Potential
    Anal., {\bf 42}(2015), no. 2, 483--497. MR3306693.

\bibitem{H1982} Hamilton, Richard S. \textit{Three-manifolds with positive Ricci curvature},
J. Differential Geom., {\bf 17}(1982), no. 2, 255--306. MR0664497 (84a: 53050)

\bibitem{KW2017} Kennard, Lee; Wylie, William. \textit{Positive
weighted sectional curvature}, Indiana Univ. Math. J., {\bf 66}(2017), no. 2, 419--462. MR3641482.

\bibitem{KWY2017} Kennard, Lee; Wylie, William; Yeroshkin, Dmytro. \textit{The weighted connection and sectional curvature for manifolds with
    density}, arXiv: 1707. 05376

\bibitem{K1999} Klainerman. \textit{PDE as a unified subject}, GAFA 2000 (Tel Aviv, 1999), Geom. Funct. Anal. 2000, Special Volume, Part I, 279--315. 35--02 MR1826256 (2002e:35001)

\bibitem{KRS2015} Klainerman, Sergiu; Rodnianski, Igor; Szeftel, Jeremie. \textit{The bounded $L^{2}$ curvsture conjecture}, Invent. Math., {\bf 202}(2015), no. 1, 91--216. MR3402797

\bibitem{KRS2012} Klainerman, Sergiu; Rodnianski, Igor; Szeftel, Jeremie. \textit{Overview of the proof of the bounded $L^{2}$ curvature conjecture}, arXiv: 1204.1772.

\bibitem{K2010} Kotschwar, Brett L. \textit{Backwards uniqueness for the Ricci flow}, Int. Math. Res. Not. IMRN 2010, no. 21, 4064--40977. MR2738351 (2012c: 53100)

\bibitem{K2014} Kotschwar, Brett L. \textit{An energy approach to the problem of uniqueness for the Ricci flow}, Comm. Anal. Geom., {\bf 22}(2014), no. 1, 149--176. MR3194377

\bibitem{K2016} Kotschwar, Brett L. \textit{An energy approach to uniqueness for higher-order geometric flows}, J. Geom. Anal., {\bf 26}(2016), no. 4, 3344--3368. MR3544962.

\bibitem{K2016} Kotschwar, Brett L. \textit{A short proof of backward uniqueness for some geometric evolution equations}, Internat. J. Math., {\bf 27}(2016), no. 12, 1650102, 17 pp. MR3575926

\bibitem{KMW2016} Kotschwar, Brett; Munteanu, Ovidiu; Wang, Jiaping.
\textit{A local curvature estimate for the Ricci flow}, J. Funct. Anal., {\bf 271}(2016), no. 9, 2604--2630. MR3545226

\bibitem{LM2016} LeFloch, Philippe, G.; Ma, Yue. \textit{The global nonlinear stability of Minkowski space for self-gravitating
massive fieldse}, Comm. Math. Phys., {\bf 346}(2016), no. 2, 603--665. MR353896

\bibitem{LM2017} LeFloch, Philippe, G.; Ma, Yue. \textit{The global nonlinear stability of Minkowski space. Einstein equations, $f(R)$-modified
gravity, and Kelin-Gordon fields}, arXiv: 1712.10045

\bibitem{LL2014} Li, Songzi; Li Xiang-Dong. \textit{Harnack inequalities and $W$-entropy formula for Witten Laplacian on Riemannian manifolds with $K$-super Perelman Ricci flow}, arXiv: 1412.7034v2.

\bibitem{LL2015} Li, Songzi; Li Xiang-Dong. \textit{The $W$-entropy formula for the Witten Laplacian on manifolds with time dependent metrics and potentials}, Pacific J. Math., {\bf 278}(2015), no. 1, 173--199. MR 3404671

\bibitem{LL20171} Li, Songzi; Li Xiang-Dong. \textit{On Harnack inequalities for Witten Laplacian on Riemannian manifolds with super Ricic flow},
    Asian J. Math., {\bf 22}(2018), no. 4, 577 -- 598.

\bibitem{LL20172} Li, Songzi; Li Xiang-Dong. \textit{$W$-entropy, super Perelman Ricci flows and $(K,m)$-Ricci solitons}, arXiv: 1706.07040v1.

\bibitem{LL20173} Li, Songzi; Li Xiang-Dong. \textit{Hamilton differential Harnack inequality and $W$-entropy for Witten Laplacian on Riemannian manifolds}, J. Funct. Anal., {\bf 274}(2018), no. 11, 3263 -- 3290. MR3782994

\bibitem{LL20174} Li, Songzi; Li Xiang-Dong. \textit{$W$-entropy formulas on super Ricci flows and Langevin deformation on Wasserstein space over Riemannian manifolds}, Sci. China Math., {\bf 61}(2018), no. 8, 1385 -- 1406. MR3833742

\bibitem{LY0} Li, Yi. \textit{Generalized Ricci flow I: Higher derivartives estimates for compact manifolds}, Anal. PDE, {\bf 5}(2012), no. 4, 747--775. MR3006641

\bibitem{LY1} Li, Yi. \textit{Generalized Ricci flow II: Existence for complete noncompact maniflds}, arXiv:1309.7710.

\bibitem{LY2} Li, Yi. \textit{Long time existence and bounded scalar
curvature in the Ricci-harmonic flow}, J. Differential Equations, {\bf 265}(2018), no. 1, 69 -- 97. MR3782539

\bibitem{LY3} Li, Yi. \textit{Long time existence of Ricci-harmonic flow}, Front. Math. China, {\bf 11}(2016), no. 5, 1313--1334. MR3547931

\bibitem{LL2015} Li, Yi; Liu, Kefeng. \textit{A geometric heat flow for
vector fields}, Sci. China Math., {\bf 58}(2015), no. 4, 673--688. MR3319305

\bibitem{List2005} List, Bernhard. \textit{Evolution of an extended Ricci flow system}, PhD thesis, AEI Potsdam, 2005.


\bibitem{List2008} List, Bernhard. \textit{Evolution of an extended Ricci flow system}, Comm.
Anal. Geom., {\bf 16}(2008), no. 5, 1007--1048. MR2471366 (2010i: 53126)



\bibitem{MY1973} Montgomery, Deane; Yang, C. T. \textit{On homotopy seven-spheres that admit differentiable pseud-free circle actions}, Michigan Math. J., {\bf 20}(1973), 193--216. MR0319219 (47 $\#$7764)



\bibitem{Muller2009} M\"uller, Reto. \textit{The Ricci flow coupled with harmonic map flow}, PhD thesis, ETH Z\"urich, doi: 10.3929/ethz-a-005842361, 1009.

\bibitem{Muller2010} M\"uller, Reto. \textit{Monotone volume formulas for geoemtric flow}, J. Reine
Angew. Math., {\bf 643}(2010), 39--57. MR2658189 (2011k: 53086)

\bibitem{Muller2012} M\"uller, Reto. \textit{Ricci flow coupled with harmonic map flow}, Ann.
Sci. \'Ec. Norm. Sup\'er, (4) {\bf 45}(2012), no. 1, 101--142. MR2961788

\bibitem{P1} Perelman, Grisha. \textit{The entropy formula for the Ricci flow and its geoemtric
applications}, arXiv: math/0211159.

\bibitem{R2009} Ringstr\"om, Hans. \textit{The Cauchy problem in general relativity}, {\bf ESI} Lectures in Mathematics and Physics, European Mathematical Society (EMS), Z\"urich, 2009. xiv+294 pp. ISBN: 978-3-03719-053-1 (MR2527641) (2010j: 83001)

\bibitem{Sesum2005} Sesum, Natasa. \textit{Curvature tensor under the Ricci flow}, Amer.
J. Math., {\bf 127}(2005), no. 6, 1315--1324. MR2183526 (2006f:53097)

\bibitem{Simon2015} Simon, Miles. \textit{$4D$ Ricci flows with bounded scalar curvature},
arXiv: 1504.02623v1.

\bibitem{S1} Szeftel, Jeremie. \textit{Parametrix for wave equations on a rough background I: regularity of the phase at initial time}, arXiv: 1204.1768

\bibitem{S2} Szeftel, Jeremie. \textit{Parametrix for wave equations on a rough background II: construction of the parametrix and control at initial time}, arXiv: 1204.1769

\bibitem{S3} Szeftel, Jeremie. \textit{Parametrix for wave equations on a rough background III: space-time regularity of the phase}, arXiv: 1204.1770

\bibitem{S4} Szeftel, Jeremie. \textit{Parametrix for wave equations on a rough background IV: control of the error term}, arXiv: 1204.1771

\bibitem{S5} Szeftel, Jeremie. \textit{Sharp Strichartz estimates for the wave equation on a rough background}, Ann. Sci. \'Ec. Norm. Sup\'er, (4){\bf 49}(2016), no. 6, 1279 -- 1309. MR3592358

\bibitem{Tuschmann1997} Tuschmann, Wilderich. \textit{On the structure of compact simply-connected manifolds of positive sectional curvature}, Geom. dedic., {\bf 67}(1997), no. 1, 107--116. MR1468863 (98i: 53057)

\bibitem{V2018} Van de Moortel, Maxime. \textit{Stability
and instability of the sub-extremal Reissner-Mordstr\"om black hole interior for the Einstein-Maxwell-Klein-Gordon equations in spherical symmetry}, Comm. Math. Phys., {\bf 360}(2018), no. 1, 103--168. MR3795189

\bibitem{Wang2016} Wang, Qian. \textit{An intrinsic hyperboloid approach for Einstein Klein-Gordon equations}, arXiv: 1607.01466v1.

\bibitem{Wang} Wang, Qian. \textit{Global existence for the Einstein equations
with massive scalar fields}, in preparation.

\bibitem{WJH2018} Wang, Jinghua. \textit{Future stability
of the $1+3$ Milne model for Einstein-Klein-Gordon system}, arXiv: 1805.01106

\bibitem{Wu2018} Wu, Guoqiang. \textit{Scalar curvature bound and
compactness results for Ricci harmonic solitons}, Proc. Amer. Math. Soc., {\bf 146} (2018), no. 8, 3473 -- 3483. MR3803672


\bibitem{Wylie2015} Wylie, William. \textit{Sectional
curvature for Riemann manifolds with density}, Geom. Dedicata, {\bf 178}(2015), 151--169. MR3397488.

\bibitem{WY2016} Wylie, William; Yeroshkin, Dmytro. \textit{On the geometry of Riemannian manifolds with density}, arXiv: 1602.08000v1.


\bibitem{Yano1} Yano, Kentaro. \textit{On harmonic and Killing vector fields}, Ann. of Math., (2){\bf 55}(1952), 38--45. MR0046122 (13, 689a)

\bibitem{Yano2} Yano, Kentaro. \textit{Integral formulas in Riemannian geometry}, Pure and Applied Mathematics, No. {\bf 1},
    Marcel Dekker, Inc., New York, 1970. ix+156pp. MR0284850 (44$\#$ 2174)

\bibitem{YB1953} Yano, K.; Bochner, S. \textit{Curvature and Betti numbers},
Annals of Mathematics Studies, No. {\bf 32},
Princeton University Press, Princeton, N. J., 1953. ix+190pp. MR0062505 (15, 989f)

\bibitem{Zhang2010} Zhang, Zhou. \textit{Scalar curvature behavior for finite-time singularity of
K\"ahler-Ricci flow}, Michigan Math., {\bf 59}(2010), no. 2, 419--433. MR2677630 (2011j: 53128)

\end{thebibliography}
\end{document}